\documentclass[12pt,leqno]{amsart}
\usepackage[all]{xy}
\usepackage{amssymb,mathrsfs,euscript,enumitem,marginnote}
\usepackage[usenames,dvipsnames]{pstricks}
\usepackage{amsmath,bbold}
\usepackage{fourier}
\usepackage[hypertex]{hyperref}
\usepackage[backgroundcolor=yellow,colorinlistoftodos]{todonotes}
\textheight9in \def\DATE{December 29, 2022}
\hoffset - 2cm  \hbadness=100000
\newtheorem{theorem}{Theorem}
\newtheorem{corollary}[theorem]{Corollary}
\newtheorem*{theoremA}{Theorem~A}
\newtheorem*{theoremB}{Theorem~B}
\newtheorem*{theoremC}{Theorem~C}

\newtheorem{lemma}[theorem]{Lemma}
\newtheorem{proposition}[theorem]{Proposition}

\theoremstyle{definition}
\newtheorem{example}[theorem]{Example}
\newtheorem{exercise}[theorem]{Exercise}
\newtheorem{remark}[theorem]{Remark}

\newtheorem{definition}[theorem]{Definition}

\def\uB{\overline{\Blob}}
\def\doubless#1#2{{
\def\arraystretch{.5}
\begin{array}{c}
\mbox{\scriptsize $\scriptstyle #1$}
\\
\mbox{\scriptsize $\scriptstyle #2$}
\end{array}\def\arraystretch{1}
}}
\def\oMn{\overline{\oM}}
\def\fopn{\overline{\fop}}
\def\Cha{{\tt Chaos}}
\def\ublob{\overline{\blob}}
\def\dublob{\ublob(\C)}
\def\leaderfill{\leaders\hbox to 1em{\hss.\hss}\hfill}
\def\mod{{\tt Mod}}
\def\MOD{{\tt MOD}}
\def\pinv#1{{p^{-1}(#1)}}
\def\LL{{\EuScript L}}
\def\Krtek#1#2{\fopn\left(  \rule{-.8em}{2.1em} \right.
\raisebox{1.5em}{
\xymatrix@R=1em{#1
\\
#2 \ar@{^{(}->}[u] \rule{0em}{1em} 
}
}
\left. \rule{-.6em}{2.1em}  \right) 
}
\def\Krtekmodularni#1#2{\oMn\left(  \rule{-.8em}{2.1em} \right.
\raisebox{1.5em}{
\xymatrix@R=1em{#1
\\
#2 \ar@{^{(}->}[u] \rule{0em}{1em} 
}
}
\left. \rule{-.6em}{2.1em}  \right) 
}
\def\sfL{{\mathcal L}}
\def\bod{{\raisebox{.15em}{\,$\bigodot$\,}}}
\def\ttQ{{\tt Q}}
\def\CA{{\ttO}_{A}}
\def\ttP{{\int_\ttO \oS}}
\def\oS{{\EuScript S}}
\def\oB{{\overline B}}

\def\lunit{{\EuScript I}_c}
\def\ofop{{\widehat{\fop}_c}}
\def\calB{{\EuScript B}}
\def\Span{{\rm Span}}
\def\colorop #1(#2;#3){{#1}
   \left(\rule{0pt}{15pt}\right.
         \hskip -3mm \begin{array}{c}
	              #3\\#2
                     \end{array}
         \hskip -3mm \left. 
   \rule{0pt}{15pt} \right)
}
\def\fop{{\EuScript {F}}}
\def\fopc{{\EuScript {F\!}_c}}
\def\loc{{\EuScript U}}
\def\C{{\EuScript C}}
\def\uttm{\overline{\tt m}}
\def\Mbar{{\overline\ttM}}
\def\+{\!+\!}
\def\dBlob{{\tt Blob}(\C)}
\def\Blob{{\tt Blob}}

\def\bM{{\tt M}}  
\def\ttF{{\tt F}}    
\def\ttD{{\tt D}}      
\def\ttm{{\tt m}}
\def\blob{{\tt blob}}
\def\dblob{{\tt blob}(\C)}
\def\udblob{{\ublob(\C)}}
\def\Set{{\tt Set}}
\def\Int{\mathring}
\def\Beta#1{{\beta_{#1}(\oP,\oM)}}
\def\BETA#1(#2,#3){{\beta_{#1}\left(#2,#3\right)}}
\def\hMod{{{\tt Mod}_\ttM(\oP)}}

\def\bfOmega{{\boldsymbol \Omega}}
\def\Coll{{\tt Coll_{\ttM}}}
\def\oE{{\EuScript E}}
\def\Free{{\mathbb F}}
\def\hFree{{\mathring{\mathbb F}}}
\def\WBU{{\tt WBU}}
\def\pa{{\partial}}
\def\Vect{\hbox{{$R$-{\tt Mod}}}}
\def\qi{{quasi-isomorphism}}

\def\ttV{{\tt V}}
\def\ttB{{\tt B}}
\def\ttL{{\tt L}}
       
\def\crada#1#2{#1\cdots#2}       
\def\Rada#1#2#3{#1_{#2},\dots,#1_{#3}}
\def\unit{{\mathbf 1}}
\def\oM{{\EuScript M}}
\def\oMc{{\EuScript M}_c}
\def\oA{{\EuScript A}}
\def\ot{\otimes}
\def\ttM{{\tt M}}
\def\oP{{\EuScript P}}
\def\oM{{\EuScript M}}
\def\Tau{\hbox{${\mathcal T}$\hskip -.3em}}
\def\ttCat{{\tt Cat}}
\def\ttA{{\tt A}}
\def\sfD{{\mathcal D}}

\def\ttC{{\tt C}}
\def\id{\hbox{\rm 1\hskip -1.5mm 1}}
\def\id{{\mathbb 1}}
\def\fib{\triangleright}
\def\vfib{\rotatebox{90}{$\triangleright$}}
\def\Afib{\raisebox{.5em}{\rotatebox{-90}{$\triangleright$}}}
\def\ttO{{\tt O}}
\def\Fib{{\mathscr F}}
\def\fM{{\mathbb M}}
\def\fD{{\mathscr D}}
\def\Gib{{\mathscr G}}
\def\Twr{{\boldsymbol {\EuScript T}\!}}

\catcode`\@=11

\def\@evenfoot{\rule{0pt}{20pt}[\DATE] \hfill [{\tt \jobname.tex}]}
\def\@oddfoot{\rule{0pt}{20pt}{[\tt \jobname.tex}]\hfill [\DATE]}
\@namedef{subjclassname@2010}{%
  \textup{2020} Mathematics Subject Classification}
\catcode`\@=13


\makeatletter
\providecommand\@dotsep{5}
\def\listtodoname{List of Todos}
\def\listoftodos{\@starttoc{tdo}\listtodoname}
\makeatother

\title[Operads and blobs]
{Operads, operadic categories and the blob complex}

\thanks{Supported by RVO: 67985840 and Praemium Academi\ae.}

\begin{document}
\bibliographystyle{plain}

\parskip3pt plus 1pt minus .5pt
\baselineskip 18pt  plus 1.5pt minus .5pt

\author[M.\ Batanin, M.\ Markl]{Michael Batanin \& Martin Markl}
\address{The Czech Academy of Sciences, Institute of Mathematics, {\v Z}itn{\'a} 25,
         115 67 Prague 1, The Czech Republic}

\begin{abstract}
We will show that the Morrison\textendash{}Walker blob complex appearing
in Topological Quantum
Field Theory is an operadic bar resolution of a certain operad
composed of fields and local relations. 
As a by-product we develop the
theory of unary operadic categories and study some 
novel and interesting phenomena arising in this context.
\end{abstract}

\maketitle

\tableofcontents

\section*{Introduction}

The blob complex was introduced by S.~Morrison and K.~Walker
in~\cite{blob}. It associates  to an
\hbox{$n$-dimensional} manifold $\fM$, equipped with a system of
fields $\C$ containing an ideal of local relations $\loc$, the {\em blob complex 
 $\calB_*(\fM,\C)$}, which is a chain complex whose salient feature
is the isomorphism
\[
H_0\big(\calB_*(\fM,\C)\big) \cong A(\fM) := \C(\fM)/\loc(\fM),
 \]
where $A(\fM)$ is the skein module associated to $\fM$. If the
fields come from an $n$-category $\ttC$ with strong duality,
$A(\fM)$ is the usual topological quantum
field invariant of $\fM$ associated to $\ttC$. The name originated from a
{\em blob\,}, defined as the standard $n$-dimensional ball embedded in $\fM$.

The initial impulse for the present work was a seminar given by K.~Walker
at MSRI, Berkeley, in the winter of 2020. The second author
noticed a striking similarity between the diagrams drawn by Kevin on
the board, and pictures representing elements of 
free operads over graph-related
operadic categories that can be found in~\cite[Section~5]{kodu}. 
This inspired the idea
that the blob complex might be the bar resolution of an operad 
over a suitable operadic category.

That hope indeed turned out to be true; there even exist two related
but non-equivalent ways to interpret the blob complex within operad
theory.  The {\em first\/} interpretation produces a complex quasi-isomorphic
to the Morrison\textendash{}Walker blob complex -- 
the bar construction of a
certain operad of fields over an operadic
category of blob configurations in $\fM$. 
The {\em second\/} interpretation identifies the blob complex with Fresse's bar
construction of a traditional coloured operad, reminiscent of the
little discs operad; the colours
are blobs in $\fM$ with boundaries decorated by fields. Both approaches thus lead 
to several versions of `blob complexes.'
Their relations are summarized in Figure~\ref{Pozitri prof. Vondra.}
at the end of this~paper. 

\vskip .5em
\noindent 
{\bf Disclaimer.} 
The present work does not bring anything new to the theory of blob
complexes per se, neither it adds anything to the
explicit calculations given in~\cite{blob}. 
The free, acyclic resolution of the skein module in 
Theorem~B on page~\pageref{Veta B} might 
however pave way for the study of derived TQFT invariants.  

\vskip .5em
\noindent 
{\bf Novelties.}
The operadic category  of blobs is unary,
meaning that the cardinalities of all its objects are one.
The blob complex thus represents  an interesting,
highly nontrivial example of an unary operadic category 
and justifies careful
analysis of operads,  operadic modules, fibrations and
various versions of the bar
construction in this context. 
Several interesting and new phenomena were discovered 
en~route.    

Conceptual understanding of 
the relationship between  the decorated version of the 
unital operadic category of blobs and the
un-decorated one in Section~\ref{Boli mne prava noha.} inspired the
notion of partial discrete operadic fibrations, 
partial operads and the associated partial Grothendieck
construction, 
given in~Section~\ref{Podari se mi premluvit Jarku abych mohl
  jet do Prahy uz dnes?}. The concrete partial operad that arose in this
context is unital in a weak, unexpected sense,
formalized in Definition~\ref{Uz se tech
  prohlidek bojim.}  by introducing pseudo-units. Pseudo-unitality   
is a new, nontrivial concept even in the realm of traditional algebra, as
Example~\ref{Mam prodlouzene ARC pro Tereje.}~shows. 
We also introduce several versions of
the `standard' 
unitality condition for operads over operadic categories that are not equipped
with the chosen local terminal objects required
in~\cite[Section~1]{duodel}. 

While free modules over classical operads have simple structure,
cf.~e.g.~\cite[Sub\-section~2.10.1]{Fresse}, this is not true in 
the world of operads and their modules over general operadic
categories, where the structure depends on the shape
of the operadic category, as illustrated in Example~\ref{Je to k zblazneni to
  blyskani.}. A structure result can however still 
be obtained
under the condition of rigidity introduced in Definition~\ref{Opet
  budu nytovat podvozek.}, which has no analog in the standard operad
theory.  We believe that all 
the above notions admit generalizations to non-unary operadic categories.

The present paper has two parts. Part~1 develops general theory of unary
operadic categories, operads, modules and resolutions, Part~2 
is devoted to applications to the blob complex.
The main results of the article 
are Theorem~A on page~\pageref{Veta A}, Theorem~B on
page~\pageref{Veta B} and Theorem~C on page~\pageref{Veta C}.
Propositions~\ref{Mam v pokoji 18 stupnu a je
  pulka listopadu.} and~\ref{Jak to vsechno dopadne?} have no counterparts
in traditional algebra.

\vskip .5em
\noindent 
{\bf Requirements and conventions.} 
We will assume working knowledge of operads;
suitable references are the
monograph~\cite{markl-shnider-stasheff:book} and the 
overview~\cite{markl:handbook}. Operadic categories and related
notions were introduced in~\cite{duodel}, but all necessary material
from that paper is recalled in Sections~\ref{Porad se mi bliska.}
and~\ref{Cuka mi v oku.}.  Some preliminary 
knowledge of~\cite{blob} may ease reading Part~2.

Categories will be denoted by typewriter letters such as
$\ttC, \ttO, \ttQ$, \&c,  operads and their modules 
written in script, e.g.~$\oP$, $\oS$, $\oM$, \&c. 
From Section~\ref{Zitra pujdu s Jarkou nakupovat bundicku.} on, all
algebraic objects will live in the monoidal category
$\Vect$ of graded 
modules over a unital commutative associative ring~$R$. 
Chain complexes will be non-negatively graded, with differentials of degree
$-1$. By a \qi\ we mean a morphism of chain
complexes that induces an isomorphism of homology.
Preprints~\cite{env,kodu} are under permanent revision, so we
indicated explicitly which concrete versions we referred to. 

\noindent {\bf Acknowledgment.}  
The first author wishes to express his sincere gratitude to Joachim Kock for
many illuminating discussions on d\'ecollage comonad and operadic
categories from simplicial point of view. The second author is indebted
to Benoit Fresse for pointing to the results of his impressive
book~\cite{Fresse} that were relevant to our work.

\part{Unary operadic categories}

\section{Operadic categories and operads}
\label{Porad se mi bliska.}

Our immediate aim is to rephrase the definitions of
operadic categories and their operads as given
in~\cite[Section~1]{duodel} to the particular, unary case when the
cardinality functor is constant and equals $1$. We believe that this
would make this article independent of~\cite{duodel}.  

Unary operadic categories will appear in Definition~\ref{Svedeni snad prestava.} as categories equipped with fiber functors;
Propositions~\ref{Musim se objednat.} and~\ref{Musim se objednat asi
  jeste dnes.} then describe them also as algebras for a certain
monad. Operads over unary operadic categories are introduced in
Definition~\ref{Je to tady poustevnictvi.}. In this section we
however, unlike in~\cite{duodel},
do not require the existence of chosen local terminal objects in operadic
categories, neither we assume units of operads. A refined analysis of
these additional structures is given in~Section~\ref{Budu letat
  jeste pristi rok?}.   

\begin{remark} 
The approach to unary operadic categories
presented in this section has a lot of overlaps with the
material from~\cite{GarnerKockWeber} concerning the use of the
d\'ecollage comonad $\sfD$ recalled in~\eqref{Jesu salvator saeculi}. But in contrast with
\cite{GarnerKockWeber}, where $\sfD$-coalgebra structure is fixed
and the fiber-functor structure is imposed on top of it, we begin
with imposing a fiber-structure functor and adding some pieces of
\hbox{$\sfD$-coalgebra} structure only when it is necessary. This allows us to
analyze subtler unitality properties of operadic
categories.
\end{remark}

\begin{lemma}
\label{Mozna mam jen otlacenou ruku od operadla.}
Each family $\big\{\Fib_S : \ttO/S \to \ttO \ | \ S \in \ttO\big\}$ of
functors indexed by objects of\/ $\ttO$ canonically induces a~family 
$\big\{\Fib_c : \ttO/S \to \ttO/\Fib_T(c) \ | \  c :S \to T\big\}$ of functors  indexed
by arrows of\/ $\ttO$.
\end{lemma}

\begin{proof}
Assume that $c: S \to T$ is a morphism of $\ttO$. The functor $\Fib_c$
acts, by definition, on an object $f: X \to S \in \ttO/S$ as
follows. Embed $f$ in the diagram
\[
\xymatrix{&X \ar[ld]_h \ar[d]^f
\\
T&S\ar[l]_c
}
\]
in which $h := cf$.
Interpreting the arrows $h$ and $c$ as objects of $\ttO/T$, $f$ appears as
a morphism $f:h \to c$ in $\ttO/T$. We then define 
\[
\Fib_c(f) := \Fib_T(f) : \Fib_T(h) \longrightarrow \Fib_T(c) \in \ttO/\Fib_T(c).
\]
To describe the action of $\Fib_c$ on a morphism ${{\Phi}}:
b \to a$ in $\ttO/S$ given by the commutative diagram
\[
\xymatrix@C=1em{Y\ar[rr]^\phi\ar[dr]_b&&X \ar[dl]^a
\\
&\, S,&
}
\]
embed that diagram into
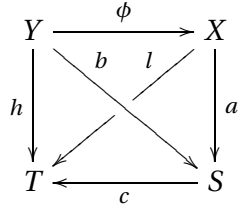
\begin{equation}
\label{Lakuji si blokady.}
\xymatrix@C=5em@R=4em{Y\ar[d]_h  \ar[r]^\phi \ar[rd]^(.3)b&X \ar[d]^a
  \ar[dl]_(.3)l|!{[d];[l]}\hole
\\
T& \ar[l]^cS
}
\end{equation}
in which $h := bc$ and $l:= ca$. Then $\Fib_c({\Phi}) :
\Fib_c(b) \to \Fib_c(a)$ is given by the commutative diagram
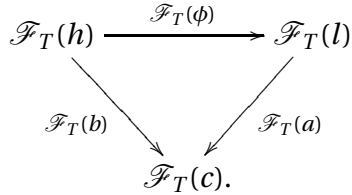
\begin{equation}
\label{Pojedu na Zimni skolu?}
\xymatrix@R=3.5em@C=.8em{\Fib_T(h)\ar[rr]^{\Fib_T(\phi)}\ar[dr]_{\Fib_T(b)}&&
\Fib_T(l) \ar[dl]^{\Fib_T(a)}
\\
&\ \Fib_T(c).&
}
\end{equation}
The verification that the above rules define a functor is straightforward.
\end{proof}

\begin{definition}
\label{Zacina mne zase svedit i leva ruka?}
A family $\Fib_S$ of functors as in Lemma~\ref{Mozna mam jen otlacenou ruku od
  operadla.} is called a family of {\em fiber functors\/} if, for each
arrow $c:S \to T$ in $\ttO$, the diagram of functors
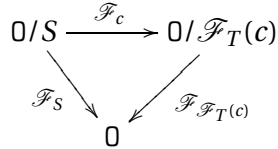
\begin{equation}
\label{Prodlouzi mi papiry?}
\xymatrix@C=1em{\ttO/S \ar[rr]^(.4){\Fib_c} \ar[rd]_{\Fib_S} &&
  \ttO/\Fib_T(c)
\ar[ld]^{\Fib_{\Fib_T(c)}}
\\
&\ttO&
}
\end{equation}
commutes.
\end{definition}

To expand Definition~\ref{Zacina mne zase svedit i leva ruka?},
introduce the following notation and terminology. Given a map \hbox{$f:X \to
S$}, we call $\Fib_S(f)$ the {\em fiber\/} of $f$ and denote it simply
by $\Fib(f)$. The fact that $F \in \ttO$ is the fiber of $f$ 
will also be expressed by
writing $F \fib X \stackrel f\to S$. For a morphism $F$ in $\ttO/S$
given by the diagram
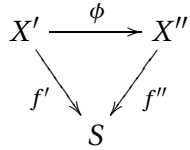
\[
\xymatrix@C=1em{X'\ar[dr]_{f'}\ar[rr]^{\phi}& &X'' \ar[dl]^{f''}
\\
&S&
}
\]
denote by $\phi_S$ the induced map $\Fib_S(F) : \Fib(f') \to
\Fib(f'')$ between fibers.
Then the commutativity of~(\ref{Prodlouzi mi papiry?}) on objects is
expressed by the equality of the fibers in the diagram
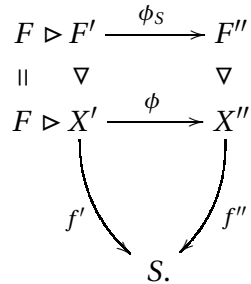
\begin{equation}
\label{Po navratu ze Zimni skoly.}
\xymatrix@R=0.1em@C=.8em{\hskip -1.8em F\ \fib F'
  \ar[rr]^{\phi_S}   
&&F'' \hskip -.5em
\\
\hskip -2em\raisebox{.7em}{\rotatebox{270}{$=$}}\hskip 1.65em\raisebox{.7em}{{\rotatebox{270}{$\fib$}}} && 
\hskip -.5em\raisebox{.7em}{{\rotatebox{270}{$\fib$}}} \hskip -.5em
\\
\hskip -1.8em F\ \fib X' \ar[rr]^\phi  \ar@/_.9em/[ddddddr]_{f'}   &&
X'' \ar@/^.9em/[ddddddl]^{f''} \hskip -.5em
\\&& 
\\&& 
\\&& 
\\&&
\\ &&
\\
&\ S .&
}
\end{equation}
In words, the fiber of a map equals the fiber of the induced map
between its fibers.

To expand the commutativity of~(\ref{Prodlouzi mi papiry?}) on
morphisms, notice that in the above notation, diagram~(\ref{Pojedu na
  Zimni skolu?}) reads
\[
\xymatrix@R=2.5em@C=.8em{
\Fib(h)\ar[rr]^{\phi_T}\ar[dr]_{b_T}&&
\Fib(l) \ar[dl]^{a_T}
\\
&\ \Fib(c).&
}
\]
In the situation of~(\ref{Lakuji si blokady.}),
diagram~(\ref{Prodlouzi mi papiry?}) requires that
\[
(\phi_T)_{\Fib(c)} = \phi_S.
\]

\begin{definition}
\label{Svedeni snad prestava.}
A strict {\em unary (nonunital) operadic category\/} is a category
$\ttO$ equipped with a family of fiber functors as per
Definition~\ref{Zacina mne zase svedit i leva ruka?}.  A strict
{\em operadic functor\/} $\Phi : \ttO' \to \ttO''$ between strict unary
operadic categories is a functor that commutes with the associated
fiber functors.
\end{definition}

Since all operadic categories and operadic functors in this article
will be {\em strict}, we will for brevity omit this adjective. If not
indicated otherwise, by an operadic category we will always mean
unary and nonunital one.

\begin{remark}
\label{Dnes mi vypalovali bradavici.}
Recall that a simplicial set is a~collection $S_\bullet = \{S_n\}_{n
\geq 0}$ of sets together with maps
\begin{align*}
d_i &: S_n \to S_{n-1},\ 0 \leq i \leq n,\ n \geq 1, \ \mbox { and}
\\
s_i&: S_n \to
S_{n+1},\ 0 \leq i \leq n,\ n \geq 0, 
\end{align*}
that satisfy the identities:
\begin{align*}
\nonumber 
d_id_j &= d_{j-1}d_i \mbox { if $i < j$},
\\
\nonumber 
s_is_j &= s_{j+1}s_i \mbox { if $i \leq j$},
\\
\label{1.1.2012}
d_i s_j &= s_{j-1}d_i \mbox { if $i < j$},
\\
\nonumber 
d_is_j &= \mbox { identity } =  d_{j+1}s_i \mbox {, \ and}
\\
\nonumber 
d_i s_j &= s_{j}d_{i-1} \mbox { if $i >  j+1$.}
\end{align*}
An important example is the {\em simplicial nerve\/} $N_\bullet(\ttO)
= \{\ttO_n\}_{n \geq 0}$ 
of the category $\ttO$, which is a simplicial complex of the form 
\[
\xymatrix@C=7em{\ttO_0 \ar[r]|-{s_0}     
&  \ar@/^1em/[l]|-{d_0}  \ar@/_1em/[l]|-{d_1}  \ttO_1
\ar@/_1em/[r]|-{s_0}
\ar@/^1em/[r]|-{s_1}
& \ttO_2   \ar@/^2em/[l]|-{d_0}  \ar[l]|-{d_1}  \ar@/_2em/[l]|-{d_2} 
&\hskip -10em\cdots \ \ .
}
\]
In the above display, $\ttO_n$ consists of chains 
\begin{equation}
\label{Vcera jsme s Jarkou vypousteli vodu na chalupe.}
T_0 \stackrel
{f_0}\longrightarrow T_1 \stackrel {f_1}\longrightarrow \cdots T_{i-1} \stackrel
{f_{i-1}}\longrightarrow T_i \stackrel {f_i}\longrightarrow T_{i+1}
\stackrel {f_{n-1}}\longrightarrow T_n
\end{equation} 
of arrows of $\tt0$. The operator $d_0$ acts on~(\ref{Vcera jsme s
  Jarkou vypousteli vodu na chalupe.}) by removing $T_0$, $d_n$
removes $T_n$. The operator $d_i$ with $0 < i < n$ replaces the part
$T_{i-1} \stackrel
{f_{i-1}}\longrightarrow T_i \stackrel {f_i}\longrightarrow 
T_{i+1}$ of~(\ref{Vcera
  jsme s Jarkou vypousteli vodu na chalupe.}) by the single arrow $f_i
f_{i+1} : T_{i-1} \to T_{i+1}$. The operator $s_i$ replaces $T_i$
in~(\ref{Vcera jsme s Jarkou vypousteli vodu na chalupe.}) by the
identity automorphism $\id :T_i \to T_i$.

The extra structure on $\ttO$ given by Lemma~\ref{Mozna mam jen
otlacenou ruku od operadla.} is equivalent to adding to 
the nerve of $\ttO$ a sequence of additional face operators 
\[
d_{n+1}:\ttO_n\to \ttO_{n-1}, \ n \geq 1,
\]
as in
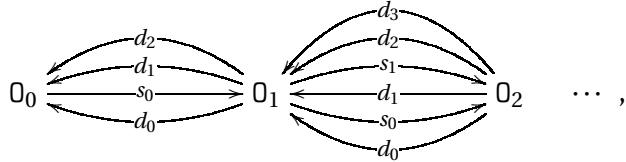
\[
\xymatrix@C=7em{\ttO_0 \ar[r]|-{s_0}     
&  \ar@/^1em/[l]|-{d_0}  \ar@/_1em/[l]|-{d_1}
\ar@/_2em/[l]|-{d_2}  \ttO_1
\ar@/_1em/[r]|-{s_0}
\ar@/^1em/[r]|-{s_1}
& \ttO_2   \ar@/^2em/[l]|-{d_0}  \ar[l]|-{d_1}  \ar@/_2em/[l]|-{d_2} 
\ar@/_3em/[l]|-{d_3} 
&\hskip -10em\cdots \ \ ,
}
\]
which satisfies all standard simplicial identities except for the top
one, that is, the composite $d_nd_{n+1}:\ttO_n\to \ttO_{n-2}$ is not
necessary equal to $d_nd_{n}:\ttO_n\to \ttO_{n-2}.$ The condition of
Definition \ref{Zacina mne zase svedit i leva ruka?}  restores this
relation.
\end{remark}

Let us denote by $\sfD(\ttA)$ the coproduct
\begin{equation}
\label{Jesu salvator saeculi}
\sfD(\ttA) := \coprod_{c \in A} \ttA/c
\end{equation} 
of the slice categories
over the objects of a small category $\ttA$. It is simple
to verify that $\sfD(\ttA)$ is an unary operadic category,
with the fiber functor that assigns to each morphism $\varphi : f' \to
f''$ in $\sfD(\ttA)$ given by the diagram
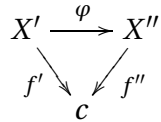
\[
\xymatrix@C=.5em@R=1.5em{X'\ar[rr]^\varphi \ar[dr]_{f'} &&X''\ar[dl]^{f''}
\\
&c&
}
\]
the object $\varphi \in \ttA/X'' \subset \sfD(\ttA)$.

The functor $\sfD:\ttCat \to \ttCat$ has a natural structure of a
comonad in the category of small categories called {\em d\'ecollage
  comonad}, cf.~\cite{GarnerKockWeber}. We briefly describe its
structure morphisms.
Its~counit 
\begin{equation}
\label{Pisi den pred Stedrym dnem.}
\epsilon:\sfD(\ttA)\to \ttA
\end{equation} 
sends an object $X\to c \in \sfD(\ttA)$ to $X$.  

To understand the
comultiplication $\delta: \sfD(\ttA) \to \sfD^2(\ttA)$ we observe that
the objects of $\sfD^2(\ttA)$ are commutative triangles
 \begin{equation}
\label{Jaruska pece cukrovi.} 
    \xymatrix@C = +1em@R = +1em{
      X      \ar[rr]^f \ar[dr]_h & & Y \ar[dl]^g
      \\
      &c&
    }
\end{equation} 
and the set of morphisms between such triangles can be nonempty only
if the right arrow \hbox{$g:Y\to c$} is the same in both triangles. In this
case a morphism from
\begin{equation*} 
    \xymatrix@C = +1em@R = +1em{
      X'      \ar[rr]^{f'} \ar[dr]_{h'} & & Y \ar[dl]^g
      \\
      &c&
    }
\end{equation*} 
to~(\ref{Jaruska pece cukrovi.}) is given by a morphism $\phi: X'\to X$ in $\ttA$
which makes the tetrahedron 
\[
\xymatrix{
&X\ar[rd]^f   \ar[dd]^(.3)h|-\hole      &
\\
X'  \ar[rr]^(.34){f'}    \ar[ru]^\phi\ar[rd]^{h'}  &&Y \ar[ld]_g
\\
&c&
}
\]
commutative. The comultiplication $\delta$ assigns to an object
$f:X\to c$  the commutative triangle
\[
    \xymatrix@C = +1em@R = +1em{
      X      \ar[rr]^f \ar[dr]_f & & c \ar[dl]^1
      \\
      &c&
}
\]
It follows from the following elementary 
categorical  fact that
the functor $\sfD:\ttCat \to \ttCat$
is also a~nonunital monad.

\begin{lemma}
  \label{Za chvili s Jarkou na obed.}
Let $M : \ttC \to \ttC$ be an  endofunctor on a category $\ttC$ and $d :
M \rightsquigarrow \id_\ttC$ a natural transformation of $M$ to the
identity functor on \/ $\ttC$. Then the pair $(M,\mu)$ with $\mu
:= Md$ is a
nonunital monad over \/ $\ttC$, and $d_A :MA \to A$ is an $M$-algebra for
each $A \in \ttC$.
\end{lemma}

\begin{proof}
The diagram of natural transformations 
\begin{equation}
\label{Vanoce}
\xymatrix@C=3em{M^2 \ar@{~>}[d]_{dM}  \ar@{~>}[r]^{Md} & M\ar@{~>}[d]^d
\\
M\ar@{~>}[r]^d & \id_\ttC
}
\end{equation}
commutes by the naturality of $d$.  Applying $M$ on that diagram and
substituting $\mu$ for $Md$ gives the commutative diagram
\[
\xymatrix@C=3em{M^3 \ar@{~>}[d]_{\mu_M}  \ar@{~>}[r]^{M\mu} & M^2\ar@{~>}[d]^{\mu}
\\
M^2\ar@{~>}[r]^{\mu} &\ M
}
\]
meaning that $M$ is a nonunital monad. The second part of the lemma is clear.
\end{proof}

\begin{corollary}
\label{Nesmim zapomenout prinest Jarce tac.}
The functor $\sfD:\ttCat\to \ttCat$ has a natural nonunital monad structure.
\end{corollary}

\begin{proof} 
The assumptions of Lemma~\ref{Za
  chvili s Jarkou na obed.} are fulfilled for  $M := {\sfD}$ and $d$ the
counit $\epsilon$ in~\eqref{Pisi den pred Stedrym dnem.}.
\end{proof} 

\begin{proposition}
\label{Musim se objednat.}
Unary (non-unital) operadic categories with small sets of objects
are algebras for the non-unital
monad~$\,\sfD$ in ${\tt Cat}$.
\end{proposition}

\begin{proof}
A  family $\big\{\Fib_S : \ttO/S \to \ttO \ | \ S \in \ttO\big\}$ of
functors 
is the same as a single functor $\Fib:\sfD(\ttO) \to \ttO$. It is easy
to see that $\Fib$ is a $\sfD$-algebra if and
only if  $\{\Fib_S \}_{S\in \ttO}$ are the fiber functors in
  Definition~\ref{Zacina mne zase svedit i leva ruka?}.  
\end{proof}

\begin{definition}
\label{Dnes prselo a bylo tesne nad nulou.}
The {\em tautological\/} unary operadic category 
generated by a small category $\ttA$ is the category
$\Tau(\ttA) :=\ttA \sqcup \sfD(\ttA)$ with the fiber functors given by the composite
\[
\sfD(\Tau(\ttA))=  
\sfD(\ttA \sqcup \sfD(\ttA)) \cong \sfD(\ttA) \sqcup \sfD^2(\ttA)
\stackrel{\id \sqcup \kappa}  \longrightarrow  \sfD(\ttA)
\hookrightarrow \ttA \sqcup \sfD(\ttA)= \Tau(\ttA).
\]
\end{definition}

Explicitly, the fiber of a morphism $g:S \to T$ in $\ttA \subset
\Tau(\ttA)$ is $g$ again, but
interpreted as an object of 
$\ttA/T \subset\sfD(\ttA) \subset \Tau(\ttA)$, while the fiber of a morphism $F: h \to
c$ of \/ $\sfD(\ttA) \subset \Tau(\ttA)$ represented by the diagram
\[
\xymatrix@R=.5em{&X \ar[dd]^f \ar[ld]_h
\\
T&
\\
&S\ar[ul]^c
}
\]
is $f \in \ttA/S \subset \sfD(\ttA)  \subset \Tau(\ttA)$.

The operadic  category $\Tau(\ttA)$ 
is an operadic subcategory of
$\sfD(\ttA_\odot)$, 
where~$\ttA_\odot$ is the result of formally
adjoining a terminal object $\odot$ to  $\ttA$, which means adding to the
morphisms of $\ttA$ the unit
endomorphism of $\odot$ and one new
morphism $X \stackrel!\to \odot$ for any object $X \in
\ttA$. The inclusion
\begin{equation}
\label{Nic nenalezeno}
\Tau(\ttA) = \ttA \sqcup \sfD(\ttA) \hookrightarrow \sfD(\ttA_\odot)
\end{equation}
sends an object $X \in \ttA$ to $X \stackrel!\to \odot \in
\sfD(\ttA_\odot)$ and objects $X \to Y$ of $\sfD(\ttA)$ into the
corresponding objects of $ \sfD(\ttA_\odot)$. The image of~(\ref{Nic
  nenalezeno}) covers all objects of $\sfD(\ttA_\odot)$ except  $\odot
\stackrel!\to \odot$. Inclusion~(\ref{Nic nenalezeno}) will be used 
in Section~\ref{Boli mne prava noha.} in our description of the operadic
categories of blobs.

\begin{proposition}
\label{Musim se objednat asi jeste dnes.}
The assignment $\ttA \mapsto \Tau(\ttA)$ gives rise to a (unital) monad in
${\tt Cat}$ whose algebras are (non-unital) operadic categories with
small sets of objects. 
\end{proposition}

\begin{proof}
Direct verification.
\end{proof}

\begin{remark} Comparing Propositions~\ref{Musim se objednat.} and~\ref{Musim se
  objednat asi jeste dnes.},   
one may wonder why two very different
monads have the same algebras. The explanation lies in the
unitality of $\Tau$ versus the non-unitality of~$\,\sfD$. 
As a~simple example of this phenomenon, 
consider the nonunital monad $\mathring {\sf T}$ in $\Set$
that sends a set $X$ to the coproduct $\coprod_{n \geq 2} X^{\times
  n}$, and the unital monad ${\sf T}$ given by ${\sf T} X :=\coprod_{n
  \geq 1} X^{\times  n}$. Both $\mathring {\sf T}$ and  ${\sf T}$ have
the same algebras, namely associative (non-unital)
monoids. Ideologically,  ${\sf T}$ is obtained from $\mathring {\sf
  T}$ by freely adjoining the monadic unit. The relation between
$\Tau(\ttA)$ and $\sfD(\ttA)$ is of the same nature. 
\end{remark}

In the following statement, $\Box : {\tt OpCat} \to  {\ttCat}$ is the
obvious forgetful functor from the category of unary (non-unital)
operadic categories with small sets of objects 
to the category of small categories.

\begin{proposition}
For an arbitrary small category $\ttA$ and an unary (non-unital) operadic
category~$\ttO$, one has a natural isomorphism of functor sets
\[
{\tt OpCat}(\Tau(\ttA),\ttO) \cong {\ttCat}(\ttA,\Box\ttO).
\]
In other words, $\Tau(\ttA)$ is the {\em free unary
non-unital operadic category\/} generated by a small category~$\ttA$.
\end{proposition}

\begin{proof}
Direct verification, cf.~\cite[Theorem~2.2]{EM}.
\end{proof}

We are going to introduce (non-unital) operads over (non-unital) unary
operadic categories. Our definition is the non-unital, unary version 
of~\cite[Definition~1.11]{duodel}. 
Unital operadic categories and unital operads will be the subject of
Section~\ref{Budu letat jeste pristi rok?}.

\begin{definition}
\label{Je to tady poustevnictvi.}
Let $\ttV = (\ttV, \ot, \unit)$ be a monoidal, not necessarily
symmetric, 
category and $\ttO$ an unary
operadic category. A~(non-unital) {\em operad\/} for $\ttO$ in $\ttV$,
or simply an {\em
 $\tt0$-operad\/}, is a collection \hbox{$\oP = \{\oP(A)\}_A$}
of objects of $\ttV$ indexed by objects of $\ttO$ together with
structure morphisms
\[
\gamma_h: \oP(F) \ot \oP(B) \longrightarrow \oP(A)
\]
given for any arrow $h : A \to B$ in $\ttO$ with fiber $F$. Moreover,
the  {\em associativity\/} requires that,
for any pair of composable arrows $A \stackrel
f\to B \stackrel g\to C$ in $\ttO$, the diagram
\begin{equation}
\label{Uz se prohlidky blizi.}
\xymatrix@C=3em{\oP(F) \ot \oP(Y) \ot \oP(C)   \ar[d]_{\gamma_{f_C} \ot
    \id}
\ar[r]^(.6){\id \ot \gamma_g}    &\oP(F) \ot \oP(B) \ar[d]^{\gamma_f}
\\
\oP(X) \ot \oP(C) \ar[r]^{\gamma_{gf}}
&\oP(A)
}
\end{equation}
commutes. The meaning of the symbols in that diagram is explained
by the following instance of~(\ref{Po navratu ze Zimni skoly.}):
\begin{equation}
\label{Prohlidka 30.brezna 2022.}
\xymatrix@R=0.1em@C=.8em{\hskip -1.8em F\ \fib X
  \ar[rr]^{f_C}   
&&Y \hskip -.5em
\\
\hskip -2em\raisebox{.7em}{\rotatebox{270}{$=$}}\hskip 1.65em\raisebox{.7em}{{\rotatebox{270}{$\fib$}}} && 
\hskip -.5em\raisebox{.7em}{{\rotatebox{270}{$\fib$}}} \hskip -.5em
\\
\hskip -1.8em F\ \fib A \ar[rr]^f  \ar@/_.9em/[ddddddr]_{gf}   &&
B \ar@/^.9em/[ddddddl]^{g} \hskip -.5em
\\&& 
\\&& 
\\&& 
\\&&
\\ &&
\\
&\ C .&
}
\end{equation}
\end{definition}

We will see later in this paper that operads over unary operadic
categories share many features with associative algebras, 
but also ones that do not have analogs in classical
algebra. We believe that they in their own rite might present an 
interesting theme of research.

\begin{definition}
\label{Tento tyden byl Dominik v Praze.}
Let $\Phi : \ttO' \to \ttO''$ be an operadic functor and $\oP$ an
$\ttO''$-operad. The {\em restriction\/} of $\oP$ along $\Phi$ is the
$\ttO'$-operad $\Phi^*(\oP)$ with components 
$\Phi^*(\oP)(s) := \oP\big(\Phi(s)\big)$, $s \in \ttO''$.
\end{definition}

\begin{example}
\label{Zitra mam sluzbu Radio.}
Let $\ttA$ be a small category and \/ $\sfD(\ttA)$ the associated operadic
category~\eqref{Jesu salvator saeculi}.  A \/
$\sfD(\ttA)$-operad in $\ttV$ is the same as a collection $\oA = \{\oA(f)\}_f$ of
objects of $\ttV$ indexed by morphisms of $\ttA$ with a
`multiplication' $\oA(f) \ot \oA(g) \to \oA(gf)$ for any pair of
composable arrows $A \stackrel f\to B \stackrel g\to C$ of $\ttA$. 
The~multiplication is required to be associative in the obvious sense.
Intuitively, $\oA$ is a non-unital ${\rm Mor}(\ttA)$-graded 
associative algebra in $\ttV$. The unital version of this example has
a nice categorical interpretation, cf.~Example~\ref{PCR test pozitri} below.
\end{example}

\begin{example}
\label{Zitra mam sluzbu.}
An operad over the tautological operadic category  $\Tau(\ttA)$ 
in Definition~\ref{Dnes prselo a bylo tesne nad nulou.} 
is the same as a pair $(\oA,\oM)$ of 
a $\sfD(\ttA)$-operad $\oA$ as in
Example~\ref{Zitra mam sluzbu Radio.}
and a collection $\oM = \big\{\oM(A)\big\}_{A\in \ttA}$ of objects of $\ttV$ indexed by objects of
$\ttA$, equipped with the `actions' $\oA(f) \ot \oM(B) \to \oM(A)$
given for any morphism  $A \stackrel f\to B$ in $\ttA$. It is moreover
required that the diagram
\[
\xymatrix{\oA(f) \ot \oA(g)\ot \oM(C) \ar[r] \ar[d] &\ar[d] \oA(f) \ot \oM(B)
\\
\oA(gf)\ \ot \oM(C) \ar[r] &\oM(A)
}
\]
commutes for any pair  $A \stackrel
f\to B \stackrel g\to C$ of composable morphisms in $\ttA$.
Intuitively,  $\oM$ is a left ${\rm Ob}(\ttA)$-graded 
$\oA$-module.
\end{example}

\section{The shades of unitality}
\label{Budu letat jeste pristi rok?}

In this section we discuss sundry versions of unitality for operadic
categories and their operads, with or without the presence of chosen local
terminal objects. This refinement of definitions given in~\cite{duodel}
is required in Part~2 by the applications to the blob complex.

\begin{definition}
\label{Poleti se zitra?}
Suppose  that the set $\pi_0(\ttO)$ of path-connected components of a unary
operadic category~$\ttO$ is small and a family 
\begin{equation}
\label{Jsem posledni den v Cuernavace.}
\big\{U_c \in \ttO \ | \ c \in \pi_0(\ttO) \big\}
\end{equation}
of {\em chosen local terminal objects\/} is specified, with $U_c$
belonging to the connected component~$c$. 

We say that $\ttO$ is {\em
  left unital\/} if the fiber functor 
$\Fib_S : \ttO/S \to \ttO$ sends the identity
automorphism $\id:S \to S$ to one of the chosen local terminal objects of
$\ttO$, for each $S \in \ttO$.  The category $\ttO$ is {\em right
  unital\/} if $\Fib_{U_c} : \ttO/U_c \to \ttO$ is the domain functor
for each $c \in \pi_0(\ttO)$. Finally, $\ttO$ is {\em unital\/} if it
is both left and right unital.  An operadic functor between left and/or
right unital operadic categories is assumed to preserve the chosen
local terminal objects.
\end{definition}

Unital operadic categories in the sense of the above definition are
precisely the unary versions of operadic categories introduced 
in~\cite[Section~1]{duodel}.  

\begin{remark}
\label{Dovoli mi bradavice lyzovat?}
A choice of local terminal objects is equivalent to the existence of
a sequence of maps
$s_0 := U:\pi_0(\ttO)\to \ttO_0$ and $s_{n+1}:\ttO_n\to \ttO_{n+1}, \ n\ge
0$, in the nerve of $\ttO$ which makes the diagram 
\[
\xymatrix@C=7em{\pi_0(\ttO_0) \ar@/^1em/[r]|-{s_0}     
& \ttO_0 \ar[l]|-{d_0}     
\ar[r]|-{s_0}
\ar@/^2em/[r]|-{s_1}
& \ttO_1   \ar@/^1em/[l]|-{d_0}  \ar@/_1em/[l]|-{d_1}  
\ar@/_1em/[r]|-{s_0} \ar@/^1em/[r]|-{s_1} \ar@/^3em/[r]|-{s_2}
&\ttO_2   \ar[l]|-{d_1}\ar@/^2em/[l]|-{d_0}\ar@/_2em/[l]|-{d_2}
&\hskip -10em\cdots
}
\]
in which $d_0: \ttO_0 \to \pi_0(\ttO_0)$ assigns to an object of
$\ttO$ the corresponding connected component, 
an~`almost' simplicial set.  All simplicial identities that make sense 
should be satisfied
by this new degeneracy operators. Another way to say it is that specifying
locals terminal is the same as to equip $\ttO$ with a
$\sfD$-coalgebra structure \cite[Proposition~5]{GarnerKockWeber}.

The only thing which is missing to make this new complex a simplicial
set is the top face operator $d_{n+1}: \ttO_n\to\ttO_{n-1}$.
As we saw in Remark~\ref{Dnes mi vypalovali bradavici.}, 
a fiber functor on $\ttO$ adds such an operator for all $n>0.$ 
We can also define 
\[
d_1:\ttO_0\to \pi_0(\ttO)
\] 
as the connected component of the fiber of the identity automorphism.
The left unitality then says that these new degeneracy and face operators
satisfy the usual simplicial identities 
\[
d_{n+1}s_i = s_i d_n
\]
for all $i<n$. 
The right unitality implies 
that these operators satisfy the remaining relation 
\[
d_{n+1}s_n = \id.
\] 
And, of course, right and left unitality together mean that the
diagram 
\[
\xymatrix@C=7em{\pi_0(\ttO_0) \ar[r]|-{s_0}     
& \ttO_0
\ar[r]|-{s_0} \ar@/^1em/[l]|-{d_0}\ar@/_1em/[l]|-{d_1}
\ar@/^2em/[r]|-{s_1}
& \ttO_1   \ar@/^1em/[l]|-{d_0}  \ar@/_1em/[l]|-{d_1}
 \ar@/_3em/[l]|-{d_2}  
\ar@/_1em/[r]|-{s_0} \ar@/^1em/[r]|-{s_1} \ar@/^3em/[r]|-{s_2}
&\ttO_2   \ar[l]|-{d_1}\ar@/^2em/[l]|-{d_0}\ar@/_2em/[l]|-{d_2}
 \ar@/_4em/[l]|-{d_3}  
&\hskip -10em\cdots
}
\]
is a simplicial set. The unital unary operadic category 
structure on $\ttO$ is exactly this
extra structure on the nerve of $\ttO$, 
which can be  deduced from the characterization 
of unary operadic categories given in~\cite[Section~4]{GarnerKockWeber}.  

For general operadic categories there exists a similar but more subtle
characterization which involves actions of the symmetric groups on
fibers. This will be the subject of an upcoming work~\cite{BataninKockWeber}.
\end{remark}

\begin{exercise}
\label{Pres vikend nebude pocasi nic moc.}
The operadic category $\sfD(\ttA)$ is unital, with the set of the chosen
local terminal objects
$\big\{c \stackrel\id\to c \in \sfD(\ttA) \ | \ c \in \ttA\big\}$.  The
tautological operadic category $\Tau(\ttA)$ is left unital if and only if
each connected component of $\ttA$ has a terminal object. It is however
never right unital.

To see why, assume that $U \in \ttA$ is a local terminal
object of $\ttA$. By definition, the fiber of a~morphism $a \to U$ is  $a \to U$
again, but now being an object $\sfD(\ttA) \subset \Tau(\ttA)$. Thus the
fiber of  $a \to U$ is not $a$, so $\Fib_U   :\Tau(\ttA)/U \to
\Tau(\ttA)$ is not the domain functor as required by the right
unitality. The~inclusion $\sfD(\ttA) \subset \Tau(\ttA)$ is an example
of a unital operadic subcategory of a non-unital one.
\end{exercise}

\begin{example}
\label{Dnes je tady chladneji.}
Each unital associative monoid $A$ determines a unital unary
operadic category $\CA$ as follows. Its objects are the elements of $A$,
and morphisms are pairs $(x,a) : xa \to a$, for $a,x \in A$. The
composition of the chain 
\[
\xymatrix{
yxa \ar[r]^{(y,xa)} & xa \ar[r]^{(x,a)}& a
}
\]
is the morphism $(yx,a) : yxa \to a$. The identity automorphisms are
the pairs $(e,a) : a  \to a$, where $e$ is the unit of $A$. The fiber
of $(x,a): xa \to x$ is $x$ and the only chosen terminal object is $e$. 
This example and also Examples~\ref{Mam prodlouzene ARC pro Tereje.}
and~\ref{Vecer mi to bliskalo.} below are put in more general context in
Exercise~\ref{Sedim v lese.}. 
\end{example}

\begin{example}
\label{Mam prodlouzene ARC pro Tereje.}
Suppose that the associative monoid $A$ in Example~\ref{Dnes je tady chladneji.}  
possesses, instead of a two-sided unit $e$, 
a family $\{e_b\ | \ b \in A\}$ such that 
each $e_b$ is a right unit for $A$, i.e.\
\begin{subequations}
\begin{equation}
\label{u1}
z\, e_t = z\  \mbox { for each }\  z, t \in A,
\end{equation} 
and a form of the left unitality requiring that
\begin{equation}
\label{u2}
e_{t} \,t =  t \ \hbox { and }    \
e_{tb} \,t = t  \  \mbox { for each }\  t, b \in A,
\end{equation}
\end{subequations}
is fulfilled. Notice that the first equation of~(\ref{u2}) with $b:= e_b$ implies the
second one, but for the reasons explained in the next paragraph 
we keep both of
them. We will call the structures $\big(A,\{e_t\}_{t\in A}\big)$ as above {\em
  pseudo-unital monoids.\/}

Let  $\CA$ be the modification of the category constructed in 
Example~\ref{Dnes je tady chladneji.} with the unit
automorphisms defined by $(e_a,a) : a  \to a$, for $a \in A$.
The first equation in~(\ref{u2}) guarantees that $(e_a,a)$
is indeed an automorphism of $a$, the second one that
$\{(e_t,t)\}_{t \in A}$ are the left units for the composition in
$\CA$, and~(\ref{u1}) that they are the right units. 
We leave as an exercise to prove that \hfill\break  
\hglue 1em -- each $e_t$ is a global terminal object of $\CA$, and \hfill\break 
\hglue 1em -- with an arbitrary $e_t$ as the chosen terminal
object, $\CA$ is a right unital operadic category. \hfill\break 
Moreover,
\[
\hbox{$\CA$ is unital} \iff \hbox{$\CA$ is left unital} \iff
\hbox{$A$ is unital.}
\]
\end{example}

\begin{example}
\label{Vecer mi to bliskalo.}
The set $A : = \{u,v\}$  with a binary operation
given by
\[
uu := u,\ vv := v,\ uv := u \ \mbox { and  } \ vu := u, 
\]
and the `pseudo-units' $e_u := u$ and $e_v := v$, is a pseudo-unital
monoid in the sense of Example~\ref{Mam prodlouzene ARC pro Tereje.}.
The associated category $\CA$ consists of two isomorphic objects $u$
and $v$, related by the isomorphisms $(u,v) : u \to v$ and  $(v,u) : v \to u$.
 
More generally, an arbitrary set
$X$ with a multiplication $X \times X \to X$ given by the projection to the first
factor and pseudo-units $e_t := t$ for $t \in X$ is a pseudo-unital
monoid. The associated operadic category $\ttO_X$ is the chaotic
groupoid generated by $X$. 
\end{example}

\begin{remark}
\label{Pujdu na ocni?}
Let us try to `categorify' pseudo-unital monoids of Example~\ref{Mam
  prodlouzene ARC pro Tereje.} by assembling  the pseudo-units 
$\{e_b\ | \ b \in A\}$
to a single map $\eta:A \to A$ with $\eta(b) := e_b$, $b \in A$. 
The right unitality~(\ref{u1}) is expressed by the diagram
\[
\xymatrix@R=2em{A \times A   \ar[r]^\mu   & A \ar@{=}[d]
\\
\ar[u]^{\id \times \eta} \ar[r]^{\pi_1}
A \times A & A
}
\]
in which $\mu$ is the multiplication in $A$ 
and $\pi_1$ the projection to
the second factor. Similarly,~(\ref{u2}) can be expressed via the diagrams
\[
\xymatrix{A\times A   \ar[r]^{\eta \times \id} & A\times A\ar[d]^\mu
\\
\ar[u]^\Delta
A  \ar@{=}[r]    &A
}
\ \hbox { and } \
\xymatrix{A \times A \times A \ar[r]^{\id \times \tau} &
\ar[r]^{\mu \times \id}A \times A \times A & A\times A
\ar[d]^{\eta \times \id}
\\
\ar[u]^{\Delta \times \id} A\times A 
\ar[r]^{\pi_1}& A &\ar[l]_\mu A\times A
}
\]
where $\Delta$ is the diagonal $a \mapsto a \times a$, $a \in A$.
Since the diagrams above
involve the projection and diagonal, 
pseudo-unitality admits a categorification
only inside a
cartesian monoidal category.
\end{remark}

Below we present three versions of the unitality for
$\ttO$-operads. The unitality in the sense of the
first one is precisely the unary version of the standard
definition~\cite[Definition~1.11]{duodel}.

\begin{definition}
\label{9 dni v Mexiku}
Assume that the unary operadic category $\ttO$ is unital in the sense
of Definition~\ref{Poleti se zitra?}, with the set~(\ref{Jsem posledni den v Cuernavace.}) of chosen
local terminal objects $\big\{U_c \ | \ c
  \in \pi_0(\ttO)\big\}$.
Let $\oP$ be  an $\ttO$-operad in $\ttV$
equipped with a family of morphisms
\begin{equation}
\label{Zitra jedem na pyramidy.}
\big\{\eta_c : \unit \to \oP(U_c) \ | \ c \in \pi_0(\ttO)\big\}.
\end{equation}
%\begin{itemize} 
%\item[(i)] 
We say that $\oP$ is
{\em left unital \/} if, for any  $T\in \ttO$,
the diagram
\begin{equation}
\label{Za chvili zacnu.}
\xymatrix{\oP(U_c) \ot \oP(T) \ar[r]^(.62){\gamma_{\unit}} & \oP(T) \ar@{=}[d]
\\
\ar[u]^{\eta_{c} \ot \id}
\hbox{\hskip 2em $\unit \ot \oP(T)$}\ar[r]^(.6)\cong & \oP(T)
}
\end{equation} 
in which $U_c$ is the fiber  of  the identity automorphism $\id : T \to T$,
commutes.
%\item[(ii)] 

The operad $\oP$ is {\em right unital \/} if, 
for any $F \in \ttO$ and the unique  morphism $!  : F \to U_c$ to
some local chosen terminal object $U_c$, the diagram
\[
\xymatrix{\oP(F) \ot \oP(U_c) \ar[r]^(.62){\gamma_{!}} & \oP(F) \ar@{=}[d]
\\
\ar[u]^{\id \ot \eta_c}
\hbox{\hskip -1.8em $\oP(F) \ot \unit$}\ar[r]^(.5)\cong & \oP(F)
}
\]  
commutes.
%\end{itemize}
Finally, $\oP$ is {\em unital\/} if it is both left and right
unital.
\end{definition}

If the background monoidal category is the cartesian
category $\Set$ of sets, the family~(\ref{Zitra jedem na pyramidy.}) is determined by a choice of {\em units\/} $e_c  := \eta_c(\star)
\in \oP(U_c)$, $c \in \pi_0(\ttO)$, where $\star$ is the unique
element of the monoidal unit $\{\star\}$ of $\Set$.

\begin{example}
\label{PCR test pozitri}
Let $\oA$ be the operad over the operadic category $\sfD(\ttA)$ described
in Example~\ref{Zitra mam sluzbu Radio.}.
Exercise~\ref{Pres vikend nebude pocasi nic moc.} tells us that 
$\,\sfD(\ttA)$ is a unital operadic category  with the chosen
local terminal objects $\big\{\id_c :c \to c \in \sfD(\ttA)
\ | \ c \in \ttA\big\}$. 
Assume the existence of a family
$\{\eta_c : \unit \to \oA(\id_c)\}_{c \in \ttA}$ of morphisms in $\ttV$ indexed
by objects of $\ttA$. The left (resp.~right) unitality of $\oA$ is
then expressed by the  left (resp.~right) diagram below:
\[
\xymatrix{\oA(\id_c) \ot \oA(f) \ar[r] & \oA(f) \ar@{=}[d]
\\
\ar[u]^{\eta_c \ot \id}
\hbox{\hskip 2em $\unit \ot \oA(f)$}\ar[r]^(.6)\cong & \oA(f)
}
\hskip 4em
\xymatrix{\oA(f) \ot \oA(\id_B) \ar[r] & \oA(f) \ar@{=}[d]
\\
\ar[u]^{\id \ot \eta_B}
\hbox{\hskip -1.6em $\oA(f) \ot \unit$}\ar[r]^(.6)\cong & \oA(f)
}
\] 
which are required to commute for any morphism $f: c \to B$ of $\ttA$.

We notice that unital  $\sfD(\ttA)$-operads in $\ttV$ are the same as 
lax $2$-functors $\ttA \to \Sigma \ttV$, with  $\ttA$ considered
as a $2$-category with trivial $2$-cells, and
$\Sigma \ttV$ the $2$-category with one object and $\ttV$ as the
category of morphisms. In particular, if $\ttA$ is the chaotic 
groupoid $\Cha(I)$ on the set $I$, then \hbox{$\sfD(\ttA)$-operads} are
small $\ttV$-enriched categories with the sets of objects $I$. 
\end{example}

The next version of unitality makes sense also if the base operadic
category $\ttO$ is not unital, i.e.~when the chosen local terminal objects as in
Definition~\ref{Svedeni snad prestava.} are not available.

\begin{definition}
\label{O vikendu ma byt krasne a ja trcim v Mexiku.}
Let $\oP$ be an $\ttO$-operad equipped with a family of morphisms
\begin{equation}
\label{Nejdriv si musim prinytovat kolo.}
\big\{\eta_T : \unit \to \oP(U_T)
\ | \ T \in \ttO, \ \hbox {$U_T$ is the fiber of $\id : T \to T$} \big\}.
\end{equation}
We say that $\oP$ is
%\begin{itemize} 
%\item[(i)] 
{\em left unital \/} if, for any $T \in \ttO$, 
the diagram
\begin{equation}
\label{Jeste neco pres tyden v Mexiku.}
\xymatrix{\oP(U_T) \ot \oP(T) \ar[r]^(.62){\gamma_{\unit}} & \oP(T) \ar@{=}[d]
\\
\ar[u]^{\eta_{T} \ot \id}
\hbox{\hskip 2em $\unit \ot \oP(T)$}\ar[r]^(.6)\cong & \oP(T)
}
\end{equation} 
commutes.
%\item[(ii)] 

For an arbitrary morphism $f: T \to S$ of $\ttO$, the axioms of 
unary operadic categories provide the diagram
\begin{equation}
\label{Asi jsem jednal nevhodne.}
\xymatrix@R=0.1em@C=.8em{\hskip -1.8em F\ \fib F
  \ar[rr]^{f_T}   
&&U_T \hskip -.5em
\\
\hskip -1.9em\raisebox{.7em}{\rotatebox{270}{$=$}}\hskip 1.55em\raisebox{.7em}{{\rotatebox{270}{$\fib$}}} && 
\hskip -.5em\raisebox{.7em}{{\rotatebox{270}{$\fib$}}} \hskip -.8em
\\
\hskip -1.8em F\ \fib S \ar[rr]^f  \ar@/_.9em/[ddddddr]_{f}   &&T \ar@/^.9em/[ddddddl]^{\id} \hskip -.5em
\\&& 
\\&& 
\\&& 
\\&&
\\ &&
\\
& T &
}
\end{equation}
with $f_T$ the induced map between fibers. The {\em right unitality\/} 
requires that the induced diagram
\[
\xymatrix{\oP(F) \ot \oP(U_T) \ar[r]^(.62){\gamma_{f_T}} & \oP(F) \ar@{=}[d]
\\
\ar[u]^{\id \ot \eta_T}
\hbox{\hskip -1.8em $\oP(F) \ot \unit$}\ar[r]^(.6)\cong & \oP(F)
}
\]  
commutes for an arbitrary morphism $T \stackrel f\to S$ of $\ttO$.
%\end{itemize}
The operad $\oP$ is {\em unital\/} if it is both left and right
unital.
\end{definition}

\begin{example} 
Consider the algebra $(\oA,\oM)$
over the (non-unital) tautological operadic category  $\Tau(\ttA)$
described in  Example~\ref{Zitra mam sluzbu.}. The left unitality of $(\oA,\oM)$
requires that the $\sfD(\ttA)$-operad $\oA$ is left unital,
cf.~Example~\ref{PCR test pozitri}, and the commutativity of the diagram 
\[
\xymatrix{\oA(\id_c) \ot \oM(c) \ar[r] & \oM(c) \ar@{=}[d]
\\
\ar[u]^{\eta_c \ot \id}
\hbox{\hskip 2em $\unit \ot \oM(c)$}\ar[r]^(.6)\cong & \ \oM(c)
}
\]
for all $c \in \ttA$.
The the right unitality of  $(\oA,\oM)$ is just the right
unitality of $\oA$.
Notice that when  $\ttV$  
is the cartesian category of sets and
$\oA$ the terminal $\sfD(\ttA)$-operad, i.e.\ when if $\oA(f)$
is the one-point set for each morphism~$f$ of $\ttA$, 
unital  $\Tau (\ttA)$-operads are the same as presheaves over $\ttA$.
\end{example}

\begin{proposition}
\label{Poleti to zitra?}
Suppose that the operadic category \/ $\ttO$ is unital.  Then the left
(resp.~right) unitality in the sense of  Definition~\ref{9 dni v
  Mexiku} implies the  left
(resp.~right) unitality in the sense of
Definition~\ref{O vikendu ma byt krasne a ja trcim v Mexiku.}.
In the opposite direction, the (two-sided) unitality of Definition~\ref{O vikendu
  ma byt krasne a ja trcim v Mexiku.} implies the (two-sided) unitality of
Definition~\ref{9 dni v   Mexiku}. 
\end{proposition}

\begin{remark}
Notice that we do not claim that 
the left (resp.~right) unitality of 
Definition~\ref{O vikendu ma byt krasne a ja trcim v Mexiku.}
alone implies the left (resp.~right) unitality of
Definition~\ref{9 dni v  Mexiku}; 
the second part of the
proposition is true only for the (two-sided) unitality. This shall be
compared with the obvious fact that, while an associative algebra admits
at most one two-sided unit, it might have several left or~right
units. An immediate implication of the proposition is that, for operads
over unital unary 
operadic categories, the two definitions provide equivalent notions of
(two-sided) unitality.
\end{remark}

\begin{proof}[Proof of Proposition~\ref{Poleti to zitra?}]
Let $\big\{U_c \ | \ c \in \pi_0(\ttO)\big\}$ be the set~(\ref{Jsem
    posledni den v Cuernavace.}) of chosen local terminal objects
  of~$\ttO$. 
Since the fiber of each identity automorphism in a unital
  unary operadic category is a chosen local terminal object, each $U_T$ in
  Definition~\ref{O vikendu ma byt krasne a ja trcim v Mexiku.}
  equals $U_c$ for some $c \in \pi_0(\ttO)$ uniquely determined by
  $T$. Therefore the
  family~(\ref{Zitra jedem na pyramidy.}) determines
  a~family~(\ref{Nejdriv si musim prinytovat kolo.}) which clearly
  fulfills the left (resp.~right) unitality if the family~(\ref{Zitra
    jedem na pyramidy.})  does. Thus the left (resp.~right) unitality
  of Definition~\ref{9 dni v Mexiku} implies the left (resp.~right)
  unitality of Definition~\ref{O vikendu ma byt krasne a ja trcim v
    Mexiku.}.

Suppose now that $\oP$ is unital in the sense of Definition~\ref{O
  vikendu ma byt krasne a ja trcim v Mexiku.} and prove
that the  map $\eta_T$ in the
family~(\ref{Nejdriv si musim prinytovat kolo.}) 
depends only on $U_T$, not on a concrete $T$. By this we mean that if 
$T$ and $S$ are objects of $\ttO$ such that the identity 
automorphisms $T \to T$ and
$S \to S$ have the same fiber, say $U$, then $\eta_{T} =
\eta_{S}$. It clearly suffices to verify this property for $S=U$.
To show that $\eta_T = 
\eta_{U}$, consider the diagram 
\begin{equation}
\label{Kazdy treti den}
\xymatrix@C=.1em{
&\oP({U})&
\\
&\oP({U}) \ot \oP({U})   \ar[u]^{\gamma_\id}&
\\
\oP({U}) \ot \unit \ar[ur]^{\id \ot \eta_T}  \ar@/^2.2em/[uur]^\cong  && \unit \ot \oP({U})
\ar[ul]_{\eta_{U} \ot \id}
\ar@/_2.2em/[uul]_\cong
\\
&  \ar[ul]_{\eta_{U} \ot \id}\unit \ot \unit \ar[ur]^{\id \ot \eta_T}
&
\\
&\unit \ar[u]^\cong&
}
\end{equation}
associated to the diagram
\[
\xymatrix@R=0.1em@C=.8em{\hskip -1.8em {U}\ \fib {U}
  \ar[rr]^{\id}   
&&{U} \hskip -.5em
\\
\hskip -1.9em\raisebox{.7em}{\rotatebox{270}{$=$}}\hskip 1.55em\raisebox{.7em}{{\rotatebox{270}{$\fib$}}} && 
\hskip -.5em\raisebox{.7em}{{\rotatebox{270}{$\fib$}}} \hskip -.8em
\\
\hskip -1.8em {U}\ \fib T \ar[rr]^\id  \ar@/_.9em/[ddddddr]_{\id}   &&T \ar@/^.9em/[ddddddl]^{\id} \hskip -.5em
\\&& 
\\&& 
\\&& 
\\&&
\\ &&
\\
&\ T .&
}
\]
The commutativity of~(\ref{Kazdy treti den}) 
follows from the left unitality of 
$\eta_{U}$ and the right unitality of $\eta_T$. It~implies 
that the composite of the leftmost arrows equals the composite of
the rightmost arrows, i.e.\ $\eta_U = \eta_T$.

The family~(\ref{Nejdriv si musim prinytovat kolo.}) thus determines a
family~(\ref{Zitra jedem na pyramidy.}) by $\eta_c :=
\eta_{U_c}$ that satisfies the left and right unitality of
Definition~\ref{9 dni v Mexiku}.  
For instance,~(\ref{Za chvili zacnu.}) is fulfilled with
$\eta_T$ in place of $\eta_c$ by~(\ref{Jeste neco pres tyden v
  Mexiku.}), but $\eta_T = \eta_c$ as proven above. The right
unitality is discussed similarly.  
\end{proof}

The following definition extends the concept of pseudo-unitality of
associative monoids introduced in Example~\ref{Mam prodlouzene ARC pro
  Tereje.} to operads. The background monoidal
category will be crucially the cartesian category of sets, 
cf.~Remark~\ref{Pujdu na ocni?}.

\begin{definition}
\label{Uz se tech prohlidek bojim.}
Let $\oS$ be an $\ttO$-operad in $\Set$ equipped with a family of elements
\begin{equation}
\label{Dnes jsem posledni den v Mexiku.}
\big\{
e_t
\in \oS(U_T) \ | \ T \in \ttO, \ t \in \oS(T),\ 
\hbox {$U_T$ is the fiber of $\id : T \to T$}
\big\}.
\end{equation}
Then $\oS$ is
%\begin{itemize}
%\item[(i)]
{\em left pseudo-unital \/} if, 
for any $T \in \ttO$ and $t \in \oS(T)$, 
$
\gamma_{\id} (e_t,t) = t
$
and
\[
\gamma_{\id_C}(e_{\gamma_{\xi}(\rho,c)},\rho) = \rho 
\]
for arbitrary diagram
\begin{equation}
\label{Jak to dopadne s tou moji aurou?}
\xymatrix@R=0.1em@C=.8em{\hskip -1.8em U_T\ \fib R
  \ar[rr]^{\id_C}   
&&R\hskip -.5em
\\
\hskip -1.7em\raisebox{.7em}{\rotatebox{270}{$=$}}\hskip 2.2em\raisebox{.7em}{{\rotatebox{270}{$\fib$}}} && 
\hskip -.5em\raisebox{.7em}{{\rotatebox{270}{$\fib$}}} \hskip -.8em
\\
\hskip -1.8em U_T\ \fib T \ar[rr]^{\id}  \ar@/_.9em/[ddddddr]_{\xi}   &&T \ar@/^.9em/[ddddddl]^{\xi} \hskip -.5em
\\&& 
\\&& 
\\&& 
\\&&
\\ &&
\\
&\ C &
}
\end{equation}
of morphism in $\ttO$ and elements $\rho \in \oS(R)$, $c\in
\oS(C)$.
%\item[(ii)] 
The operad $\oS$ is {\em right pseudo-unital \/} if,
in the situation of diagram~\eqref{Asi jsem jednal nevhodne.},
\[
\gamma_{f_T}(\varphi,e_t) = \varphi
\] 
for any $t \in \oS(U_T)$ and $\varphi \in \oS(F)$.
%\end{itemize} 
Finally, $\oS$ is {\em pseudo-unital\/} if it is both left and
right pseudo-unital.
In this case we call the elements of the collection~\eqref{Dnes jsem
  posledni den v Mexiku.} the {\em pseudo-units\/} of $\oS$. 
\end{definition}

\begin{proposition}
Assume that $\oS$ is an $\ttO$-operad in the category of sets, 
left (resp.~right) unital in the sense
of Definition~\ref{O vikendu ma byt krasne a ja trcim v Mexiku.}. Then
it is left (resp.~right) pseudo-unital. 
\end{proposition}

\begin{proof}
The monoidal unit of the category $\Set$ is the one-point set $\{\star\}$. The
family~(\ref{Nejdriv si musim prinytovat kolo.}) determines a family
as in~(\ref{Dnes jsem posledni den v Mexiku.}) by
$e_t := \eta_T(\star)$. Notice that this $e_t$ depends only on $U_T$,
not on a concrete $t \in \oS(U_T)$. It is simple to verify that if
$\oS$ is left (resp.~right) unital in the sense of  Definition~\ref{O
  vikendu ma byt krasne a ja trcim v Mexiku.}, then 
$\{e_t\}$ are    left (resp.~right) pseudo-units of $\oS$.
\end{proof}

\begin{example}
\label{Psano u Kotsbere.}
Let $\ttO$ be an unital unary operadic category. The operad
$\emptyset$ with $\emptyset(T)$ the empty set for each $T \in \ttO$, is a
pseudo-unital operad, which is however not unital.
This shows that pseudo-unitality 
is less demanding than unitality even when $\ttO$ is unital. 
Below is a less trivial example.
\end{example}

\begin{example}
\label{Asi mam auru ne li neco horsiho.}
Let $\bod$ be the terminal unary unital 
operadic category, i.e.\ the category
with one object $\odot$ and one morphism $\id : \odot \to \odot$ with
fiber $\odot$, which is simultaneously the unique chosen terminal object.
Nonunital $\bod$-operads in $\Set$ are non-unital monoids, i.e.\ sets with one
binary associative operation.
Unital $\bod$-operads are unital monoids, and pseudo-unital
$\bod$-operads are pseudo-unital monoids introduced in
Example~\ref{Mam prodlouzene ARC pro Tereje.}. 
\end{example}

\section{Discrete operadic fibrations}
\label{Cuka mi v oku.}

Discrete operadic fibrations appearing as operadic Grothendieck 
constructions~\cite[Section~2]{duodel} are strong, useful tools for
constructing new operadic categories from old ones, as several
examples given~\cite[Section~4]{env} convincingly show.  
In the first part of this section we recall unital versions of the
relevant definitions given in~\cite{duodel}, the second part is devoted to the
non-unital case required in Part~2 by our applications to the blob complex.

\begin{definition}
\label{psano_v_Myluzach}
  A operadic functor $p:\ttQ\to \ttO$ between unary unital
  operadic categories is  a~{\it discrete operadic
    fibration} if 
\begin{itemize}
\item[(i)]
$p$ induces an epimorphism $\pi_0(\ttQ) \twoheadrightarrow
\pi_0(\ttO)$ of the sets of connected components, and
\item[(ii)]
for any morphism $f : T\to S$ in $\ttO$ with fiber $F$ and any two objects
  $\epsilon, s \in \ttQ$ such that $p(s) = S \mbox { and } p(\epsilon) = F$,
  there exists a unique morphism $\sigma : t\to s$ in $\ttQ$ with fiber
  $\epsilon$ such that $p(\sigma) = f$. Schematically,
\begin{equation}
\label{bliska se mi}
\xymatrix@C=0em{\ \ \ttQ:\ar[d]_p   &\epsilon 
\ar@{|~>}[d]
& \fib& t \ar@{|~>}[d]  \ar@{-->}[rrrrr]_{!}^{\sigma}     &&&&& s\ar@{|~>}[d]
\\
\ \ \ttO :& F & \fib& T   \ar[rrrrr]^f     &&&&& \ S.
}
\end{equation}
\end{itemize}
\end{definition}

The unary version of the operadic
{\em Grothendieck construction\/} \cite[page~1647]{duodel}
associates to a unital $\Set$-valued operad $\oS$ over a unary
unital operadic category~$\ttO$, cf.~Definition~\ref{9 dni v Mexiku}, 
a unary unital operadic category $\ttP$ together with an
operadic functor  $p: \ttP \to \ttO$ as follows. 
Objects of $\ttP$ are elements \hbox{$t \in \oS(T)$} 
for some $T \in
\ttO$. Given $s \in \oS(S)$ and $t \in \oS(T)$, 
a morphism $\sigma: s \to t$  in $\ttP$ is a pair $(\epsilon,f)$
consisting of a morphism $f : S \to T$
in $\ttO$ and an element $\epsilon \in \oS(F)$, where $F$ is the fiber
of $f$, such that $\gamma_{\!f}(\epsilon,t) = s$. The fiber of 
a morphism $\sigma: s \to t$ of this form
is $\epsilon \in \oS(F)$. The unit automorphism
$\id_t: t \to t$ of $t\in \oS(T)$ in $\ttP$ is the pair $(e_c,\id_T)$, where 
$e_c := \eta_{c}(\star) \in \oS(U_c)$,  $U_c$ is the fiber of
the identity automorphism $T \to T$ and $\star$ is the only element of
the monoidal unit  $\{\star\}$ of the category $\Set$. The chosen
local terminal objects of $\ttP$ 
are $\big\{e_c \in \oS(U_c) \ | \ c \in \pi_0(\ttO)\big\}$. 

\label{Budu letat jeste pristi rok? Snad ano.}
The categorical composition in $\ttP$ is given as follows.
Assume that $a \in \oS(A)$, $b \in \oS(B)$ and $c \in \oS(C)$ are objects of
$\ttP$, and $\phi :a \to b$, resp.~$\psi :b \to c$ their morphisms given
by pairs $(\omega,f)$, resp.~$(y,g)$, where $f: A \to B$, resp.~$g:
B\to C$ are morphisms 
of $\ttO$ with the fibers $F$ resp.~$Y$, and $\omega \in \oS(F)$ resp.~$y \in
\oS(Y)$ are
such that 
\[
a = \gamma_f(\omega,b)  \ \hbox { and } \  b =\gamma_g(y,c).
\]
The composite  $\psi\phi$
in  $\ttP$ is defined to be 
the pair  $(x,gf)$, where $x :=  \gamma_{f_C}(\omega,y)$ and $f_C$ is
as in diagram~(\ref{Prohlidka 30.brezna 2022.}).

It turns out that the functor 
$p: \ttP \to \ttO$ that sends $t \in \oS(T)$ to $T \in \ttO$ 
is a discrete operadic fibration. The
correspondence \hbox{$\oS \mapsto \ttP$} is one-to one, as claims the
following  a unary version of~\cite[Proposition~2.5]{duodel}.

\begin{proposition}
 \label{sec:discr-fibr-betw}
The Grothendieck construction provides an equivalence between the category
of unital \/ $\ttO$-operads in the monoidal category of sets, and 
the category of discrete
operadic fibrations of unital unary operadic categories over\/ $\ttO$.
\end{proposition}

Given a discrete operadic 
fibration $p: \ttQ \to \ttO$, the corresponding $\Set$-operad $\oS$
has the components
\begin{equation}
\label{Jarka jede na chalupu sama, snad to zvladne.}
\oS(T) := \{t \in \ttQ \ | \ p(t) = T\}, \ T \in \ttO.
\end{equation}
Any discrete operadic fibration induces and isomorphism $\pi_0(\ttQ)
\stackrel \cong\to \pi_0(\ttO)$ by~\cite[Lemma~2.2]{duodel}. For each
chosen local terminal $U_c \in\ttO$
therefore exists precisely one
chosen local terminal $u_c \in \ttQ$ with $p(u_c) = U_c$. 
The units of $\oS$ are then defined as $e_c := u_c
\in \oS(U_c)$, $c \in \pi_0(\ttO)$.
 
Let us proceed to the non-unital situation. The
modification of discrete operadic fibrations is straightforward:

\begin{definition}
\label{Mam nove ARC}
An operadic functor $p:\ttQ\to \ttO$ between unary, not necessary unital,  
operadic categories is a {\it discrete operadic
fibration} if it has the lifting property in item~(ii) of
Definition~\ref{psano_v_Myluzach}.
\end{definition}

The {\em non-unital\/} version of the
Grothendieck construction has as its input a
pseudo-unital $\Set$-valued operad $\oS$ 
as in~Definition~\ref{Uz se tech prohlidek bojim.}.
The objects of the resulting non-unital operadic category $\ttP$
are the same as in the unital case, and also the morphisms and their
compositions are defined as before. The unit automorphism
$\id_t: t \to t$ of $t\in \oS(T)$ in $\ttP$ is however now 
the pair $(e_t,\id_T)$, where $e_t\in \oS(U_T)$ is as 
in~\eqref{Dnes jsem posledni den v Mexiku.}.

\begin{proposition}
\label{Mam auru?}
The above version of the Grothendieck construction  is 
an equivalence between the category
of pseudo-unital
$\ttO$-operads in $\Set$ and the category of discrete
operadic fibrations over an
unary operadic category \/ $\ttO$.
\end{proposition}

\begin{proof}
Given a discrete operadic
fibration $p: \ttQ \to \ttO$, the corresponding $\Set$-operad $\oS$
has the components as in~(\ref{Jarka jede na chalupu sama, snad to
  zvladne.}). 
The pseudo-unit $e_t \in \oS(U_T)$ associated to  $t \in \oS(T)$ 
is, by definition, the fiber
of the identity automorphism $t \to t$ in $\ttQ$. To verify that this
recipe is the inverse of the Grothendieck construction is straightforward. 
\end{proof}

\begin{example}
\label{Je druhy neletovy den na Safari}
Let $\ttO$ be an unital operadic category.
The Grothendieck construction $\int_\ttO{\bf 1}_\ttO$ of the terminal
unital $\ttO$-operad in $\Set$ is isomorphic to $\ttO$. The Grothendieck
construction $\int \emptyset$ of the `empty' pseudo-unital operad from
Example~\ref{Psano u Kotsbere.} gives the discrete operadic fibration
$\hbox{${\not \hskip -.1em \ttO}$} \to \ttO$ of non-unital operadic
categories, where \hbox{${\not
    \hskip -.1em \ttO}$} is the trivial operadic category (no objects). 
\end{example}

\begin{exercise}
\label{Sedim v lese.}
Let $\bod$ be the terminal unital operadic category in
Example~\ref{Asi mam auru ne li neco horsiho.}. Verify that the
operadic category $\CA$ discussed in Example~\ref{Dnes je tady
  chladneji.} resp.~\ref{Mam prodlouzene ARC pro Tereje.} equals the
Grothendieck constructions $\int_\odot A$, with $A$ interpreted as an
unital, resp.~pseudo-unital $\bod$-operad, cf.~Example~\ref{Asi mam
  auru ne li neco horsiho.}. 
Then show that there are one-to-one correspondences between\hfill\break 
\hglue 1em -- 
unital associative monoids,\hfill\break 
\hglue 1em -- 
unital $\bod$-operads, and\hfill\break 
\hglue 1em -- 
discrete operadic fibrations of unital unary operadic categories over $\bod$.\hfill\break 
Likewise, there are one-to-one correspondences between \hfill\break  
\hglue 1em -- 
pseudo-unital associative monoids,\hfill\break 
\hglue 1em -- 
pseudo-unital $\bod$-operads, and \hfill\break 
\hglue 1em -- 
discrete operadic fibrations of
operadic categories over $\bod$. \hfill\break 
In particular, the chaotic groupoid generated by $X$ is the
Grothendieck construction $\int_\odot X$ of the pseudo-unital monoid
$X$ discussed in Example~\ref{Vecer mi to bliskalo.}, with the fibers
given by the domain functor.
\end{exercise}

\section{Partial operads, partial fibrations}
\label{Podari se mi premluvit Jarku abych mohl jet do Prahy uz dnes?}

The Grothendieck construction used in~\eqref{Divna doba.} of Part~2 to decorate
blobs by fields on their boundaries  uses a
pseudo-unital operad $\oS$ whose structure operations are only
partially defined. This requires further generalization of the
material of Section~\ref{Cuka mi v oku.}. 
Namely, we formulate a~`partial' version of Proposition~\ref{Mam
  auru?} tailored for the context of
Proposition~\ref{Predevcirem jsem se nechal ockovat proti
  chripce.} in Part~2. 

\begin{definition}
\label{Uz se tech prohlidek bojim-partial.}
A partial $\ttO$-operad is a collection  of sets  $\oS = \{\oS(A)\}_{A \in
  \ttO}$ with structure operations
\[
\gamma_h: {\EuScript D}(h) \longrightarrow \oS(A), \ h: A \to B
\hbox { a morphism of $\ttO$ with fiber } F\, ,
\]
defined on a subset ${\EuScript D}(h) \subset\oS(F) \times \oS(B)$. The
domains  $\big\{{\EuScript D}(h)\big\}_h$ are such that, for each
diagram as in~(\ref{Prohlidka 30.brezna 2022.}),
\begin{equation}
\label{Poletam v patek?}
\big(\oS(F) \times \gamma_g({\EuScript D}(g))\big) \cap {
\EuScript D}(f)
= \big(\gamma_{f_C}({\EuScript D}(f_C)) \times \oS(C)\big) \cap {\EuScript D}(gf)
\end{equation}
and $\gamma_f(\id \times \gamma_g)  =  \gamma_{gf}(\gamma_{f_C} \times
\id)$ on the set in~(\ref{Poletam v patek?}). 
\end{definition}

Equation~(\ref{Poletam v patek?}) means that the composites  
$\gamma_f(\id \times \gamma_g)$ and
$\gamma_{gf}(\gamma_{f_C} \times \id)$  are defined on the same subset of
$\oS(F) \times \oS(Y) \times \oS(C)$.

\begin{definition}
\label{Uz se tech prohlidek bojim - bliska se mi.}
Let $\oS$ be a partial $\ttO$-operad as in Definition~\ref{Uz se tech
  prohlidek bojim-partial.} equipped with a family of elements 
\begin{equation}
\label{Dnes jsem posledni den v Mexiku - uz v Praze.}
\big\{
e_t
\in \oS(U_T) \ | \ T \in \ttO, \ t \in \oS(T),\ 
\hbox {$U_T$ is the fiber of $\id : T \to T$}
\big\}.
\end{equation}
We say that~$\oS$ is {\em left pseudo-unital \/} if, 
for any $T \in \ttO$ and $t \in \oS(T)$, 
$\gamma_{\id} (e_t,t)$ is defined and equals $t$. We moreover require
that, for any diagram as in~(\ref{Jak to dopadne s tou moji aurou?})
and elements $\rho \in \oS(R)$, $c\in
\oS(C)$ for which $\gamma_{\xi}(\rho,c)$ is defined,
$\gamma_{\id_C}(e_{\gamma_{\xi}(\rho,c)},\rho)$ is defined and equals $\rho$. 

We say that $\oS$ is {\em right pseudo-unital \/} if,
in the situation of diagram~\eqref{Asi jsem jednal nevhodne.},
$\gamma_{f_T}(\varphi,e_t)$ is defined for any $t \in \oS(U_T)$ and
$\varphi \in \oS(F)$, and equals $\varphi$.
Finally, $\oS$ is {\em pseudo-unital\/} if it is both left and
right pseudo-unital.
\end{definition}

\begin{example}
Partial pseudo-unital operads over the terminal unital operadic
category $\bod$ are partial pseudo-unital monoids. We define them
as partial
associative monoids $A$ equipped with a family $\{e_b \in A\ | \ b \in A\}$ such that 
the product $z\, e_t$ is defined for each $z,t \in A$ and equals $z$ and, if
$tb$ is defined, then $e_{tb} \,t$ is defined and equals $t$, for each
$t,b \in A$. Such partial pseudo-unital monoids are, of course,
partial versions of
pseudo-unital monoids introduced in Example~\ref{Mam prodlouzene ARC pro Tereje.}.
\end{example}

The Grothendieck construction recalled in Section~\ref{Cuka
  mi v oku.} works even when $\oS$ is only a partial
unital $\Set$-valued operad. The objects of the modified 
category $\ttP$ are elements $t \in \oS(T)$, $T \in \ttO$, as before, but the
pair  $(\epsilon,f)$ with $f :S \to T$, $\epsilon \in \oS(F)$ and $F$ the
fiber of $f$, is a morphism $s \to t$ in $\ttP$ 
only if  $\gamma_{\! f}(\epsilon,t)$ is
defined (and equals $s$).

Let us verify that~\eqref{Poletam v patek?} 
guarantees that the categorical composition is defined
for all pairs of morphisms of $\ttP$ whose targets and domains match
as usual.
Assume that $\phi :a \to b$ and $\psi :b \to c$ are as in the
paragraph on page~\pageref{Budu letat jeste pristi rok? Snad ano.},
Section~\ref{Cuka mi v oku.},
where the composition in $\ttP$ is described. Since 
$\gamma_g(y,c)$ is defined and equals $b$, and  $\gamma_f(\omega,b)$
is also defined, the composite $\gamma_f(\omega,\gamma_g(y,c))$
is defined and, thus,  $\gamma_{gf}(\gamma_{f_C}(\omega,y),c)$ is
defined by~\eqref{Poletam v patek?}. In particular, 
$\gamma_{f_C}(\omega,y)$ must be defined, and we define the composite
$\psi\phi$ to be the pair $(x,gf)$, where $x :=  \gamma_{f_C}(\omega,y)$.

The unit automorphism $\id_t :t \to t$ of $t \in \ttP$ is the pair
$(e_t,t)$, where $e_t$ is as in~(\ref{Dnes jsem posledni den v
  Mexiku - uz v Praze.}); notice that $\gamma_{\id} (e_t,t)$ is defined.
The projection $\pi: \ttP \to \ttO$ of unary operadic
categories sends the object $t \in
\oS(T)$ of $\ttP$ to $T \in \ttO$. Let us formulate a `partial' version of
Definition~\ref{Mam nove ARC}.

\begin{definition}
A {\em partial discrete operadic fibration\/} is an operadic functor
$p: \ttQ \to \ttO$ between unary operadic
categories, together with a choice of subsets 
\begin{equation}
\label{Budu asi az do patku.}
\LL(f)
\subset p^{-1}(F)\! \times \! p^{-1} (S), \
\hbox { $f : T \to S$ is a morphism in~$\ttO$ with fiber $F$}.
\end{equation}
The sets $\{\LL(f)\}_f$ are
such that 
\begin{itemize}
\item [(i)]
for any $(\varepsilon,s) \in \LL(f)$ there exists a unique lift
$\sigma$ as in~(\ref{bliska se mi}),
\item [(ii)]
for any morphism $\sigma : t \to s$ in $\ttQ$ with fiber
$\varepsilon$, one has $(\varepsilon,s) \in \LL\big(p(\sigma)\big)$, and
\item [(iii)]
for any $T \in \ttO$ and $t\in p^{-1} (T)$, one has $(u_t, t) \in \LL(\id_T)$,
where $u_t$ is the fiber of the identity automorphism $\id_t : t \to t$.  
\end{itemize}
Denote the lift $\sigma$ of $(\varepsilon,s) \in \LL(f)$ in item (i) above
by $\ell(f,\varepsilon,s)$.
Consider the diagram~(\ref{Prohlidka 30.brezna 2022.}) in $\ttO$
and elements
$y \in \pinv Y$, $r \in \pinv C$ and $\varepsilon \in \pinv F$. 
We require that
\begin{equation}
\label{Je to od zubu?}
(y,c) \in \LL(g)\ \& \ \big(\varepsilon,\ell(g,y,c)\big) \in \LL(f)
\  \iff \
(\varepsilon,y) \in \LL(f_C)\ \& \ \big(\ell(f_C,\varepsilon,y), c\big)
\in \LL(gf).
\end{equation}
\end{definition}

Equivalence~(\ref{Je to od zubu?}) expresses that the lift of the
composite $gf$ exists if and only if there exist composable lifts of $f$ and $g$.
We leave the proof of the following `partial' version of Proposition~\ref{Mam auru?}
to the reader.

\begin{proposition}
\label{Mam v pokoji 18 stupnu a je pulka listopadu.}
The `partial' Grothendieck construction is an equivalence 
between the category
of partial pseudo-unital $\ttO$-operads in $\Set$ 
and the category of partial discrete
operadic fibrations over a unary operadic 
category  \/ $\ttO$.
\end{proposition}

\section{{Operadic modules}}
\label{Zblaznim se z toho?}

The inputs of the classical bar resolution~\cite[Section~X.2]{Homology}
are an associative algebra $\Lambda$ and its (left) $\Lambda$-module $C$. In
Section~\ref{Snad si jeste jednou vytahnu Tereje.} we generalize
the input data to an 
operad $\oP$ and its suitably defined $\oP$-module $\oM$. Operadic modules
are the content of the present section; its floor plan is similar to
that of Section~\ref{Porad se mi bliska.}.

While $\oP$ is, as before, defined over an operadic category $\ttO$, 
$\oP$-modules live over a categorical  `module' $\ttM$ over $\ttO$.
This feature has no analog in the classical algebra. 
The word `module' in the rest of this paper might thus mean either a
categorical module over an operadic category, or a module over an operad. We
believe that the concrete meaning will always be clear from the~context.

\begin{definition}
Let $\ttC$ be a category. A categorical (left) {\em module\/} $\ttL$ {\em over\/}
$\ttC$, or simply a {\em left $\ttC$-module\/},
consists of a set of  objects  $\ttL_0$  and, for each such $L\in \ttL_0$ and each
object $S$ of $\ttC$, a (possibly empty) set of `arrows' $\ttL(L,S)$. 
These data are equipped  with the `actions'
\[
\ttL(L,S) \times \ttC(S,T)  \ni (\alpha , g) 
\longmapsto g\alpha   \in \ttL(L,T), \ L \in \ttL_0,\ S \in \ttC,
\]
which are associative, i.e.\ $(fg)\alpha = f(g\alpha)$ for $\alpha$ and
$g$ as above and $f \in \ttC(T,R)$, and unital, meaning that 
$\id_S \alpha = \alpha$ for each $\alpha \in \ttL(L,S)$ and the
identity automorphism $\id_S \in \ttC(S,S)$.
\end{definition}

Right $\ttC$-modules as well as
$\ttC$-bimodules can be defined analogously, but we will not need
them here.

\begin{remark} 
The notion of a categorical left $\ttC$-module admits the following
categorical and the related simplicial interpretations. Let $\ttL$ be a
categorical left $\ttC$-module. For each $L\in \ttL_0$ consider the
category $L/\ttC$ with object the arrows $\alpha: L\to T$ in
$\ttL(L,T)$ and morphisms the `commutative' triangles
\[
\xymatrix{&L \ar[ld]_\alpha \ar[rd]^\beta&
\\
T&&S\ar[ll]_f
}
\] 
where $\alpha,\beta$ are arrows from $\ttL$ and $f$ is a morphism in
$\ttC.$ Let $\sfD^\ttC(\ttL) = \coprod_{L\in \ttL_0} L/\ttC.$
Taking the target provides us with a functor
$d_0: \sfD^\ttC(L) \to \ttC$ with the lifting property of a discrete Grothendieck
opfibration. We also have the
source-map $d_1:\sfD^\ttC(L) \to \ttL_0$ which obviously factorizes as
\[
\sfD^\ttC(L)\xrightarrow{\pi_0} \pi_0(\sfD^\ttC(L))\xrightarrow{s}
\ttL_0.
\]

Conversely, let ${\sfL}$ be a category equipped with a map of sets
$\gamma:\pi_0(\sfL)\to L_0$ and a discrete opfibration
$t:{\sfL}\to \ttC$. We claim that these data determine a
categorical left $\ttC$-module $\ttL$ with
\[
\ttL_0 = L_0 \ , \ \sfD^{\ttC}(\ttL) \cong \sfL,\  d_0
=t \ \hbox { and } \ d_1 = \gamma\pi_0.
\] 
Indeed, any category is the coproduct of
its connected components, in particular
\[
\sfL = \coprod_{c\in \pi_0(\sfL)}\sfL^c = \coprod_{L\in
  L_0}\Big(\coprod_{c\in \gamma^{-1}(L)}\sfL^c\Big).
\]
Now for $L\in L_0$ and $S\in \ttC$ we define the set of arrow $\ttL(L,S)$ as the
set of objects of $\alpha\in \coprod_{c\in \gamma^{-1}(L)}\sfL^c$ such
that $t(\alpha) = S.$ The action of $\ttC$ is defined using the lifting property
of the opfibration $t$.

Translated to the language of simplicial sets 
this construction amounts to the following. The
simplicial nerve of the category $\sfD^\ttC(\ttL)$,
cf.~Remark~\ref{Dnes mi vypalovali bradavici.},  consists of the
sets $\ttL_n, n\ge 1$, of all possible composable chains
\[
L\stackrel{\alpha}\longrightarrow S_0\stackrel{f_0}\longrightarrow \cdots
\xrightarrow{f_{n-1}}S_{n-1}
\]
where $L\in \ttL_0$ and $S_0,\ldots,S_{n-1}\in \ttC$. The functor
$d_0$ induces a diagram of sets
\begin{equation}
\label{catmod}
\xymatrix@C=7em{
\ttL_0&\ttL_1 \ar[r]|-{s_1}      \ar[l]_{d_1} \ar[dd]^{d_0}
&
\ar@/^1em/[l]|-{d_1}  \ar@/_1em/[l]|-{d_2}  \ttL_2
\ar@/_1em/[r]|-{s_1}
\ar@/^1em/[r]|-{s_2} \ar[dd]^{d_0}
& \ttL_2   \ar@/^2em/[l]|-{d_1}  \ar[l]|-{d_2}  \ar@/_2em/[l]|-{d_3}
 \ar[dd]^{d_0} 
&\hskip -10em\cdots \ \ 
\\
\\
&\ttC_0 \ar[r]|-{s_0}     
&  \ar@/^1em/[l]|-{d_0}  \ar@/_1em/[l]|-{d_1}  \ttC_1
\ar@/_1em/[r]|-{s_0}
\ar@/^1em/[r]|-{s_1}
& \ttC_2   \ar@/^2em/[l]|-{d_0}  \ar[l]|-{d_1}  \ar@/_2em/[l]|-{d_2} 
&\hskip -10em\cdots \ \ .
}
\end{equation}
In this diagram all usual simplicial identities hold, the bottom and the
shifted top simplicial sets satisfy Segal conditions and, moreover,
all (commutative) diagrams involving top horizontal face operators
\[
\xymatrix@C=4em{
\ttL_n  \ar@{<-}[r]^{d_{n+1}}   \ar[d]^{d_0} & \ttL_{n+1} \ar[d]^{d_0}
\\
\ttC_{n-1}  \ar@{<-}[r]^{d_{n}}  & \ttC_{n}
}
\]
are pullbacks. Conversely, any diagram with the above properties
is the `nerve' of a categorical left module.
\end{remark}

\begin{remark}
The rule $(\alpha , g) \mapsto g\alpha$ does not look as a left
action, one would expect $(\alpha , g) \mapsto \alpha g$
instead. This unpleasing feature is due to the bad but favored
convention of writing `$\alpha$ followed by~$g$' as $g\alpha$.
\end{remark}

\begin{example}
  \label{Bojim se.}
Given a category $\ttC$ and a set $S$, one has the chaotic
  $\ttC$-module $\Cha(S,\ttC)$  
with exactly one arrow $L \to T$ for every $L \in
S$ and $T \in \ttC$. A concrete example will be given in
Section~\ref{Boli mne prava noha.}.   
\end{example}

\begin{example}
If both $\ttC$
and $\ttL$ have just one object, the resulting structure is the standard
left module over an associative unital algebra.
\end{example}

\begin{example}
\label{profunctor} 
Each category is a left module over
  itself. More generally, any functor $F:\ttD\to\ttC$ determines a
  left module $\ttL(F)$ over $\ttC$ whose set of objects are objects
  of $\ttD$ and whose set of arrows $\ttL(d,c)$ equals $\ttC(F(d),c)$,
  for $d \in \ttD$, $c \in \ttC$.
  Still more generally, any functor $F:\ttD^{\rm op}\times \ttC \to \Set$
  determines a left $\ttC$-bimodule by similar formulas.  Such
  functors are also known as profunctors, distributors or bimodules from
  $\ttD$ to $\ttC.$
\end{example}

\begin{example} 
\label{Druhy den ve Stokholmu.}
If $\ttL$ is a $\ttC$-module and $c \in \ttC$, then there exits the left
`overmodule' $\ttL/c$  over $\ttC/c$ whose objects 
are arrows $\alpha : L \to c$ in
$\ttL$. Arrows from $\alpha$  to $g: T \to c \in \ttC/c$ are
diagrams 
\[
\xymatrix@R=1.2em@C=1.2em{
L \ar[rr]^\varphi \ar[dr]_\alpha  & & T \ar[dl]^g
\\
&c&
}
\]
in which $\varphi \in \ttL(L,T)$ is such that $\alpha = g\varphi$. 
\end{example}

\begin{definition}
We denote by $\MOD$ the category whose objects are pairs $(\ttC,\ttL)$
consisting
of a category $\ttC$ and its left module $\ttL$. 
Morphisms from $(\ttC',\ttL')$ to
$(\ttC'',\ttL'')$ are pairs $(\Phi,\Psi)$ consisting of a functor
$\Phi: \ttC' \to \ttC''$ and of a rule that assigns to each object $L'$
of $\ttL'$ an object $\Psi(L')$ of $\ttL''$, and to each arrow
\hbox{$\alpha : L' \to S'$} of $\ttL'$ an arrow \hbox{$\alpha_* : \Psi(L') \to
\Phi(S')$} of $\ttL''$ such that the diagram
\[
\xymatrix@R=.5em{& \Phi(S')\ar[dd]^{f_*}
\\
\Psi(L') \ar[ur]^{\alpha_*} \ar[rd]_{(f\alpha)_*}   &
\\
& \Phi(T')
}
\] 
commutes for an arbitrary morphism $f : S' \to T'$ of the category $\ttC'$.
\end{definition}

In the following analog of Lemma~\ref{Mozna mam jen otlacenou
  ruku od operadla.}, $\ttM$ is a left module over an unary operadic 
category $\ttO$
and, for each $S \in \ttO$, $\ttM/S$ denotes the left $\ttO/S$-module
introduced in Example~\ref{Druhy den ve Stokholmu.}.

\begin{lemma}
\label{Asi nemam jen otlacenou ruku od operadla.}
Each family $\big\{(\Fib_S,\Gib_S) : (\ttO/S,\ttM/S) \to (\ttO,\ttM) \ | \ S \in \ttO\big\}$ of
morphisms in $\MOD$ indexed by objects of\/~$\ttO$ induces a~family 
\begin{equation}
\label{Ve stokholmu je obleva.}
\big\{(\Fib_c,\Gib_c) : 
(\ttO/S,\ttM/S) \to (\ttO/\Fib_T(c),\ttM/\Fib_T(c)) \ | \  c :S \to
T\ \big\}
\end{equation} of morphisms in $\MOD$ indexed
by arrows of\/ $\ttO$, with $\Fib_c$ as in Lemma~\ref{Druhy den ve
  Stokholmu.}. 
\end{lemma}

\begin{proof}
Analogous to the proof of  Lemma~\ref{Mozna mam jen otlacenou
  ruku od operadla.}.  
\end{proof}

\begin{definition}
\label{Zacina mne zase svedit i leva ruka!}
Let $\ttO$ be an operadic category
with the associated family $\Fib_S$ of fiber functors, and $\ttM$ a
left $\ttO$-module.
We call a family $(\Fib_S,\Gib_S)$ in Lemma~\ref{Asi nemam jen
  otlacenou ruku od operadla.} a family of {\em fiber
  morphisms\/} if the module part of the extension~(\ref{Ve stokholmu
  je obleva.}) is such that 
\[
\xymatrix@C=1em{\ttM/S \ar[rr]^(.4){\Gib_c} \ar[rd]_{\Gib_S} &&
  \ttM/\Fib_T(c)
\ar[ld]^{\Gib_{\Fib_T(c)}}
\\
&\ttM&
}
\]
commutes for any $c:S \to T$.
\end{definition}

We use similar notation and terminology 
as for operadic categories. That is,
given an arrow \hbox{$\alpha:M \to
S$} in $\ttM$, we call $\Gib_S(\alpha)$ the {\em fiber\/} of $\alpha$ 
and denote it simply
by $\Gib(\alpha)$. The fact that $G = \Gib(\alpha)$ 
will be abbreviated by
$G \fib M \stackrel \alpha\to S$. For a diagram
\begin{equation}
\label{Zase se mi zdalo ze jsem neprosel prohlidkou.}
\xymatrix@C=1em{L\ar[dr]_{\beta}\ar[rr]^{\alpha}& &X \ar[dl]^{g}
\\
&S&
}
\end{equation}
where $\alpha : L \to X$ is an arrow in $\ttM$\,, $g: X \to S$ is a
morphism of $\ttO$ and $\beta = g\alpha$, we
denote by $\alpha_S$ the induced arrow $\Gib(\beta) \to
\Fib(g)$ between the fibers.
The module analog of diagram~(\ref{Po navratu ze Zimni skoly.})
associated to~(\ref{Zase se mi zdalo ze jsem neprosel prohlidkou.}) reads
\begin{equation}
\label{Po navratu ze Zimni skoly II.}
\xymatrix@R=0.1em@C=.8em{\hskip -1.8em G\ \fib H
  \ar[rr]^{\alpha_S}   
&&F \hskip -.5em
\\
\hskip -2em\raisebox{.7em}{\rotatebox{270}{$=$}}\hskip 1.65em\raisebox{.7em}{{\rotatebox{270}{$\fib$}}} && 
\hskip -.5em\raisebox{.7em}{{\rotatebox{270}{$\fib$}}} \hskip -.9em
\\
\hskip -1.95em G\ \fib L \ar[rr]^\alpha  \ar@/_.9em/[ddddddr]_{\beta}   &&
X \ar@/^.9em/[ddddddl]^{g} \hskip -.5em
\\&& 
\\&& 
\\&& 
\\&&
\\ &&
\\
&\ S .&
}
\end{equation}

\begin{definition}
\label{Svedeni asi neprestava.}
An {\em operadic module\/} over a not necessarily unital 
operadic category $\ttO$  
is a categorical left $\ttO$-module $\ttM$ 
equipped with
a family of fiber morphisms as per
Definition~\ref{Zacina mne zase svedit i leva ruka!}. A~{\em
  morphism\/} from an operadic module $\ttM'$ over $\ttO'$ to an operadic module
$\ttM''$ over $\ttO''$ is a morphism $(\Phi,\Psi) : (\ttO',\ttM') \to
(\ttO'',\ttM'')$ in $\MOD$ that commutes with the associated fiber morphisms.
\end{definition}

\begin{remark} 
As for operadic categories, cf.~Remark~\ref{Dovoli mi bradavice
  lyzovat?}, there exists a simplicial
interpretation of operadic modules.
The fiber functor of  a left operadic module $\ttM$ adds to a diagram
as in~(\ref{catmod}) 
additional top face operators, leading to the diagram
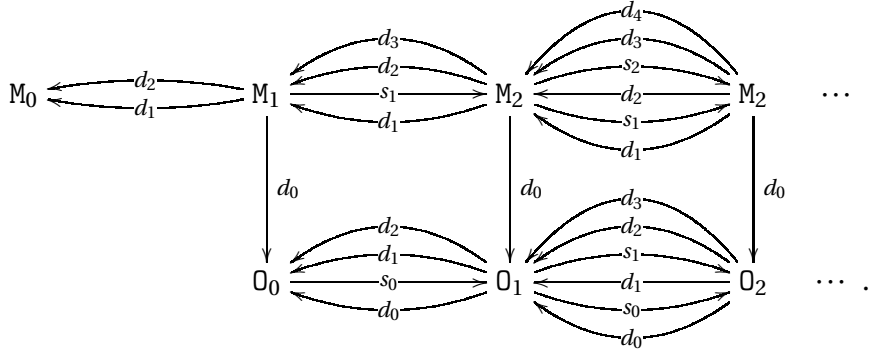
\[
\xymatrix@C=7em{
\ttM_0&\ttM_1 \ar[r]|-{s_1}   \ar@/_.5em/[l]|-{d_2}  
\ar@/^.5em/[l]|-{d_1} \ar[dd]^{d_0}
&
\ar@/^1em/[l]|-{d_1}  \ar@/_1em/[l]|-{d_2} \ar@/_2em/[l]|-{d_3}  \ttM_2
\ar@/_1em/[r]|-{s_1}
\ar@/^1em/[r]|-{s_2} \ar[dd]^{d_0}
& \ttM_2   \ar@/^2em/[l]|-{d_1}  \ar[l]|-{d_2}  \ar@/_2em/[l]|-{d_3}
\ar@/_3em/[l]|-{d_4}
 \ar[dd]^{d_0} 
&\hskip -10em\cdots \ \ 
\\
\\
&\ttO_0 \ar[r]|-{s_0}     
&  \ar@/^1em/[l]|-{d_0}  \ar@/_1em/[l]|-{d_1}  \ttO_1
\ar@/_1em/[r]|-{s_0}  \ar@/_2em/[l]|-{d_2}
\ar@/^1em/[r]|-{s_1}
& \ttO_2   \ar@/^2em/[l]|-{d_0}  \ar[l]|-{d_1}  \ar@/_2em/[l]|-{d_2}
 \ar@/_3em/[l]|-{d_3} 
&\hskip -10em\cdots \ \ .
}
\]
The new face operators satisfy all usual simplicial relations.
\end{remark}

\begin{example}
\label{Za chvili mam sraz s tim postdokem.}
Let $\ttA$ be a small category and $\ttB$ a left $\ttA$-module. Since
each overmodule in
the coproduct $\sfD_\ttA(\ttB):=\coprod_{c \in \ttA}  \ttB/c$ over objects
of $\ttA$ is a module over the category $\sfD(\ttA)$, cf.~Example~\ref{Druhy den ve Stokholmu.},
$\sfD_\ttA(\ttB)$ is  a left $\sfD(\ttA)$-module. With the fiber
functor assigning to an arrow $\psi : \alpha \to f$ in
$\sfD_\ttA(\ttB)$ of the form
\[
\xymatrix@C=1em{L\ar[dr]_{\alpha}\ar[rr]^{\psi}& &X \ar[dl]^{f}
\\
&c&
}
\]
the object $\psi : L \to X$ of $\sfD_\ttA(\ttB)$, the module 
$\sfD_\ttA(\ttB)$ becomes a left operadic
module over the operadic category $\sfD(\ttA)$.
It is
easy to verify that also $\Tau_\ttA(\ttB) : = \ttB \sqcup
\sfD_\ttA(\ttB)$ is a left operadic module over the tautological operadic
category $\Tau(\ttA)$ of Definition \ref{Dnes prselo a bylo tesne nad
  nulou.}. We call $\Tau_\ttA(\ttB)$ the {\em tautological\/}
$\Tau(\ttA)$-module generated by $\ttB$.
\end{example}

\begin{example} 
Let $F:\ttD\to \ttC$ be a functor and $\ttL(F)$ the left
$\ttC$-module constructed in Example~\ref{profunctor}.  Consider
  the category
\[
\sfD(F) =\coprod_{c\in \ttC} F/c.
\] 
There is a natural functor 
\[
\ttF:\sfD(F)\to \sfD(\ttC)
\] 
which sends an object
$F(d)\to c$ of $\sfD(F)$ to the same object but considered as an
object of $\sfD(\ttC).$ Then $\sfD_\ttC(\ttL(F)) = \ttL(\ttF)$ as a
left operadic $\sfD(\ttC)$-module.
\end{example}

Let us formulate a monadic description of left modules over operadic categories, 
analogous to Propositions~\ref{Musim se
  objednat.} and~\ref{Musim se objednat asi jeste dnes.}. Since it
is a straightforward generalization of the material in 
Section~\ref{Porad se mi bliska.}, we will be telegraphic.

Denote by $\mod$ the subcategory of $\MOD$ consisting 
of pairs $(\ttA,\ttB)$ with  $\ttA$ a
small category. One has the functor $\sfD_\mod : \mod \to \mod$ 
given by $\sfD_\mod(\ttA,\ttB) := 
\big(\sfD(\ttA), \sfD_\ttA(\ttB)\big)$
which turns out to be a
nonunital monad. Likewise, the functor  $\Tau_\mod : \mod \to
\mod$ given by 
 $\Tau_\mod(\ttA,\ttB) := 
\big(\Tau(\ttA), \Tau_\ttA(\ttB)\big)$ is a (unital) monad.

\begin{proposition}
Algebras for the nonunital monad $\,\sfD_\mod$, resp.~for the unital monad
$\Tau_\mod$ are pairs $(\ttO,\ttM)$, where $\ttO$ is a non-unital
unary operadic category with a small set of objects, and $\ttM$ its
left operadic module.
\end{proposition}

\begin{definition}
\label{Zitra jedu s Jarkou na chalupu.}
Let $\ttM$ be a left operadic module over an unary
operadic category $\ttO$, and $\oP$ an \hbox{$\ttO$-operad} in $\ttV$. 
A (right) {\em  $\oP$-module\/} in $\ttV$ is a collection $\oM = \{\oM(M)\}_M$
of objects of $\ttV$ indexed by objects of $\ttM$ along with the `actions'
\[
\nu = \nu_\alpha: \oM(G) \ot \oP(L) \to \oM(X)
\]
given for any arrow $\alpha : L \to S$ in $\ttM$ with fiber $G:= \Gib(\alpha)$. Moreover, for any arrow $L \stackrel
\alpha\to B$ in $\ttM$ and a morphisms $B \stackrel g\to C$ in $\ttO$, the diagram
\[
\xymatrix@C=3em{\oM(G) \ot \oP(F) \ot \oP(C)   \ar[d]_{\nu_{\alpha_C} \ot
    \id}
\ar[r]^(.56){\id \ot \gamma_g}    &\oM(G) \ot \oP(B) \ar[d]^{\nu_\alpha}
\\
\oM(X) \ot \oP(C) \ar[r]^{\nu_{g\alpha}}
&\oM(L)
}
\]
is required to commute. The symbols in that diagram are explained
by the following instance of~(\ref{Po navratu ze Zimni skoly II.}):
\[
\xymatrix@R=0.1em@C=.8em{\hskip -1.8em G\ \fib X
  \ar[rr]^{\alpha_C}   
&& F \hskip -.5em
\\
\hskip -2em\raisebox{.7em}{\rotatebox{270}{$=$}}\hskip 1.65em\raisebox{.7em}{{\rotatebox{270}{$\fib$}}} && 
\hskip -.5em\raisebox{.7em}{{\rotatebox{270}{$\fib$}}} \hskip -.5em
\\
\hskip -1.8em G\ \fib L \ar[rr]^\alpha  \ar@/_.9em/[ddddddr]_{g\alpha}   &&
B \ar@/^.9em/[ddddddl]^{g} \hskip -.5em
\\&& 
\\&& 
\\&& 
\\&&
\\ &&
\\
&\ C .&
}
\]
\end{definition}

\begin{definition}
\label{Je to tak na 50 procent.}
Let $(\Phi,\Psi) : (\ttO',\ttM') \to  (\ttO'',\ttM'')$ be a morphism
of left operadic modules, $\oP$ an $\ttO''$-ope\-rad, $\Phi^*(\oP)$ the
restriction of $\oP$ along $\Phi$ as in 
Definition~\ref{Tento tyden byl Dominik v Praze.}, and $\oM$ a
$\oP$-module. The {\em restriction\/} $\Psi^*(\oM)$ of $\oM$ along
$(\Phi,\Psi)$ is the $\Phi^*(\oP)$-module with the components
$\Psi^*(\oM)(m) := \oM\big(\Psi(m)\big)$, for $m \in \ttM''$.
\end{definition}

\begin{example}
\label{Dnes opet prorezavani vetvi.}
Any operadic category $\ttO$ is a left operadic module over
itself. Having this in mind, each $\ttO$-operad is a module over itself. 
\end{example}

\begin{definition}
\label{Porad se mi bliska v oku.}
Suppose that $\oP$ is left unital in the sense of Definition~\ref{O
  vikendu ma byt krasne a ja trcim v Mexiku.} and $\oM$ a $\oP$-module.   
An arbitrary arrow $\alpha: M \to S$ of $\ttM$ induces
the diagram
\[
\xymatrix@R=0.1em@C=.8em{\hskip -1.8em G\ \fib G
  \ar[rr]^{\alpha_T}   
&&U_T \hskip -.5em
\\
\hskip -1.9em\raisebox{.7em}{\rotatebox{270}{$=$}}\hskip 1.55em\raisebox{.7em}{{\rotatebox{270}{$\fib$}}} && 
\hskip -.5em\raisebox{.7em}{{\rotatebox{270}{$\fib$}}} \hskip -.8em
\\
\hskip -1.8em G\ \fib M \ar[rr]^\alpha  \ar@/_.9em/[ddddddr]_{\alpha}   &&T \ar@/^.9em/[ddddddl]^{\id} \hskip -.5em
\\&& 
\\&& 
\\&& 
\\&&
\\ &&
\\
& T &
}
\]
with $\alpha_T$ the induced map between the fibers. The $\oP$-module
$\oM$ is {\em unital\/}  
if the  diagram
\[
\xymatrix{\oM(G) \ot \oP(U_T) \ar[r]^(.62){\nu_{\alpha_T}} & \oM(G) \ar@{=}[d]
\\
\ar[u]^{\id \ot \eta_T}
\hbox{\hskip -1.8em $\oM(G) \ot \unit$}\ar[r]^(.6)\cong & \oM(G)
}
\]  
in which $\eta_T$ is as in~(\ref{Nejdriv si musim prinytovat kolo.}), 
commutes for an arbitrary arrow $M \stackrel \alpha\to T$ in~(\ref{Asi
  jsem jednal nevhodne.}).
\end{definition}

\begin{example}
A unital operad $\oP$ is a unital module over itself,
cf.~Example~\ref{Dnes opet prorezavani vetvi.}. 
\end{example}

\begin{definition}
\label{Vcera jsem behal po promenade.}
Let $(\oM',\nu')$ and $(\oM'',\nu'')$ be left operadic $\oP$-modules. A {\em
  morphism\/} $\bfOmega :  \oM' \to \oM''$ is a family $\bfOmega =
\big\{\Omega_M :  \oM'(M) \to \oM''(M)\big\}$ of morphisms in $\ttV$ indexed by objects of $\ttM$
such that the diagram
\[
\xymatrix{\oM'(G) \ot \oP(L)\ar[r]^(.6){\nu'_\alpha} \ar[d]_{\Omega_G \ot
    \id} & 
\oM'(X)\ar[d]^{\Omega_X}
\\
\oM''(G) \ot \oP(L)\ar[r]^(.6){\nu''_\alpha} & \oM''(X)
}
\]
commutes for each $\alpha : X \to L$ with fiber $G$. We denote by $\hMod$
the corresponding category.
\end{definition}

\section{Free modules}
\label{Zitra pujdu s Jarkou nakupovat bundicku.}

In this section we study the structure of free operadic modules. The
main result, Proposition~\ref{Pozitri mi prijdou
  zapojit novou mycku.}, requires a certain rigidity property that
has no analog in the classical algebra. 
The base monoidal category will be from this point on the category
$\Vect$ of (graded) vector spaces over a commutative unital ring
$R$ though any closed monoidal category would do as well. 

To warm up, we recall the following simple classical facts. 
Let $E$ be a~vector space and $\Lambda$ a non-unital associative
algebra. Then
\[
\hFree(E) := E \oplus (\Lambda \ot E)
\] 
with the left $\Lambda$-action given by
$\lambda(e \oplus a\ot f) := 0 \oplus (\lambda \ot e  + \lambda a\ot
f)$, for $\lambda,a \in \Lambda$ and
$e,f \in E$, is the {free\/} left
$\Lambda$-module generated by $E$.

Assume now that $\Lambda$ possesses a two-sided unit $1\in \Lambda$ and restrict
to the subcategory of
left $\Lambda$-modules on which $1$ acts as the identity
endomorphism. The free $\Lambda$-module in this category is obtained by
identifying, in $\hFree(E)$, $e \oplus 0$ with $0
\oplus 1\ot e$ for each $e\in E$, explicitly
\begin{equation}
\label{Dnes odpoledne byly postdoktorandske seminare.}
\Free(E) :=  \frac{\hFree(E)}{\big(e\oplus 0 = 0
\oplus (1\ot e)\big)}   \cong \Lambda \ot E
\end{equation}
with the left $\Lambda$-action on the right hand side given by
$\lambda(a\ot e):= \lambda a \ot e$.

\begin{remark}
Let us act on both sides of the equality
$e \oplus 0 = 0\oplus (1\ot e)$ in the
denominator of~(\ref{Dnes odpoledne byly postdoktorandske seminare.})
by some $\lambda \in \Lambda$. By the definition of the left $\Lambda$-action, we get the equality $0 \oplus (\lambda
\ot e) = 0 \oplus (\lambda \cdot 1 \ot e)$, which implies the relation
\begin{equation}
\label{Vcera jsme s Jarkou nakoupili bundicky.}
\lambda \cdot 1 \ot e \sim \lambda \ot e\ 
\hbox { for each $\lambda \in \Lambda$, $e \in E$}
\end{equation}
that is $\lambda \cdot 1 \sim \lambda$ for each $\lambda \in
\Lambda$. The assumption that $1$ is also a right, not only the left
unit of $\Lambda$ guarantees that the `unexpected' relation
in~(\ref{Vcera jsme s Jarkou nakoupili bundicky.}) is satisfied automatically.
\end{remark}

Free modules in the operadic context 
have a similarly simple structure only when
the  following unary version of the weak blow-up 
axiom~\cite[Section~2]{env},
abbreviated $\WBU$, is fulfilled.

\vskip .5em
\noindent 
{\bf Weak blow-up} (category version).
For each morphism $f' : X' \to S$ in $\tt0$ with fiber
$F'$, and another morphism $\phi: F' \to F''$,  the
left diagram in~(\WBU) below can be {\em uniquely \/} completed to the diagram in
the right hand side so that $\phi$ will became the map between
the fibers induced by $\varphi$:
\[
\tag{{\tt WBU}}
\xymatrix@R=0.1em@C=.8em{F' \ar[rr]^{\phi}   &&F'' \hskip -.5em
\\
\hskip -2em{\hskip 1.65em\raisebox{.7em}{{\rotatebox{270}{$\fib$}}}}
&&
\\
X'  \ar@/_.9em/[ddddddr]_{f'}   &&
 \hskip -.5em
\\&& 
\\&& 
\\&& 
\\&&
\\ &&
\\
& S &
}
\hskip 6em
\xymatrix@R=0.1em@C=.8em{F' \ar[rr]^{\phi}   &&F'' \hskip -.5em
\\
\hskip -2em{\hskip 1.65em\raisebox{.7em}{{\rotatebox{270}{$\fib$}}}} &&\raisebox{.7em}{{\rotatebox{270}{$\fib$}}}
\\
X' \ar[rr]^\varphi  \ar@/_.9em/[ddddddr]_{f'}   &&
\ X'' . \ar@/^.9em/[ddddddl]^{f''} \hskip -.5em
\\&& 
\\&& 
\\&& 
\\&&
\\ &&
\\
&S&
}
\]
\noindent 
{\bf Weak blow-up} (module version). A straightforward
modification of the operadic blow-up with $\phi, \varphi$ and $f'$
arrows of \, $\ttM$ and $f''$ a morphism in $\ttO$.

\begin{exercise}
The operadic categories $\sfD(\ttA)$ and $\Tau(\ttA)$, as well as the left
modules $\sfD_\ttM(\ttA)$ and $\Tau_\ttM(\ttA)$, satisfy $\WBU$.
\end{exercise}

Let $\oE = \{\oE(M)\}_M$ be a collection of graded vector spaces
indexed by objects of~$\ttM$. For a given object $M \in \ttM$, put
\begin{equation}
\label{Uz mam skoro dva tydny ve Stockholmu za sebou.}
\textstyle
\hFree(\oE)(M) := \oE(M) \oplus \bigoplus_\alpha \big(\oP(T)\ot_\alpha \oE(G)\big), 
\end{equation}
where $\alpha$ runs over arrows $M \stackrel \alpha\to T$ in $\ttM$, $G$ is the
fiber of $\alpha$,
and $\oP(T)\ot_\alpha \oE(G):= \oP(T)\ot \oE(G)$, the subscript
$\alpha$ of the
tensor product symbol
indicating the summand corresponding to this concrete~$\alpha$.

For $\alpha : M \to T$ with the fiber $G$, the action
$\nu_\alpha:\oP(T) \ot \hFree(\oE)(G) \to  \hFree(\oE)(M)$ is
described as follows.
If $e \in \oE(G) \subset  \hFree(\oE)(G)$, the action is
`tautological,' i.e.\
\[
\nu_\alpha(e,t) := t \otimes e \in \oP(T) \ot_\alpha \oE(G) \in \hFree(\oE)(M). 
\]
Let  $s \ot h \in \oP(S) \ot_\beta
\oE(H) \in \hFree(\oE)(G)$, where 
$\beta : G \to S$ has the fiber $H$, and $\alpha$ be as before. In
this situation we have the diagram
\[
\xymatrix@R=-.1em{S &\ar[l]_\beta G \triangleleft H \hskip -2em
\\
\Afib& \Afib 
\\
X\ar@/_.9em/[rdddddd]^{g}
& M\ar[dddddd]^\alpha \ar[l]_(.4){\omega} \triangleleft H  \hskip -2em
\\
\\
\\
\\
\\
\\
& T
}  
\] 
constructed from the initial data $\alpha : M \to T$ and $\beta: G \to
S$  invoking the \WBU. We define the action of $t \in \oP(T)$ by
\[
\nu_\alpha(s \otimes h, t) := \gamma_g(s,t) \ot h \in \oP(X) \ot_{\omega}
\oE(H) \in \hFree(\oE)(M).
\] 

\begin{proposition}
\label{Z tech brejli se mi toci hlava.}
The above structure makes \/ $\hFree(\oE)$ an operadic
\/ $\oP$-module. It is the free left operadic module  generated by $\oE$. 
\end{proposition}

\begin{proof}
The first part of the proposition is a simple exercise. Let us
attend to the freeness. Denote by $\Coll$ the category of collections
$\oE = \{\oE(M)\}_M$ of vector spaces indexed by objects of $\ttM$ and
their component-wise morphisms, and recall from Definition~\ref{Vcera
  jsem behal po promenade.} the category $\hMod$ of left operadic
$\oP$-modules. There is an obvious forgetful functor \, $\Box : \hMod \to
\Coll$. The freeness of~$\hFree(\oE)$ means that, 
for each $\oE \in \Coll$, $\oM \in
\hMod$ and a morphism $\omega :\oE \to \Box\oM$ in $\Coll$, there exists
a unique morphism $\bfOmega :\hFree(\oE) \to \oM$ in $\hMod$ such that the diagram 
\[
\xymatrix{\oE\ar[r]^\omega\ar[d]_(.45)\iota & \Box \oM
\\
\ \Box\hFree(\oE) \ar@/_1em/[ur]^{\Box \bfOmega}
}
\] 
in which $\iota :\oE \to \Box\hFree(\oE)$ is the obvious inclusion,
commutes in $\Coll$.
We prove this claim by  giving an explicit formula for $\bfOmega$.
Namely, for 
\[
e \oplus (t \ot g) \in \oE(M) \oplus (\oP(T) \ot_\alpha \oE(G)) \subset
\hFree(\oE)(M),\ M \in \ttM,
\] 
we put
\[
\Omega_M(e \oplus (t \ot g)) := \omega_M(e) + 
(-1)^{|g||t|}\cdot \nu_\alpha(\omega_G(g),t) \in \oM(M).
\]
It is not difficult to verify that the above formula defines the
required morphism in $\hMod$, 
and that such a morphism is unique.
\end{proof}

Let us discuss the unital version of the above constructions, assuming that
that the operad $\oP$ is left unital in the sense of Definition~\ref{O
  vikendu ma byt krasne a ja trcim v Mexiku.} .
In the situation captured by diagram~\eqref{Asi jsem jednal nevhodne.} denote $1_T :=
\eta_T(1) \in \oP(U_T)$. For each $\alpha : M \to T$ in~\eqref{Asi
  jsem jednal nevhodne.} 
we identify, 
in~(\ref{Uz mam skoro dva tydny ve Stockholmu za
  sebou.}), $e \oplus 0 \in \oE(M)  
\subset \hFree(\oE)(\Upsilon)$
with 
\[
0 \oplus( 1_T
\ot_{\alpha_T} e) \in \big( \oP(U_T) \ot_{\alpha_T} 
\oE(M)\big) \subset \hFree(\oE)(M).
\] 
Finally, we denote  by
$\Free(\oE)$ the quotient of the free nonunital
$\oP$-module $\hFree(\oE)$ 
by the relation generated by the above identifications.

\begin{proposition}
The \/ $\oP$-module  \/
$\Free(\oE)$ is the free unital operadic \hbox{$\oP$-module} generated
by $\oE$. 
\end{proposition}

\begin{proof}
Follows from the freeness of \/ $\hFree(\oE)$ established in
Proposition~\ref{Z tech brejli se mi toci hlava.} combined with the
definition of unitality. 
\end{proof} 

A structure result for free unital $\oP$-modules similar to the
isomorphism in~\eqref{Dnes odpoledne byly postdoktorandske seminare.}
holds only at objects of $\ttM$ that are rigid in the sense of

\begin{definition}
\label{Opet budu nytovat podvozek.}
An object $M$ of an operadic $\ttO$-module $\ttM$ is {\em rigid\/} if
there is  precisely one object $\odot$ of $\ttO$ and  one arrow $M \to
\odot$ with fiber $M$.
\end{definition}

\begin{proposition}
\label{Pozitri mi prijdou zapojit novou mycku.}
For a rigid object $M$  of a left operadic $\ttO$-module \/ $\ttM$, one
has the isomorphism
\begin{equation}
\label{Vcera prijel Benoit Fresse.}
\Free(\oE)(M) \cong \bigoplus_\alpha \big(\oP(T)\ot_\alpha \oE(G)\big),
\end{equation}
where the direct sum runs over all arrows
$\alpha : M \to T$ of \/ $\ttM$ and where  $G$ is the fiber of $\alpha$.
\end{proposition}

\begin{proof}
The rigidity of $M$ guarantees that each $e \oplus 0 \in \oE(M)$ is
identified in $\Free(\oE)(M)$ with precisely one element of the sum in
the right hand side of~(\ref{Vcera prijel Benoit Fresse.}).
\end{proof}

\begin{example}
\label{Je to k zblazneni to blyskani.}
Let $\ttO$ be the terminal unary operadic category with one object $\odot$ and
$\ttM$ the left  $\ttO$-module with one object $\star$ and one arrow
$\star \to \odot$. A left unital $\ttO$-operad is a classical left
unital associative algebra $\Lambda$, and a unital $\oP$-module is the
classical right $\Lambda$-module on which the left unit $1 \in \Lambda$ acts
trivially. The object $\star \in \ttM$ is rigid and~(\ref{Vcera
  prijel Benoit Fresse.}) recovers the isomorphism~(\ref{Dnes odpoledne
  byly postdoktorandske seminare.}).

Let  $\ttO$ be as before, and let $\ttM$ have one object $\star$ but
{\em no\,} arrow. A unital $\oP$-module is just a vector space. 
The free unital $\oP$-module generated by $\oE$ is thus
$\oE$ again, while the right hand side of~(\ref{Vcera prijel Benoit
  Fresse.}) is trivial. This shows that~(\ref{Vcera prijel Benoit
  Fresse.}) 
need not to hold  without the rigidity assumption.   
\end{example}
 
\section{The bar resolution}
\label{Snad si jeste jednou vytahnu Tereje.}

Our aim will be to introduce an operadic analog of the
following classical construction.
Let $\Lambda$ be a (graded) unital associative algebra and $C$ a left
\hbox{$\Lambda$-module}. An (un-normalized) bar resolution of $C$, cf.~\cite[Section~X.2]{Homology}, is an augmented chain complex $\beta_*(\Lambda,C)
\stackrel\epsilon\to C$ of the form
\[
\cdots \stackrel{\partial_{n+2}}\longrightarrow \beta_{n+1}(\Lambda,C)
\stackrel{\partial_{n+1}}\longrightarrow\beta_n(\Lambda,C) 
\stackrel{\partial_{n}}\longrightarrow \beta_{n-1}(\Lambda,C) 
\stackrel{\partial_{n-1}}\longrightarrow \cdots  \stackrel{\partial_1}\longrightarrow
\beta_0(\Lambda,C) \stackrel{\epsilon}\longrightarrow C,
\]
where $\beta_n(\Lambda,C) := \Lambda \ot \Lambda^{\ot n} \ot C$ and
the differential
$\partial_n :  \beta_n(\Lambda,C) \to  \beta_{n-1}(\Lambda,C)$ is
defined as the
sum $\partial_n := \sum_0^n (-1)^i d_i$ with
\[
d_i(\lambda_0  \ot \cdots \ot \lambda_n \ot c) := 
\lambda_0 \ot \cdots \ot\lambda_i\lambda_{i+1} \ot \cdots \lambda_n \ot c
\]
if\, $0 \leq i \leq n-1$, while 
\[
d_n(\lambda_0  \ot \cdots \ot \lambda_n \ot c) := 
\lambda_0 \ot \cdots\ot \lambda_n c,
\]
for $\Rada \lambda0n \in \Lambda, \ c \in C$. The augmentation $\beta_0(\Lambda,C) \stackrel{\epsilon}\longrightarrow C$ is
defined using the left action of $\Lambda$ on $C$ by 
$\epsilon(\lambda_0 \ot c) := \lambda_0  c$. The following theorem is
classical, cf. again~\cite[Section~X.2]{Homology}.

\begin{theorem}
\label{Porad se mi blyska.}
The augmented chain complex $\beta_*(\Lambda,C)
\stackrel{\epsilon}\longrightarrow C$ is an acyclic resolution 
of $C$ via free left $\Lambda$-modules.
\end{theorem}

Let $1 \in \Lambda$ be the unit of $\Lambda$. For each $n \geq 1$
define linear operators $s_j :  \beta_n(\Lambda,C) \to
\beta_{n+1}(\Lambda,C)$, $0 \leq j \leq n$, by
\[
s_j(\lambda_0 \ot \cdots \ot \lambda_n \ot c) := 
\lambda_0 \ot \cdots \ot \lambda_i \ot 1 \ot\lambda_{i+1} \ot \cdots \lambda_n \ot c
\]
for $0 \leq j \le n$, and
\[
s_n(\lambda_0 \ot \cdots \ot \lambda_n \ot c) := 
\lambda_0 \ot \cdots \ot \lambda_n \ot 1
\ot c.
\]
The following statement is also classical.

\begin{proposition}
\label{Pres den taje, v noci mrzne.}
The family  $\beta_\bullet(\Lambda,C) = \{\beta_n(\Lambda,C)\}_{n\geq
  0}$ of vector spaces
with the operators $d_i$ and~$s_j$ defined above is a simplicial
vector space.  The  bar resolution
$\beta_\bullet(\Lambda,C)$ is its associated chain complex.
\end{proposition}

The operadic analog of  $\beta_*(\Lambda,C)$
will possess the similar properties.
The input data will be a~left module $\ttM$ over an 
operadic category $\ttO$,
an $\ttO$-operad $\oP$ and a \hbox{$\oP$-module} $\oM$.
The basic building blocks will be the diagrams $\Twr_M$ of the form
\begin{equation}
\label{Poradmnebolipravepredlokti}
\Twr_M: \
\xymatrix@R=-.1em{T_0&\ar[l]_{f_1} T_1 &\ar[l]_{f_2} \cdots &
\ar[l]_{f_{n-1}} T_{n-1} &\ar[l]_(.4){f_n} T_n &\ar[l]_\alpha M
\\
& \vfib & \cdots &\vfib& \vfib& \vfib
\\
& F_1 & \cdots & F_{n-1}& F_n& N
}  
\end{equation}
where $\Rada T0n$ are objects of \,$\ttO$,
$\Rada f1n$ are morphisms of $\ttO$, and 
 $M \stackrel \alpha\to T_n  $ is an arrow of
$\ttM$. Further, $\Rada F1n$ are the
fibers of $\Rada f1n$, respectively, and $N$ is the fiber of $\alpha$.
We will call $(T_0,\Rada F1n,N)$ the {\em fiber sequence\/} of \, $\Twr_M$.
For $n \geq 0$ denote
\begin{equation}
\label{Jarusku boli brisko.}
\beta_n(\oP,\oM)(M) :=
\bigoplus \oP(T_0) \ot \oP(F_1) \ot \cdots \ot \oP(F_n) \ot \oM(N)
\end{equation}
with the direct sum running over all towers $\Tau_M$
in~(\ref{Poradmnebolipravepredlokti}). We 
assemble the above vector spaces into an
augmented chain complex
\[
\cdots  
\stackrel{\partial_{n+1}}\longrightarrow\beta_n(\oP,\oM)(M) 
\stackrel{\partial_{n}}\longrightarrow \beta_{n-1}(\oP,\oM)(M) 
\stackrel{\partial_{n-1}}\longrightarrow \cdots  \stackrel{\partial_1}\longrightarrow
\beta_0(\oP,\oM)(M) \stackrel{\epsilon}\longrightarrow \oM(M).
\]
Its $n$th differential $\partial_n$  is the sum $\sum_0^n (-1)^i d_i$,
with $d_i$ acting on an element 
\begin{equation}
\label{Pisi v tom podivnem dome.}
t_0 \ot p_1 \ot \cdots \ot p_i \ot p_{i+1} \ot \cdots \ot p_n \ot n
\end{equation}
of \,$\oP(T_0) \ot \oP(F_1)\ot  \cdots \ot \oP(F_i)\ot \oP(F_{i+1}) \ot \cdots 
\ot \oP(F_n) \ot \oM(N)$
as follows. If $1\leq i \leq n-1$, replace first the piece
\[
\xymatrix@R=-.1em{T_{i-1} &\ar[l]_(.4){f_i} T_i &\ar[l]_{f_{i+1}} T_{i+1}
\\
&\vfib & \vfib
\\
&
F_i & F_{i+1}
}
\]
in the tower~(\ref{Poradmnebolipravepredlokti}) by 
\[
\xymatrix@R=-.1em@C=3.4em{T_{i-1} &\ar[l]_(.4){f_if_{i+1}} T_{i+1}
\\
 & \vfib
\\
 & \ F' ,
}
\]
where $F'$ is the fiber of $f_if_{i+1}$.
This situation gives rise to
the following instance of~(\ref{Po navratu ze Zimni skoly.}):
\begin{equation}
\label{Skube mi v prave ruce.}
\xymatrix@R=0.1em@C=.8em{\hskip -3em F_{i+1}\ \fib F'
  \ar[rr]^{(f_{i+1})_{T_{i-1}}}   
&&F_i \hskip -.5em
\\
\hskip -2em\raisebox{.7em}{\rotatebox{270}{$=$}}\hskip 1.65em\raisebox{.7em}{{\rotatebox{270}{$\fib$}}} && 
\hskip -.5em\raisebox{.7em}{{\rotatebox{270}{$\fib$}}} \hskip -.5em
\\
\hskip -1.8em F_{i+1}\ \fib T_{i+1} \ar[rr]^{f_{i+1}}  
\ar@/_.9em/[ddddddr]_{f_{i+1}f_i}   &&
T_i \ar@/^.9em/[ddddddl]^{f_i} \hskip -.5em
\\&& 
\\&& 
\\&& 
\\&&
\\ &&
\\
&\ T_{i-1} .&
}
\end{equation}
The operation $d_i$ now 
replaces $ p_i \ot p_{i+1}$ in~(\ref{Pisi v tom podivnem dome.}) by $\gamma(p_{i+1},p_i)\in \oP(F')$, where $\gamma$ is
the operadic composition induced by  the subdiagram 
$F_{i+1} \fib F' \to F_i$
of~(\ref{Skube mi v prave ruce.}).
To define $d_0$, we cut the left end
\[
\xymatrix@R=-.1em{T_0 &\ar[l]_(.4){f_1} T_1 &\ar[l]_{f_{2}} T_{2}&
  \ar[l]_{f_3} \cdots
\\
&\vfib & \vfib
\\
&
F_1 & F_{2}
}
\]
of the tower~(\ref{Poradmnebolipravepredlokti}) to
\[
\xymatrix@R=-.1em{ T_1 &\ar[l]_{f_{2}} T_{2}&
  \ar[l]_{f_3} \cdots
\\
&\vfib & 
\\
&
F_2 & 
}
\]
and replace $t_0 \ot p_1$ in~(\ref{Pisi v tom podivnem dome.}) by
$\gamma_{f_1}(p_1,t_0) \in \oP(T_1)$.
To define $d_n$, we replace
\[
\xymatrix@R=-.1em{T_{n-1} &\ar[l]_(.4){f_n} T_n &\ar[l]_(.4){\alpha} M
\\
&\vfib & \vfib
\\
&
F_n & N
}
\]
in~(\ref{Poradmnebolipravepredlokti}), by 
\[
\xymatrix@R=-.1em{T_{n-1} &\ar[l]_(.4){f_n\alpha} M
\\
 & \vfib
\\
 & N'
}
\]
which gives rise to
\begin{equation}
\label{Porad mi skube v prave ruce.}
\xymatrix@R=0.1em@C=.8em{\hskip -2em N\ \fib N'
  \ar[rr]^{\alpha_{T_{n-1}}}   
&&F_n \hskip -.5em
\\
\hskip -2em\raisebox{.7em}{\rotatebox{270}{$=$}}\hskip 1.65em\raisebox{.7em}{{\rotatebox{270}{$\fib$}}} && 
\hskip -.5em\raisebox{.7em}{{\rotatebox{270}{$\fib$}}} \hskip -.5em
\\
\hskip -2.2em N\ \fib M \ar[rr]^{\alpha}  
\ar@/_.9em/[ddddddr]_{f_{n}\alpha}   &&
T_n \ar@/^.9em/[ddddddl]^{f_n} \hskip -.5em
\\&& 
\\&& 
\\&& 
\\&&
\\ &&
\\
&\ T_{n-1} .&
}
\end{equation}
We finally replace $p_n \ot m$ in~(\ref{Pisi v tom podivnem dome.}) by
the composite $\nu(n,p_n)\in \oM(N')$ associated to the
subdiagram \hbox{$N \, \fib N' \to F_n$} of~(\ref{Porad mi skube v prave
  ruce.}).  

It remains to attend to $\epsilon:
\beta_0(\oP,\oM)(M) \to \oM(M)$.
By definition,
$\beta_0(\oP,\oM)(M)$ is the direct sum $\bigoplus \oP(T_0) \ot \oM(N)$
over diagrams $N \fib M \stackrel\alpha\to T_0$. 
For $t_0 \ot n \in \oP(T_0) \ot \oM(N)$ we define $\epsilon(t_0 \ot
n)$ to be $\nu_\alpha(n,t_0) \in \oM(M)$.

\begin{proposition}
  One has $\pa_n \pa_{n+1} = 0$ for all $n \geq 1$, and also \ $\epsilon
\pa_1 =0$.
\end{proposition}

\begin{proof}
An exercise on the axioms of operads and their modules.
\end{proof}

\begin{definition}
\label{Prestane mi to paleni?}
The augmented chain complex $\epsilon:
\beta_*(\oP,\oM)(M) \to \oM(M)$ is the (un-normalized) {\em bar resolution\/} of
the $\oP$-module $\oM$ at the object $M$ of $\ttM$.
\end{definition}

Suppose that $\oP$ is unital as per Definition~\ref{O vikendu ma byt
  krasne a ja trcim v Mexiku.} and $\oM$ is a unital $\oP$-module. 
For $M \in \ttM$ define linear maps
\begin{equation}
\label{Za hodinu jdu na ocni.}
s_j : \Beta n(M) \to \Beta {n+1}(M),\  0 \leq j \leq n,
\end{equation} 
as follows. Consider an element 
$u \in \Beta n(M)$ in~(\ref{Pisi v tom podivnem dome.})
associated to the tower $\Twr_M$
in~(\ref{Poradmnebolipravepredlokti}). Then modify $\Twr_M$ by inserting
the identity automorphism of $T_j$ to it, which results in the tower
\[
\xymatrix@R=-.1em{T_0&\ar[l]_{f_1} T_1 &\ar[l]_{f_2} \cdots & T_j
  \ar[l]_{f_j} & T_j \ar[l]_{\id} & \ar[l]\cdots &
\ar[l]_{f_{n-1}} T_{n-1} &\ar[l]_(.4){f_n} T_n &\ar[l]_\alpha M
\\
& \vfib & \cdots &\vfib& \vfib&\cdots & \vfib &\vfib & \vfib
\\
& F_1 & \cdots & F_{j}& U_{T_j}& \cdots&F_{n-1}&F_n& N 
}  
\]
with the fiber sequence $(T_0,\Rada F1j,U_{T_j},\Rada F{j+1}n,N)$. Then
$s_j(u) \in  \Beta {n+1}$ is defined as
\[
  s_j(u) := t_0 \ot p_1 \ot \cdots \ot p_j \ot 1_j\ot p_{j+1} \ot \cdots \ot p_n \ot n
\]
where $1_j:= \eta_{U_{T_j}}(1) \in \oP(U_{T_j})$ is given by the unitality
of the operad $\oP$, cf.~\eqref{Nejdriv si musim prinytovat kolo.}. 
We formulate the following analog of
Proposition~\ref{Pres den taje, v noci mrzne.}.

\begin{proposition}
If \/ $\oP$ is unital, 
the family  $\Beta \bullet(M) = \{\Beta n(M)\}_{n\geq
  0}$ of vector spaces
with the operators $d_i$ and~$s_j$ defined above is a simplicial
vector space for each $M \in \ttM$. The piece
$\Beta*(M)$  of the 
bar resolution  is its associated chain complex.
\end{proposition}

\begin{proof}
Direct verification. 
\end{proof}

\begin{definition}
\label{Za 70 minut.}
For $M \in \oM$ denote by $B_*(\oP,\oM)(M)$ the normalization of the
simplicial abelian group  $\Beta \bullet(M)$, i.e.\ the quotient of
$\Beta *(M)$ by the images of the degeneracy operators~(\ref{Za hodinu
  jdu na ocni.}). 
The augmented chain complex $\epsilon:
B_*(\oP,\oM)(M) \to \oM(M)$ is the {\em normalized bar resolution\/} of
the $\oP$-module $\oM$ at the object $M$ of $\ttM$. 
\end{definition}

In the rest of this section we assume that the operadic category
$\ttO$ and the left operadic $\ttO$-module $\ttM$ satisfies the weak
blow-up axiom recalled in Section~\ref{Zblaznim se z toho?}.

\begin{proposition}
\label{Dnes mne ty ruce pali mene.}
The  direct sums $\oM \oplus \Beta0$ and $\Beta{n}
\oplus \Beta{n+1}$, $n \geq 0$, are free
left \,$\oP$-modules. 
\end{proposition}

\begin{proof}
It follows directly from the definition of  $\Beta
0$ and~(\ref{Uz mam skoro dva tydny ve Stockholmu za sebou.}) 
that $\oM \oplus \Beta 0 \cong \Free(\Box \oM)$. It thus remains 
to prove that
\[
\Beta n (N) \oplus \Beta {n+1}(N)  \cong \Free(\Box \Beta n)(N),
\]
for each $n \geq 0$ and $N \in \ttM$.
Invoking~(\ref{Uz mam skoro dva tydny ve Stockholmu za sebou.}) again, we
conclude that it suffices to show that
\begin{equation}
\label{Dneska je vymena pradla.}
\Beta {n+1}(N) \cong \bigoplus_\varpi \big(\oP(S)\ot_\varpi \Beta n (M)\big),
\end{equation}
where the direct sum is taken over all arrows $\varpi: N \to S$, and where
$M$ is the fiber of $\varpi$. 
Consider the component
\[
\oP(T_0) \ot \oP(F_1) \ot \cdots \ot \oP(F_n) \ot \oM(N)
\]
in the direct sum~(\ref{Jarusku boli brisko.}) defining
$\beta_n(\oP,\oM)(M)$  
associated to the tower $\Twr_M$
in~(\ref{Poradmnebolipravepredlokti}). Using $\WBU$
\hbox{$(n+1)$}-times,
we embed $\Twr_M$ as the tower of morphisms between the fibers into the diagram
\[
\xymatrix@R=-.1em{T_0&\ar[l]_{f_1} T_1 &\ar[l]_{f_2} \cdots &
\ar[l]_{f_{n-1}} T_{n-1} &\ar[l]_(.4){f_n} T_n &\ar[l]_\alpha M
\\
\Afib & \Afib & \cdots &\Afib& \Afib& \Afib
\\
S_1\ar@/_1.9em/[rrrrrdddddd]^{g_1}& \ar@/_.9em/[rrrrdddddd]^{\varpi_{2}} S_2\ar[l]_(.4){g_2} & \ar[l]_(.4){g_3}
\cdots &\ar@/_.9em/[rrdddddd]^{\varpi_{n}} S_{n}\ar[l]_(.4){g_n}
& S_{n+1}\ar[l]_(.5){g_{n+1}}\ar@/_.9em/[rdddddd]^{\varpi_{n+1}}
& N\ar[dddddd]^\varpi \ar[l]_(.4){\beta}
\\
\\
\\
\\
\\
\\
 &  &  & & & S 
}  
\]
which contains the tower
\[
\varpi(\Twr_M): \
\xymatrix@1{S & \ar[l]_(.4){g_1}
S_1&  S_2\ar[l]_(.4){g_2} & \ar[l]_(.4){g_3}
\cdots & S_{n}\ar[l]_(.4){g_n}
& S_{n+1}\ar[l]_(.5){g_{n+1}}
& N \ar[l]_(.4){\beta}
}
\]
with the fiber sequence $(S,T_0,\Rada F1n,N)$.
We may thus interpret the summand
\[
\oP(S) \ot  \oP(T_0) \ot \oP(F_1) \ot \cdots
\ot \oP(F_n) \ot \oM(N) \subset \oP(S)\ot_\varpi \Beta n (M)
\] 
in the right hand side of~(\ref{Dneska je vymena pradla.}) as
belonging to the
component of $\Beta{n+1}(N)$ associated to the tower $\varpi(\Twr_M)$. It
is easy to show that the map
\[
\bigoplus_\varpi \big(\oP(S)\ot_\varpi \Beta n (M)\big) \longrightarrow
\Beta {n+1}(N)
\]
thus defined is an isomorphism.
\end{proof}

Let us move to the issue of acyclicity of the bar resolution.
For an object $\Upsilon \in \ttM$ denote by $\Upsilon/\ttO$ the
category whose objects are arrows $\alpha : \Upsilon \to X$, $X \in
\ttO$, and morphisms $\alpha' \to \alpha''$ are commutative diagrams 
\[
\xymatrix@R1em@C1em{&\Upsilon  \ar[ld]_{\alpha'}\ar[rd]^{\alpha''}  &
\\
X' \ar[rr] &&\ X''
}
\]
where the horizontal arrow is a morphism of $\ttO$. It will turn out
that the bar resolution is acyclic at objects $\Upsilon$ of $\ttM$ 
with the property that\label{Dovoli mi to pocasi v nedeli?}
\begin{itemize}
\item [(P1)] 
the category $\Upsilon/\ttO$ has a global terminal object $\Upsilon \stackrel!
\to \odot$ such that 
\item [(P2)]
the fiber of  $\Upsilon \stackrel!
\to \odot$ is $\Upsilon$ and the fiber functor $\Fib:\ttO/\odot \to \ttO$
is the domain functor.
\end{itemize}
The terminality in (P1) means that each
arrow $\alpha: \Upsilon \to X$ in $\ttM$ is uniquely left divisible:
\[
\xymatrix{\Upsilon \ar[r]^\alpha \ar@/_1.9em/[rr]^{!} &
  \ar@{-->}[r]^{\exists \ !} X& \odot
}.
\]
By (P2), the fiber of the unique arrow $X \stackrel!\to
\odot$ is~$X$. 

Assume that $\oP$ is left unital the sense of Definition~\ref{O
  vikendu ma byt krasne a ja trcim v Mexiku.} and $\oM$ a unital
$\oP$-module as in Definition~\ref{Porad se mi bliska v oku.}. Since
the fiber of the identity $\id : \odot \to \odot$ is $\odot$, by the
left unitality of $\oP$ there
exists a morphism $\eta_\odot : \unit \to \oP(\odot)$ for which the diagram
\[
\xymatrix{\oP(\odot) \ot \oP(X) \ar[r]^(.62){\gamma_{\unit}} & \oP(X) \ar@{=}[d]
\\
\ar[u]^{\eta_{\odot} \ot \id}
\hbox{\hskip 2em $\unit \ot \oP(X)$}\ar[r]^(.6)\cong & \oP(X)
}
\]
commutes.
Likewise, the unitality of
$\oM$ implies the commutativity of
\[
\xymatrix{\oM(\Upsilon) \ot \oP(\odot) \ar[r]^(.6){\nu_!} & \oM(\Upsilon) \ar@{=}[d]
\\
\hskip -1.4em \oM(\Upsilon) \ot R  \ar[r]^\cong \ar[u]_{\id \ot
  \eta_\odot}  
&\ \oM(\Upsilon).
}
\]

We are finally able to formulate the operadic version of
Theorem~\ref{Porad se mi blyska.}. Recall that we assume that
the operadic category
$\ttO$ and the left operadic $\ttO$-module $\ttM$ satisfy the weak
blow-up axiom, $\oP$ is a unital $\ttO$-operad in $\Vect$ and $\oM$ a
unital $\oP$-module.

\begin{theoremA}
\label{Veta A}
The augmented chain complex 
$\beta_*(\oP,\oM)(\Upsilon) \stackrel\epsilon\to \oM(\Upsilon)$ 
is acyclic whenever
$\Upsilon\in \ttM$ fulfills properties (P1)--(P2) above.
If \/ $\Upsilon$ is rigid in the sense of Definition~\ref{Opet budu
  nytovat podvozek.}, then
$\beta_n(\oP,\oM)(\Upsilon)$ is 
a~piece of a free unital \/ $\oP$-module  for each  $n \geq 0$. An obvious
similar statement holds also for the normalized bar construction $B_*(\oP,\oM)$.
\end{theoremA}

\begin{proof}
It follows from the definition of  $\Beta
0$ and~(\ref{Vcera prijel Benoit
  Fresse.}) that $\Beta 0(\Upsilon) \cong \Free(\Box \oM)(\Upsilon)$. We need
to prove that
\[
\Beta {n+1}(\Upsilon)  \cong \Free(\Box \Beta n)(\Upsilon),
\]
for each $n \geq 0$ which, 
by~(\ref{Vcera prijel Benoit Fresse.}), amounts to proving that
\begin{equation}
\label{Dneska je vymena pradla I.}
\Beta {n+1}(\Upsilon) \cong \bigoplus_\varpi \big(\oP(S)\ot_\varpi \Beta n (M)\big),
\end{equation}
where the direct sum is taken over all $\varpi: \Upsilon \to S$, and where
$M$ is the fiber of $\varpi$. But this isomorphism was established in
the proof of Proposition~\ref{Dnes mne ty ruce pali mene.},
cf.~(\ref{Dneska je vymena pradla.}) with $N = \Upsilon$. 

To prove the acyclicity of the augmented complex 
$\Beta*(\Upsilon) \stackrel \epsilon\to \oM(\Upsilon)$, we 
construct  the contracting homotopies $h: \oM(\Upsilon) \to \Beta0 (\Upsilon)$ and  $h_n: \Beta n(\Upsilon)
\to \Beta{n+1} (\Upsilon)$, $n \geq 0$, as follows. 
For $u \in \oM(\Upsilon)$ we put
\[
h(u) := 1_\odot \ot u  \in \oP(\odot) \ot_! \oM(\Upsilon)
\in \Beta0(\Upsilon).
\]

To construct $h_n$ for $n \geq 0$, we consider a tower \,$\Twr_\Upsilon$ as in~(\ref{Poradmnebolipravepredlokti}) with $M = \Upsilon$
and a related element \hbox{$u \in \Beta n(\Upsilon)$} in~(\ref{Pisi v tom podivnem
  dome.}). Since $\Upsilon \in \pi_0(\odot)$, all $T_i$'s and $T_0$ in
particular belong to $\pi_0(\odot)$, so there exists a unique $!: T_0 \to
\odot$, hence $\Twr_\Upsilon$ can be uniquely extended to
\[
h(\Twr_\Upsilon): \
\xymatrix@R=-.1em{\odot&T_0 \ar[l]_{!}&\ar[l]_{f_1} T_1 &\ar[l]_{f_2} \cdots &
\ar[l]_{f_{n-1}} T_{n-1} &\ar[l]_(.4){f_n} T_n &\ar[l]_\alpha M
\\
& \vfib & \vfib & \cdots &\vfib& \vfib& \vfib
\\
& T_0 & F_1 & \cdots & F_{n-1}& F_n& N
}  
\]
with the associated 
fiber sequence $(\odot,T_0,\Rada F1n, N)$. We finally define
$h_n(u)$ to be $1_\odot \ot u$ in the component of $\Beta{n+1} (\Upsilon)$
corresponding to the tower $h(\Twr_\Upsilon)$. The desired property of 
the contracting homotopies for $h,h_0,h_1,\ldots$ constructed above is 
easy to verify. 
\end{proof}

We note that the right unitality of $\oP$ only is sufficient for the
acyclicity of $\beta_*(\oP,\oM)(\Upsilon) \stackrel\epsilon\to
\oM(\Upsilon)$. Theorem~A has the following obvious 

\begin{corollary}
For each $\Upsilon$ fulfilling (P1)--(P2) above,  one has 
\[
H_0(\Beta*) (\Upsilon) = \frac{\oM(\Upsilon)}{{\rm Span}\{ \nu_\alpha(n,x)\}},
\]
where \,${\rm Span}\{ \nu_\alpha(n,x)\} \subset \oM(\Upsilon)$ is the
subspace spanned by all $\nu_\alpha(n,x)$'s with $\alpha : \Upsilon \to X$,
$n \in \oM(N)$ and $x \in \oP(X)$, where $N$ is the fiber of
$\alpha$. The higher homology of $\Beta*(\Upsilon)$ is trivial.
\end{corollary}

\begin{remark}
If $\ttO$ is the terminal operadic category with one object $\odot$,
and the left $\ttO$-module $\ttM$ has of one object $\star$ and one arrow $\star \to
\odot$ as in Example~\ref{Je to k zblazneni to blyskani.}, 
then the above operadic constructions reduce to the
classical machinery \`a la MacLane~\cite{Homology} recalled at the
beginning of this section. 
\end{remark}

The following kind of functoriality has no analog in classical
homological algebra.

\begin{proposition}
\label{Jak to vsechno dopadne?}
Let $(\Phi,\Psi) : (\ttO',\ttM') \to  (\ttO'',\ttM'')$ be a morphism
of operadic left modules, let $\oP$ be an $\ttO''\!$-ope\-rad, $\Phi^*(\oP)$
its 
restriction along  \/  $\Phi$, $\oM$ a
$\oP$-module and  \/  $\Psi^*(\oM)$ its restriction along  \/ $(\Phi,\Psi)$.
Then, for each $M \in \ttM'$, there exists a natural chain map
\[
\beta_*(\Phi,\Psi)(M) :\beta_*\big({\Phi^*(\oP)},{\Psi^*(\oM)}\big)(M) \to  \Beta*\big(\Psi(M)\big)
\]
such that the diagram 
\[
\xymatrix@C=5em{\beta_*\big({\Phi^*(\oP)},{\Psi^*(\oM)}\big)(M)
  \ar[d]_{\beta_*(\Phi,\Psi)(M) } \ar[r] & \ar@{=}[d]
  \Psi^*(\oM)(M)
\\
\Beta*\big(\Psi(M)\big) \ar[r] &\ \oM\big(\Psi(M)\big),
}
\]
in which the horizontal arrows are the augmentations, commutes.
\end{proposition}

\begin{proof}
The component 
\[
\Phi^*(\oP)(T_0) \ot \Phi^*(\oP)(F_1) \ot \cdots \ot \Phi^*(\oP)(F_n) \ot \Psi^*(\oM)(N)
\]
of $\beta_*\big({\Phi^*(\oP)},{\Psi^*(\oM)}\big)(M)$ corresponding to
the tower in~(\ref{Poradmnebolipravepredlokti}) is mapped to the component 
\[
\oP\big(\Phi\big(T_0)\big) \ot \oP\big(\Phi(F_1)\big) \ot \cdots 
\ot \oP\big(\Phi(F_n)\big) \ot \oM\big(\Psi(N)\big)
\]
of $\Beta*\big(\Psi(M)\big)$ corresponding to the tower
\[
\xymatrix@R=-.1em@C=3em{\Phi(T_0)&\ar[l]_{\Phi(f_1)} \Phi(T_1) &\ar[l]_(.4){\Phi(f_2)} \cdots &
\ar[l]_(.6){\Phi(f_{n-1})} \Phi(T_{n-1}) &\ar[l]_(.45){\Phi(f_n)}
\Phi(T_n) &\ar[l]_{\Psi(\alpha)}\ \Psi(M)\ .
\\
& \vfib & \cdots &\vfib& \vfib& \vfib
\\
& \Phi(F_1) & \cdots &\Phi(F_{n-1})& \Phi(F_n)& \Psi(N)
}  
\]
It is simple to check that the morphism thus constructed has the
desired properties.
\end{proof}

\part{Fields and blobs}

\section{Basic notions}
\label{Utopil jsem logger in telefon, ale vzpamatovali se.}

We will deal with manifolds, their boundaries, 
embeddings, \&c. The precise meanings of these nouns  will depend on the setup in
which we choose to work -- manifolds could be topological, smooth,
piecewise linear, with some additional structures, \&c. Since
all constructions below are of combinatorial and/or algebraic
nature, we allow ourselves to be relaxed about the
nomenclature; compare the intuitive approach in the related sections
of~\cite{Walker}. We reserve $d$ for a~non-negative integer.

\begin{definition}
\label{Clarifica me pater} 
A {\em system of fields\/} is a rule that to each
manifold $X$ of dimension $\leq d\+1$ assigns a set $\C(X)$. This
assignment should satisfy properties listed
e.g.~in~\cite[Section~2]{blob}. We will in particular need the
following:
\begin{itemize}
\item[(i)]
For each codimension-zero submanifold $Z$ of $\pa X$ one has a
functorial 
restriction \[
\C(X) \ni c  \mapsto c|_Z \in \C(Z).
\]
\item[(ii)]
Let $X' \sqcup_Z X''$ be a manifold obtained by glueing manifolds
$X'$ and $X''$ along a common piece~$Z$ of their boundaries, 
and $c' \in \C(X')$, resp.~$c'' \in \C(X'')$ be
fields whose restrictions to $Z$ agree. Then $c'$ and $c''$ can be
glued into a field $c' \sqcup_Z c'' \in \C(X' \sqcup_Z X'')$ which restricts to
the original fields on $X'$ resp.~$X''$.
\end{itemize}
\end{definition}

We will denote by $\C(X;c)$ the subset of $\C(X)$ consisting of fields
that restrict to $c \in \C(Z)$. We will always be in the situation when $Z$
in~(ii) is closed, so we will not need `glueing with corners'
described in~\cite[Section~2]{blob}. We allow $Z$ to be empty, in
which case we denote the result of the glueing by $c' \sqcup c''\in
\C(X' \sqcup X'')$.  Standard examples of fields are
$\C(X)$ the set of maps from~$X$ to some fixed space $B$, or $\C(X)$
the set of equivalence classes of $G$-bundles with connection over~$X$.

\begin{definition}
\label{Te lucis ante terminum}
Let $X$ be a $(d\!+\!1)$-dimensional, not necessarily connected, manifold,
with the (possibly empty) boundary $\pa X$ and the interior
$\Int X$.
A {\em blob\/} \,in $X$ is the image  $D$ of the 
standard closed $(d\!+\!1)$-dimen\-sional ball ${\mathbb D}^{d+1}
\subset {\mathbb R}^{\times {(d+1)}}$
embedded in $X$ in so that either
\begin{itemize}
\item[(i)]
$D \subset \Int X$, and $X \setminus \Int D$ is a 
$(d\!+\!1)$-dimensional
manifold with the boundary $\pa X \cup \pa D$, or
\item[(ii)]
$D$ is one of the connected components of $X$.
\end{itemize}
A {\em configuration of blobs\/} in $X$ is a nonempty unordered finite set
$\fD = \{\Rada D1r\}$ of pairwise disjoint blobs in $X$.
\end{definition}

We will sometimes use $\fD$ also to denote the union  $\bigcup \fD
:=  \bigcup_{i=1}^r D_i$ if the meaning is clear from the context. It
is a manifold with the
boundary $\pa \fD := \bigcup_{i=1}^r \pa D_i$ and the interior
$\Int \fD := \bigcup_{i=1}^r \Int D_i$, and $X \setminus
\Int \fD$ is a $(d\+1)$-dimensional manifold with the
boundary $\pa X \cup \pa \fD$. If
some of the blobs in $\fD$ happen to be some of 
the components of~$X$, then the
corresponding components of~$X
\setminus \Int \fD$ are $d$-dimensional 
embedded spheres, interpreted
as degenerate $(d\+1)$-dimensional manifolds with empty interiors.

In the rest of this paper we assume that the space of fields on
codimension zero submanifolds of $X$ is {\em enriched\/} in $\Vect$. This is
the case e.g.\ when the fields come from a $(d\!+\!1)$-category whose spaces
of top-dimensional cells are linearly enriched, cf.~\cite[Subsection~2.2]{blob}.

\begin{definition}
\label{zacina i leve}
A {\em local relation\/} is a collection of subspaces $\loc=\big\{\loc(D;c)
\subset \C(D;c)\big\}$ specified for any blob $D \subset X$ and any field $c \in
\C(\pa D)$, which is an ideal in
the following sense.

Suppose we are given blobs $D'$ and $D$ such that $\Int{D} \supset
D'$, and a field
$c \sqcup c' \in\C(\pa D \sqcup \pa D')$.
Then for any local relation 
 $u \in
  \loc(D';c')$ and any field $r \in \C(D \setminus
  \Int{D'};c \sqcup c')$, the glueing $u\sqcup _{\pa D'}\! r$ is a local
  relation in $\loc(D;c)$.
\end{definition}

It is a simplified version of \cite[Definition 2.3.1]{blob},
sufficient for our
purposes. The configuration of balls and fields in
Definition~\ref{zacina i leve} is depicted in the
following schematic picture:
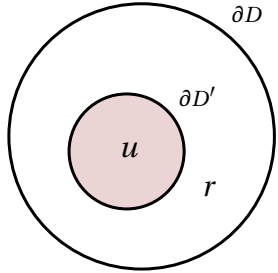
\[
\psscalebox{1.0 1.0} % Change this value to rescale the drawing.
{
\begin{pspicture}(0,-1.82)(3.6,1.82)
\definecolor{colour0}{rgb}{0.92156863,0.83137256,0.83137256}
\pscircle[linecolor=black, linewidth=0.04, fillstyle=solid, dimen=outer](1.8,-0.02){1.8}
\pscircle[linecolor=black, linewidth=0.04, fillstyle=solid,fillcolor=colour0, dimen=outer](1.6,-0.22){0.8}
%\rput[bl](1.2,-0.42){$D'$}
%\rput[bl](0.8,0.78){$D$}
\rput[bl](3.0,1.5){\scriptsize $\partial D$}
\rput[bl](2.3,0.38){\scriptsize $\partial D'$}
\rput[bl](2.6,-0.82){$r$}
\rput[bl](1.5,-.3){$u$}
\end{pspicture}
}
\]

\begin{definition}
\label{Vydrzi to jeste 14 dni?}
For local relations as in Definition~\ref{zacina i leve}, denote by
$\loc(X)$ the subspace of $\C(X)$ spanned by fields of the form $u \sqcup_{\pa
  D}\! r$, where $D \subset X$ is a blob, $r \in \C(X \setminus \Int{D}; c)$ a
field and $u \in \loc(D;c)$ a local relation. A TQFT invariant of $X$ called the 
{\em skein module\/} associated to a system of fields~$\C$ and local
relations $\loc$ is the quotient
\[
A(X) : = \C(X)/\loc(X).
\]
If $X$ has a non-empty boundary $\pa X$ and $b \in \C(\pa X)$,
one has the obvious restricted version
\[
A(X;b) : = \C(X;b)/\loc(X;b).
\]
\end{definition}

\section{Blobs via unary operadic categories}
\label{Boli mne prava noha.}

We are going to introduce various operadic categories and operadic
modules together 
with the related operads and their modules, arising from the blobs and fields in
Section~\ref{Utopil jsem logger in telefon, ale vzpamatovali se.}. Our
aim is to describe  the associated  
bar resolutions, cf.~the second half of
Section~\ref{Snad si jeste jednou vytahnu Tereje.}.
In Section~\ref{Jarka mi koupila cokoladu.} we show that they are
quasi-isomorphic to the original blob complex in~\cite{Walker}. The
notation is summarized at the end of this section.

Let us fix a connected $(d\+1)$-dimensional, not necessarily closed, non-empty 
manifold $\fM$. If its boundary $\pa \fM$ is non-empty, some
constructions below will depend on a fixed field $b$ on $\pa \fM$. We
will however often omit such a boundary condition from the notation.   

We  denote by $\blob$ the
category opposite to the category of 
configurations of blobs in~$\fM$ and their
`well-behaved' inclusions. More precisely, objects of $\blob$ are
configurations $\fD$  of blobs in $\fM$ as in Definition~\ref{Te lucis
  ante terminum}, 
and a unique map
$\fD' \to \fD''$ exists if and only if $\fD''$ is a blob configuration
in the union of blobs in $\fD'$, 
cf.~Definition~\ref{Te lucis
  ante terminum} again. In what follows, by an inclusion of blob
configurations we will {\em always mean\/} a well-behaved inclusion in this sense.

Let us turn our attention to the operadic 
category  $\sfD(\blob)$ associated to the small category $\blob$ via the
recipe of~\eqref{Jesu salvator saeculi}.
Its objects are inclusions $i' :\fD' \hookleftarrow \fD$ of blob
configurations in $\fM$.
Notice that there is a one-to-one
correspondence between these inclusions, i.e.~objects of
$\sfD(\blob)$, and {\em blob complements\/}, which we define as closed
submanifolds of~$\fM$ of the form $\bigcup \fD' \setminus \bigcup\Int{
  \fD} $ with  $\fD$ a blob configuration in $\fD'$. 
Morphisms   $i' \to i''$ of $\sfD(\blob)$ are~diagrams
\begin{equation}
\label{Prohlidka se blizi.}
\xymatrix@C=1em{\fD' && \ar@{_{(}->}[ll] \fD''
\\
& \ar@{^{(}->}[ru]^{i''}  \ar@{_{(}->}[lu]_{i'}  \fD& 
}
\end{equation}
of inclusions of blob configurations. 
The fiber of the above morphism  is
the inclusion $\fD' \hookleftarrow \fD''$ interpreted as an object of
$\blob/\fD'' \subset \sfD(\blob)$. 

Let $\ublob$ be the category $\blob$  with a 
terminal object $\varnothing$ formally added. 
In this particular case, $\varnothing$ can be
viewed as the empty configuration of blobs, whence the notation.
Inclusion~(\ref{Nic nenalezeno}) describes  the
tautological operadic category $\Blob : = \Tau(\blob)$ as the
subcategory of $\uB:=\sfD(\ublob)$
whose objects
are inclusions $i' :\fD' \hookleftarrow \fD$, where~$\fD$~is allowed to
be {\em empty\/}. If this is so, we identify $i'$ with $\fD'
\in \blob$. Morphisms of $\Blob$ then arise as diagrams
in~\eqref{Prohlidka se blizi.} with possibly empty~$\fD$.

Every system of fields $\C$ in Definition~\ref{Clarifica me pater}
leads to the decorated version $\dblob$ of the operadic category
$\blob$. Its objects are pairs $(\fD;c)$ consisting of blob
configurations $\fD$ in $\fM$ and of a field $c \in \C(\pa
\fD)$. Morphisms $(\fD';c') \to (\fD'';c'')$ are inclusions \hbox{$\fD'
\hookleftarrow \fD''$} of blob configurations subject to the
condition:
\begin{center}
if the blob configurations $\fD'$ and $\fD''$ share a common blob $D$, 
then $c'|_{\pa  D} = c''|_{\pa D}$. 
\end{center}

We will tacitly assume that all inclusions of decorated blobs
satisfy the above condition. Denoting by $\udblob$ the category $\dblob$ extended  by the empty blob,
the tautological operadic category 
$\dBlob : = \Tau(\dblob)$ becomes the full subcategory of
$\uB(\C) :=\sfD(\udblob)$ whose objects are
`extended' morphisms $i':  (\fD';c') \hookleftarrow (\fD;c)$ in  $\dblob$ with
$\fD$ allowed to be empty, and morphisms the diagrams
\begin{equation}
\label{Mam to tomu doktorovi rict?}
\xymatrix@C=.1em{(\fD';c') && \ar@{_{(}->}[ll] (\fD'';c'')
\\
& \ar@{^{(}->}[ru]^{i''}  \ar@{_{(}->}[lu]_{i'}  (\fD;c)& 
}
\end{equation}
with $\fD$ allowed to be empty. 

It turns out that $\dBlob$  is the partial operadic
Grothendieck construction, in the sense of
Section~\ref{Podari se mi premluvit Jarku abych mohl jet do Prahy uz
  dnes?}, over its un-decorated version. Explicitly
\begin{equation}
\label{Divna doba.}
\dBlob \cong \int_{\ \Blob } \oS,
\end{equation}
where $\oS$ is the following partial pseudo-unital $\Blob$-operad in
$\Set$.
The component of $\oS$ corresponding to $\fD' \hookleftarrow \fD \in \Blob$ is
the set $\C(\pa \fD' \cup \pa \fD)$. We denote this component by~$\colorop{\oS}(\fD;\fD'\,)$.
The partial composition
\begin{equation}
\label{Bojim se toho.}
\gamma:
\colorop{\oS}(\fD'';\fD'\,) \times 
\colorop{\oS}(\fD;\fD''\,)
\longrightarrow \colorop{\oS}(\fD;\fD'\,)
\end{equation}
associated to the morphism in~(\ref{Mam to tomu doktorovi rict?}) is
defined for the pairs $(x,y)$ of fields 
\[
x \in \C(\pa\fD'
\cup \pa\fD'') = 
\colorop{\oS}(\fD'';\fD'\,)
\ \hbox { and } \
y \in \C(\pa\fD''
\cup \pa\fD) = 
\colorop{\oS}(\fD;\fD''\,)
\]
such that 
\begin{equation}
\label{Zitra to na Tereje jeste nebude.}
x|_{\pa \fD''} = y|_{\pa \fD''}
\end{equation} 
in which case 
\[
\gamma(x,y) := x|_{\pa \fD'} \cup y|_{\pa \fD} \in \C(\pa \fD'\cup \pa\fD) =
\colorop{\oS}(\fD;\fD'\,).
\] 

Let $T \in \Blob$ be the
inclusion  $\fD' \hookleftarrow \fD$. The fiber $U_T$ of the identity
$T \to T$ is the inclusion $\fD' \hookleftarrow \fD'$. By definition,
\[
\oS(T)= \C(\pa \fD'
\cup \pa \fD)\ \hbox { and }\ \oS(U_T) = \C(\pa \fD').
\]
The pseudo-unit  $e_t$ 
in~\eqref{Dnes jsem posledni den v Mexiku - uz v Praze.}  
associated to a field $t \in \oS(T)$ is
the restriction  $e_t: = t|_{\pa \fD'}\in \oS(U_T)$. 

\begin{proposition}
\label{Predevcirem jsem se nechal ockovat proti chripce.}
The isomorphism~(\ref{Divna doba.}) holds for the partial pseudo-unital operad \/  $\oS$ defined above. The natural
projection $\dBlob \to \Blob$ that forgets the decorating fields 
is thus a partial discrete operadic Grothendieck fibration.
\end{proposition}

The proposition is easy to check. 
The subspace $\LL(f)$ in~(\ref{Budu asi az do patku.})
associated to the partial  discrete operadic Grothendieck fibration $\dBlob \to \Blob$ equals
\[
\LL(f) = \big\{ (\varepsilon, s) \in \C(\pa \fD' \cup \pa \fD'')
\times \C(\pa \fD \cup \pa \fD'') \ \big | \ \varepsilon|_{\pa \fD''} = 
s|_{\pa \fD''} \big\} 
\]
when $f$ is the morphism~(\ref{Mam to tomu doktorovi rict?}).

We will also need modules arising from blobs and fields. 
Let us denote by
$\ttm$ the left $\blob$-module with one object $\fM$ and the unique
arrow  $\fM \to \fD$  for each configuration $\fD$ of blobs in~$\fM$. 
In the terminology of Example~\ref{Bojim se.}, $\ttm$ is the
chaotic module $\Cha\big(\{\fM\}, \blob\big)$.
Likewise, let $\uttm := \Cha\big(\{\fM\}, \ublob\big)$ be the 
left $\ublob$-module with one object $\fM$ and one
arrow  $\fM \to \fD$  for each configuration $\fD$ of blobs,
plus one  arrow $\fM \to \varnothing$. 

Referring to Example~\ref{Za chvili mam sraz s tim postdokem.}, we
introduce the tautological operadic $\Blob$-module
$\bM:=\Tau_{\blob}(\ttm)$. It is, by the obvious analog
of the inclusion~(\ref{Nic nenalezeno}), 
the operadic submodule of the $\uB$-module  
$\Mbar :=\sfD_{\ublob}(\uttm)$.
Objects of $\bM$ appear in $\Mbar$ as inclusions $\fM \hookleftarrow \fD$ of
blob configurations, where~$\fD$ might be empty, and the diagram
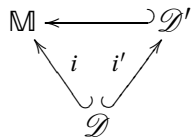
\[
\xymatrix@C=1em{\fM && \ar@{_{(}->}[ll] \fD'
\\
& \ar@{^{(}->}[ru]^{i'}  \ar@{_{(}->}[lu]_{i}  \fD& 
}
\]
in $\Mbar$ represents an arrow from  $\fM \hookleftarrow \fD \in \bM$ to  $\fD'
\hookleftarrow \fD\in \Blob$ with fiber  $\fM \hookleftarrow \fD' \in \bM$.
In the above diagram, $\fD$ is again allowed to be empty.

Every system of fields in Definition~\ref{Clarifica me pater}
gives rise to the decorated versions of the above modules. 
Namely, we have the left $\blob(\C)$-module $\ttm(\C)$ with single 
object the pair $(\fM;b)$, where
\hbox{$b \in \C(\pa \fM)$} is the fixed boundary condition. 
By definition, $\ttm(\C)$ has one arrow $(\fM;b) \to (\fD;c)$
for each configuration $(\fD;c)$ of decorated blobs such that $\fD \subset
\Int{\fM}$. If $\fD$ consists of a single blob $D$ and $\fM = D$, then
the arrow $(\fM;b) \to (\fD;c)$ exists only if and only if 
$b= c$. Thus $\ttm(\C)$ is a chaotic module unless 
$\fM$ is a ball. 

Let $\uttm(\C)$ be the left $\ublob(\C)$-module
obtained from $\ttm(\C)$ by adding one arrow $(\ttM;b) \to
\varnothing$ for each  $(\ttM;b) \in \ttm(\C)$.
The $\Blob(\C)$-module $\ttM(\C) : = \Tau_{\Blob(\C)}(\ttm(\C))$ is
then a natural submodule of $\Mbar(\C) :=\sfD_{\ublob(\C)}(\uttm(\C))$
whose objects are inclusions $(\fM;b) \hookleftarrow 
(\fD;c)$ where $\fD$ is allowed
to be~empty. 

\begin{lemma}
\label{Mirna nadeje?}
All objects of \/ $\Mbar(\C)$  that have the form 
$(\fM;b) \to  \varnothing$   are rigid in the sense
of Definition~\ref{Opet budu nytovat podvozek.} and satisfy (P1)--(P2) on page~\pageref{Dovoli mi to pocasi
  v nedeli?}. On the contrary, none of the objects of \/ $\ttM(\C)$ is rigid and none of them satisfies  (P1)--(P2).
\end{lemma}

Notice that $(\fM;b) \to  \varnothing$  is the image of $(\fM;b)$ under
the natural inclusion $\ttM(\C) \hookrightarrow \Mbar(\C)$. Thus
$(\fM;b)$ becomes rigid and satisfying (P1)--(P2) when considered as
an object of  $\Mbar(\C)$. Therefore $\Mbar(\C)$ is a kind of
completion of \/ $\ttM(\C)$, whence the notation.

\begin{proof}[Proof of Lemma~\ref{Mirna nadeje?}]
The object $(\fM;b) \to  \varnothing$ of \/  $\Mbar(\C)$ has the
desired properties for $\odot:= \varnothing \to \varnothing$. The
second part of the lemma is easy to check.
\end{proof}

Assume, as in the previous section, that the fields on
codimension-zero submanifolds of $\fM$ are linearly enriched.
We are going to define a $\uB(\C)$-operad  $\fopn$
with values in $\Vect$ as follows. If $\fD$ and $\fD'$ are nonempty
blob configurations, we~put
\[
\Krtek {(\fD';c')}{(\fD;c)}
:= \big\{f \in \C(\fD' \setminus
\Int{\fD}) \  \big| \ 
f|_{\pa \fD} = c,\ f|_{\pa \fD'} = c' \big\}.
\]
The definition is competed by setting
\[
\Krtek {(\fD;c)}\emptyset
:=\loc(\fD;c),
\]
where $\loc(\fD;c) \subset \C(\fD;c)$ is the subspace of local
relations, cf.~Definition~\ref{zacina i leve},
that restrict to $c$ at~$\pa \fD$.
Finally, 
\[
\Krtek\emptyset\emptyset:=  R,\
\hbox { the ground ring.}
\]
Thus $\fopn$ is composed of fields that extend the given ones on the boundary.
The structure operation
\[
\Krtek{(\fD';c')}{(\fD'';c'')}
\otimes 
\Krtek{(\fD'';c'')}{(\fD;c)}
\longrightarrow \Krtek{(\fD';c')}{(\fD;c)}
\]
associated to the morphism in~\eqref{Mam to tomu doktorovi rict?}
is given by the glueing of fields along $\pa \fD''$. The operad $\fopn$
is unital, with the units
\[
c \in \Krtek{(\fD;c)}{(\fD;c)}
\ \hbox { if $\fD \not= \emptyset$, and }
1 \in \Krtek\emptyset\emptyset.
\]
We will also use the  $\Blob(\C)$-operad  $\fop$ defined as 
the restriction of the $\uB(\C)$-operad $\fopn$ along the inclusion 
$\Blob(\C) \hookrightarrow \uB(\C)$. 
We define an $\fopn$-module  $\oMn$ by
\[
\Krtekmodularni{(\fM;b)}{(\fD;c)}
:= \big\{f \in \C(\fM \setminus
\Int{\fD}) \  \big| \ 
f|_{\pa \fD} = c,\ f|_{\pa \fM} = b \big\}
\]
if $\fD \not= \emptyset$, and
\[
\Krtekmodularni{(\fM;b)}{\emptyset}
:= \C(\fM;b).
\]
Denote finally by $\oM$ the 
$\fop$-module which is the restriction, cf.~Definition~\ref{Je
  to tak na 50 procent.},  of the 
$\fopn$-module  $\oMn$ along the pair $(\iota,\jmath) :
(\Blob(\C),\ttM(\C)) \to (\uB(\C),\Mbar(\C))$ of the
natural inclusions.

Referring to Definition~\ref{Prestane mi to paleni?}, we will consider
for a fixed field $b$ on the boundary of $\fM$
two augmented complexes, namely
\[
\beta_*(\fopn,\oMn)\big((\fM;b) \to \varnothing\big) \longrightarrow \C(\fM;b)\
 \hbox { and } \
\beta_*(\fop,\oM)(\fM;b) \longrightarrow \C(\fM;b).
\]
By the functoriality of Proposition~\ref{Jak to vsechno dopadne?}, the
pair $(\iota,\jmath)$ induces the commutative diagram
\[
\xymatrix@C=4em{
\beta_*(\fop,\oM)(\fM;b) \ar[r]  \ar@{^{(}->}[d]& \C(\fM;b) \ar@{=}[d]
\\
\beta_*(\fopn,\oMn)\big((\fM;b) \to \varnothing\big)
  \ar[r] &
  \C(\fM;b)
}
\]
of augmented complexes. The next theorem follows from Theorem~A, 
Lemma~\ref{Mirna nadeje?}
and an easy computation.

\begin{theoremB}\label{Veta B}
The augmented complex $\beta_*(\fopn,\oMn)\big((\fM;b) \to \varnothing\big) 
\to \C(\fM;b)$ is a component of an acyclic resolution of \/ $\oMn$ via
unital free \/ $\fopn$-modules. In particular
\[
H_0\big(\beta_*(\fopn,\oMn)\big((\fM;b) \to \varnothing)\big) \cong  \C(\fM;b).
\]
The complex $\beta_*(\fop,\oM)(\fM;b) \to \C(\fM;b)$ resolves
the skein module in Definition~\ref{Vydrzi to jeste 14 dni?}, namely
\[
H_{-1}\big(\beta_*(\fop,\oM)(\fM;b) \to \C(\fM;b)\big) \cong  A(\fM;b).
\]
An obvious similar statement holds also for 
the normalized bar construction.
\end{theoremB}

The subscript $-1$ of $H$  refers to the homology at
$\C(\fM;b)$. In the next section we prove that the complex 
$\beta_*(\fop,\oM)(\fM;b) \to
\C(\fM;b)$ is quasi-isomorphic, but not isomorphic(!), to the blob
complex introduced in \cite{blob}.

Having in mind the comparison with other `blob complexes' in
the next section,  we describe the complex $\beta_*(\fop,\oM)(\fM;b)$
more explicitly. 
The corresponding towers  in~(\ref{Poradmnebolipravepredlokti}) are in this
particular situation the same as the towers of admissible inclusions
\begin{equation}
\label{Predverem jsem se znovu narodil.}
\Twr_\fM:\ 
\xymatrix@1@C=1.5em{
 (\fD^0;c^0)\ \ar@{^{(}->}[r]^(.52){\iota^0} \
&\ (\fD^1;c^1)\  \ar@{^{(}->}[r]^(.55){\iota^1} 
&  (\fD^2;c^2)\  \ar@{^{(}->}[r]^(.65){\iota^2} 
& \ \cdots \ {\ar@{^{(}->}[r]^(.33){\iota^{n-2}}}
& \ \ar@{^{(}->}[r]^(.62){\iota^{n-1}} (\fD^{n-1};c^{n-1}) 
\ % \ar@{^{(}->}[r]^{\iota^0}
&(\fD^n;c^n)\ \ar@{^{(}->}[r]^(.55){\iota}
& (\fM;b)
}
\end{equation}
of decorated blob configurations. 

For $0 \leq k \leq n\!-\!1$ and $D \in
\fD^{k+1}$ denote by $\fD^k_D$ the sub-configuration of
$\fD^k_D$  of blobs which are subsets of $D$, i.e.\ 
\[
\fD^k_D := \big \{
D' \in \fD^k \ | \ D' \subset D \big \}
\]
with the  order induced from  $\fD^k$. Denote also 
by $(\fD^k_D,c^k_D)$ the configuration $\fD^k_D$ with the
decoration on the boundary inherited from  $\fD^k$. The product in the right hand side 
of~(\ref{Jarusku boli  brisko.}) then~equals
\begin{equation}
\label{Za 6 dni na ocni.}
\bigotimes_{D \in \fD^0} \loc(D;c_D)
\ot 
\bigotimes_{k=0}^{n-1} \left\{ \rule{0em}{2em} \right.  \hskip -.5em
\bigotimes_\doubless{D \in \fD^{k+1}}{\fD^{k}_D \not= \emptyset} \hskip -.5em
\C\big(D \setminus \Int{\fD}^{k}_D; c_D \sqcup c^{k}_D\big) \ot  
\hskip -.5em
\bigotimes_\doubless{D \in \fD^{k+1}}{\fD^{k}_D = \emptyset}
 \hskip -.5em \loc( D; c_D)  \left. \rule{0em}{2em} \right\}
\ot
\C(\fM \setminus \Int{\fD}^n; b \sqcup c^n).
\end{equation}
Since all inclusions in~(\ref{Predverem jsem se znovu narodil.}) are
admissible, if $\fD^{k}_D = (D)$, then $c_D = c_D^{k}$ and thus in~(\ref{Za 6 dni na ocni.})
\[
\C\big(D \setminus \Int{\fD}^{k}_D; c_D \sqcup c^{k}_D\big) =
\Span({c}).
\]
An analogous formula for the piece $B_*(\fop,\oM)(\fM;b)$ of the normalized bar
construction in Definition~\ref{Za 70 minut.} 
can be obtained by restricting in~(\ref{Za 6 dni na ocni.}) the first
tensor product in the curly braces to $D \in \fD^{k+1}$ such that
$\fD^{k}_D \not= (D)$.

Tower~(\ref{Predverem jsem se znovu narodil.}) determines the
planar rooted 
tree  ${\rm T}_\fM$  with $n+2$ levels. Its root is
at level zero, level~$\ell$ has one vertex for each blob $D \in \fD^{n-\ell
  +1}$, $1 \leq \ell \leq n\!+\!1$. There is 
one oriented edge $D' \to D$ for each pair $(D',D)$ with $D \in
\fD^{k+1}$ and $D' \in \fD^{k}_D$. Since blob configurations are
linearly ordered sets by definition, the set of input edges of
each vertex is linearly ordered too, so  ${\rm T}_\fM$ is
planar, cf.~\cite[Example~4.10]{env}. We invite the reader to draw a
picture.

The product~(\ref{Za 6 dni na ocni.}) is thus the space
of all vertex-decorations of~${\rm T}_\fM$ such that the root is decorated by
the fields on $\fM$, the twigs (= the vertices with no input
edges) by the local relations, and the remaining vertices  
by the fields on the blob complements, all subject to the
matching of the fields on the boundaries. This description will play
an important r\^ole in the next section.

\begin{center}
{\bf Notation recall}
\end{center} \nopagebreak[4] 
\noindent\nopagebreak[4]% 
Operadic categories:\hfill\break 
\noindent 
$\blob$ \leaderfill 
the category opposite to the category of blob 
configurations and their inclusions\break 
$\ublob$ \leaderfill  
the category $\blob$ extended by the empty configuration\break
$\Blob$ \leaderfill 
the tautological operadic category $\Tau(\blob)$
associated to $\blob$\break 
$\uB$ \leaderfill the operadic category $\sfD(\ublob)$\break
$\dblob$, $\dublob$, $\dBlob$, $\uB(\C)$ \leaderfill the $\C$-decorated
versions of the above categories\break
\noindent 
Operadic modules:\hfill\break 
$\ttm$ \leaderfill the chaotic left $\blob$-module  $\Cha(\{\fM\}, \blob)$ with one object $\fM$\break
$\uttm$ \leaderfill the chaotic left $\ublob$-module $\Cha(\{\fM\},
\ublob)$ with one object $\fM$\break
$\bM$  \leaderfill  the tautological $\Blob$-module
$\Tau_{\blob}(\ttm)$ associated to $\ttm$\break
$\Mbar$ \leaderfill the $\uB$-module $\sfD_{\ublob}(\uttm)$\break
$\ttm(\C)$, $\uttm(\C)$,  $\ttM(\C)$,   
$\Mbar(\C)$
\leaderfill 
 the $\C$-decorated
versions of the above modules\break
\noindent 
Operads and their modules:\hfill\break 
$\fopn$    \leaderfill  the  $\uB(\C)$-operad of fields\break
$\fop$ \leaderfill 
the restriction of $\fopn$ to 
$\Blob(\C)$\break
$\oMn$ \leaderfill the $\fopn$-module of fields\break
$\oM$ \leaderfill the $\fop$-module defined as the restriction of
$\oMn$ to $\ttM(\C)$\break

\section{Blobs via colored operads, and comparison theorems}
\label{Jarka mi koupila cokoladu.}

The assumptions about the base manifold, blobs, fields, local
relations \&c.,~are the same as in Section~\ref{Boli mne prava noha.}.  
We start by showing that these data
determine a traditional $R$-linear unital 
colored operad $\fopc$, cf.~\cite[Section 2]{haha} for the definition
of colored operads. Operad modules were introduced
in~\cite[Definition~1.3]{markl:zebrulka}, 
cf.~also more recent~\cite[Subsections~2.1.5--6]{Fresse}.

%\subsection{The operad}
Colored operads require colors. In our case, colors will be 
pairs $(D,c)$ consisting of a blob~$D$ in $\fM$ with a field $c \in \C(\pa
D)$ on its boundary.
Suppose that  $\fD =
\{\Rada D1r\}$ is a configuration of blobs in $D$ as in
Definition~\ref{Te lucis ante terminum},  and $r \geq 2$. 
Then
\begin{equation}
\label{Lucis}
\colorop {\fopc}(\crada{(D_1,c_1)}{(D_r,c_r)} ;{(D,c)}) := \C\big(D
\setminus \bigcup_{i=1}^r \Int{D}_i;c \sqcup c_1 \sqcup
\cdots \sqcup c_r \big),
\end{equation} 
the span of  fields in $\C(D
\setminus \Int\fD)$ that restrict to the field $c \sqcup c_1 \sqcup
\cdots \sqcup c_r$ on $\pa D
\cup \pa \fD$, cf.~Figure~\ref{treti den valky}.
\begin{figure}
\[
\psscalebox{1.0 1.0} % Change this value to rescale the drawing.
{
\begin{pspicture}(0,-2.2666667)(5.0,2.2666667)
\definecolor{colour0}{rgb}{0.92156863,0.9137255,0.9137255}
\definecolor{colour1}{rgb}{0.74509805,0.6431373,0.6431373}
\psframe[linecolor=black, linewidth=0.04, fillstyle=solid,fillcolor=colour0, dimen=outer](4.8,2.266667)(0.0,-2.2666664)
\pscircle[linecolor=black, linewidth=0.04, fillstyle=solid,fillcolor=colour1, dimen=outer](2.5333333,-0.13333318){1.7333333}
\pscircle[linecolor=black, linewidth=0.04, fillstyle=solid,fillcolor=white, dimen=outer](1.8666667,-0.6666665){0.53333336}
\pscircle[linecolor=black, linewidth=0.04, fillstyle=solid,fillcolor=white, dimen=outer](3.3333333,-0.2666665){0.6666667}
\pscircle[linecolor=black, linewidth=0.04, fillstyle=solid,fillcolor=white, dimen=outer](2.0,0.80000013){0.4}
\rput[bl](4.0,1.6000001){$\fM$}
\rput[bl](1.77,0.6266668){$D_1$}
\rput[bl](1.666667,-0.85){$D_2$}
\rput[bl](3.2,-0.46){$D_3$}
\rput[bl](2.8,0.8){$D$}
\rput[bl](3.4666667,1.35){$c$}
\rput[bl](2.1333334,1.15){$c_1$}
\rput[bl](1.6,-0.12333318){$c_2$}
\rput[bl](3.4666667,0.40000015){$c_3$}
\end{pspicture}
}
\]
\caption{A piece of the colored operad $\fopc$ -- a schematic
  picture of a `punctured blob.'\label{treti den valky}}  
\end{figure}
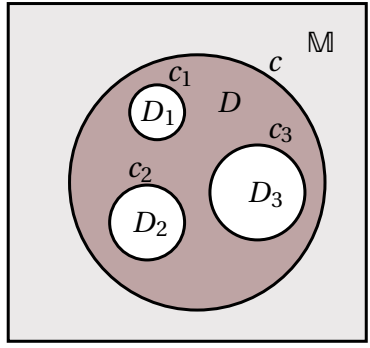
To define the component
\begin{equation}
\label{Ex more docti mystico}
\colorop {\fopc}((D',c');{(D,c)}),
\end{equation}
we distinguish two cases. If $D' \subset \Int D$, then~(\ref{Ex more
  docti mystico}) equals $\C( D \setminus \Int {D'}; c \sqcup
c')$ as expected. If $D = D'$, we moreover require that $c=c'$, and then 
\[
\colorop {\fopc}((D,c);{(D,c)}) := \Span(\{c\}),
\]
the $R$-linear span of the one-point set $\{c\}$. If $r=0$, we define
\[
\colorop {\fopc}(\emptyset;{(D,c)}) := \loc(D;c),
\]
the space of local relations. They  thus appear as operations with no input
and one output, that is, the `constants.'

\noindent 
{\bf Warning.} The symbol 
`$\emptyset$' in the above display is not
an input color, but  indicates that the set of
inputs is empty.

\begin{proposition}
\label{Prohlidky se blizi.}
The structure $\fopc$ defined above is a unital colored $R$-linear operad.
\end{proposition}

\begin{proof}
Let $x \in \colorop{\fopc}(\crada{(D_1;c_1)}{(D_r;c_r)} ;{(D;c)})$ and $x_i\in
\colorop{\fopc}(\crada{(D^i_1;c^i_1)}{(D^i_{k_i};c^i_{k_i})};{(D_i;c_i)})$,
$1\leq i \leq r$. Then the operadic composite
\begin{equation}
\label{posledni sobota ve Stokholmu}
x(\Rada x1n) \in
\colorop{\fopc}(\crada{(D^1_1;c^1_1)}{(D^1_{k_1};c^1_{k_1})}\cdots
\crada{(D^r_1;c^r_1)}{(D^r_{k_r};c^r_{k_r})};{(D,c)})
\end{equation}
is the field obtained by glueing the fields $x,\Rada x1r$ along the
boundaries of the balls $\Rada D1r$. The color-matching guarantees that
this glueing is possible. The image of the glueing
\[
\colorop{\fopc}(\crada{(D_1;c_1)}{(D_r;c_r)} ;{(D;c)}) \ot
\colorop {\fopc}({\emptyset};{(D_1,c_1)}) \ot \cdots \ot \colorop
{\fopc}({\emptyset};{(D_r,c_r)}) 
 \longrightarrow \colorop {\fopc}({\emptyset};{(D,c)})
\]
consists of local relations as expected due to the ideal property 
of Definition~\ref{zacina i leve}.
The operad axioms are immediately clear,
including the
unit property of $c\in\colorop {\fopc}((D,c);{(D,c)})$. 
\end{proof}

%\subsection{The bar construction}

Fields on the base manifold $\fM$ that restrict to a given $b
\in \C(\pa \fM)$ form a  {\em left
  $\fopc$-module \/} $\oMc$ with the components
\[
\oMc\big(\crada{(D_1,c_1)}{(D_r,c_r)}\big) := \C\big(\fM
\setminus \bigcup_{i=1}^r \Int{D}_i; c_1 \sqcup
\cdots \sqcup c_r  \sqcup b \big),
\]
with the left $\fopc$-action assigning to each $m \in
\oMc\big(\crada{(D_1,c_1)}{(D_r,c_r)}\big)$ and to $x_i$'s
as in the proof of Proposition~\ref{Prohlidky se blizi.} the element
\[
m(\Rada x1n) \in
\oMc\big(\crada{(D^1_1;c^1_1)}{(D^1_{k_1};c^1_{k_1})}\cdots
\crada{(D^r_1;c^r_1)}{(D^r_{k_r};c^r_{k_r})}\big),
\]
given by the glueing of fields as before.
In the rest of this section, by `colors' we mean the colors used in the definition
of the operad $\fopc$ and its module $\oMc$.

The operad $\fopc$ and its module $\oMc$ can be used to write
formula~\eqref{Za 6 dni na ocni.} for $\beta_*(\fop,\oM)(\fM;b)$
and for its  normalized modification $B_*(\fop,\oM)(\fM;b)$ in a
nice compact form. Given $D \in \fD^{k+1}$, denote 
\[
(\fD^k_D,c_D^k) =
\big((D_D^1,c^1_D),\ldots,(D_D^{i_D},c^{i_D}_D)\big)\ \hbox{ and }\ 
\fD^n =  \big((D_n^1,c^1_n),\ldots,(D_n^{i_n},c^{i_n}_n)\big).
\]
The right hand side of~\eqref{Za 6 dni na ocni.} then becomes
\begin{equation}
\label{Po navsteve ocniho.}
\bigotimes_{D \in \fD^0} 
\colorop{\fopc}(\emptyset;{(D;c_D)})
\ot 
\bigotimes_{k=0}^{n+1} \
\bigotimes_{D \in \fD^{k+1}} \hskip -.5em
\colorop{\fopc}({(D_D^1,c^1_D),\ldots,(D_D^{i_D},c^{i_D}_D)};{(D;c_D)})
\ot
\oMc\big((D_n^1,c^1_n),\ldots,(D_n^{i_n},c^{i_n}_n)\big).
\end{equation}
Notice that the two tensor products in the curly brackets
of~\eqref{Za 6 dni na ocni.} have been absorbed by one tensor product, 
thanks to the convenient
definition of the operad $\fopc$. The normalized variant is obtained
by assuming that in the  `big' tensor product either  $i_D \geq 2$,
or $i_D=1$ but $D_D^1 \not= D$.

As in the paragraph following formula~\eqref{Za 6 dni na ocni.}
we interpret the product~(\ref{Po navsteve ocniho.}) as the space
of all vertex-decorations of a planar rooted 
tree with $n+2$ levels and edges colored by fields, 
such that the root is decorated by $\oMc$ and the other vertices with
$\fopc$ in such a way that the output and the inputs of the
decorations match the colors of the adjacent edges. 

Comparing the above with  the material
in~\cite[Subsection~4.3.2]{Fresse} we identify~(\ref{Po navsteve
  ocniho.}) with the constant part of  
the colored version of Fresse's {\em simplicial bar
  construction \/}  $C_*(\fopc,\fopc,\oMc)$, resp.~with  
its normalization~$N_*(\fopc,\fopc,\oMc)$. 
We therefore have

\begin{proposition}
\label{Na kontrolu a rok.}
There are natural isomorphisms of chain complexes
\begin{equation}
\label{Malem jsem se zabil.}
\beta_*(\fop,\oM)(\fM;b) \cong C_*(\fopc,\fopc,\oMc)(\emptyset)\
\hbox { and } \
B_*(\fop,\oM)(\fM;b) \cong N_*(\fopc,\fopc,\oMc)(\emptyset)
\end{equation}
compatible with the augmentations.
\end{proposition}

Proposition~\ref{Na kontrolu a rok.} provides a bridge between blob
complexes viewed from the perspective of unary operadic 
categories and blob complexes based
on colored operads. Fact~4.1.7 of~\cite{Fresse} applied to $P = \fopc$ and
$R= \oMc$ may suggest that the complexes  $C_*(\fopc,\fopc,\oMc)$ and
$N_*(\fopc,\fopc,\oMc)$ are acyclic in positive dimensions. This is
however not true, because the crucial assumption of connectivity
required by Fact~4.1.7 is  violated in our situation.

B.~Fresse introduced in~\cite[Section~4]{Fresse}
the {\em differential bar construction} $B_*(L,P,R)$ of an augmented
$P$-operad with coefficients in a right $P$-module $L$ and a left
$P$-module $R$. We will use the obvious colored version of his
construction with $P = \fopc$, $L = \oMc$ and $R = \fopc$. 
Let $\lunit$ be the colored operad whose only nontrivial component is  
\[
\colorop{\lunit}({(D;c)};{(D,c)}) := \Span(\{c\}),
\] 
the $R$-linear
span of the field $\{c\}$. 
The operad of fields $\fopc$ is augmented by the obvious morphism $\varepsilon :
\fopc \to  \lunit$ of colored operads. We denote by $\ofop := {\rm
  Ker}(\varepsilon)$ the augmentation ideal. 

Let $\oB(\fopc)$ be the cofree conilpotent cooperad 
generated by the component-wise suspension of the colored collection
$\ofop$. Mimicking Fresse's definition we consider
\begin{equation}
\label{Kdyz pisu zase mne pali hrbety rukou.}
B_*(\oMc,\fopc,\fopc) := \oMc \circ \oB(\fopc) \circ \fopc, 
\end{equation}
where $\circ$ is the straightforward colored version of the composition
product \cite[\S 1.3.5]{Fresse}. The iterated product~(\ref{Kdyz pisu zase mne
  pali hrbety rukou.}) bears the
differential given by the operad structure of $\fopc$ and the
right $\fopc$-action on $\oMc$.
The differential bar construction $B_*(\oMc,\fopc,\fopc)$ is thus 
a colored collection with components 
\[
B_*(\oMc,\fopc,\fopc)\big(\crada{(D_1,c_1)}{(D_r,c_r)}\big).
\]
We will be particularly interested in the component with $r=0$
(i.e.~`no inputs'), which we denote  by
$B_*(\oMc,\fopc,\fopc)(\emptyset)$.

The elements of $B_*(\oMc,\fopc,\fopc)(\emptyset)$ can be visualized as 
finite linear
combinations of forests growing from $\oMc$,
whose trees have forks (= vertices) decorated by the fields in $\ofop$, branches (=
edges) colored by blobs with fields on the boundaries, 
and twigs (= leaves) by the fields in the  local relations, 
cf. Figure~\ref{tresnicky}. We must however be careful, since
the fields (= decorations of the vertices) are assigned degree $+1$,
cf.~(\ref{Kdyz pisu zase mne pali hrbety rukou.}), so we are in fact
dealing with the {\em equivalence classes} of forests in
Figure~\ref{tresnicky} with vertices linearly ordered compatibly 
with the partial order given by the distance from the soil
(= root). We identify a forest $F'$ with the forest $\epsilon \cdot
F''$, where $\epsilon \in \{+1,-1\}$ is the signum of the permutation
that brings the order of vertices of $F'$ to the order of vertices of
$F''$.   The differential contracts the edges, one at a time, and
decorates the new vertex thus created by the glued field.
Notice that this description is practically identical with the
definition of the blob complex in~\cite{blob}.

\begin{figure}
\[
\psscalebox{.8 .8} % Change this value to rescale the drawing.
{
\begin{pspicture}(0,-2.723504)(10.094016,2.723504)
\definecolor{colour0}{rgb}{0.74509805,0.6431373,0.6431373}
\definecolor{colour1}{rgb}{0.9372549,0.14117648,0.14117648}
\psframe[linecolor=black, linewidth=0.04, fillstyle=solid,fillcolor=colour0, dimen=outer](9.858971,-1.9215103)(1.0370396,-2.723504)
\psline[linecolor=black, linewidth=0.04](5.448005,-1.9215103)(5.448005,-0.3175229)(5.448005,-0.3175229)
\psline[linecolor=black, linewidth=0.04](7.8539863,-1.9215103)(7.8539863,0.083473966)(9.858971,1.6874614)(9.858971,1.6874614)
\psline[linecolor=black, linewidth=0.04](7.8539863,0.083473966)(7.051993,2.0884583)(7.051993,2.0884583)
\psline[linecolor=black, linewidth=0.04](7.8539863,0.083473966)(8.65598,2.4894552)(8.65598,2.4894552)(8.65598,2.4894552)
\psline[linecolor=black, linewidth=0.04](9.858971,1.6874614)(9.858971,2.4894552)
\psdots[linecolor=black, dotsize=0.40099686](7.847008,0.07649589)
\psdots[linecolor=black, dotsize=0.40099686](9.847009,1.6764959)
\psline[linecolor=black, linewidth=0.04](1.4380366,1.2864646)(2.6410272,2.0884583)
\psline[linecolor=black, linewidth=0.04](3.0420241,0.083473966)(5.0470085,1.2864646)(4.2450147,2.4894552)(4.2450147,2.4894552)
\psline[linecolor=black, linewidth=0.04](5.0470085,1.2864646)(5.849002,2.0884583)
\psline[linecolor=black, linewidth=0.04](3.0420241,-1.9215103)(3.0420241,0.083473966)(1.4380366,1.2864646)(0.23504595,1.6874614)
\psdots[linecolor=black, dotsize=0.40099686](3.0470083,0.07649589)
\psdots[linecolor=black, dotsize=0.40099686](5.0470085,1.2764959)
\psdots[linecolor=black, dotsize=0.40099686](1.4470084,1.2764959)
\psdots[linecolor=colour1, dotsize=0.5012461](5.447008,-0.3235041)
\psdots[linecolor=colour1, dotsize=0.5012461](4.2470083,2.476496)
\psdots[linecolor=colour1, dotsize=0.5012461](5.847008,2.076496)
\psdots[linecolor=colour1, dotsize=0.5012461](0.24700837,1.6764959)
\psdots[linecolor=colour1, dotsize=0.5012461](2.6470084,2.076496)
\psdots[linecolor=colour1, dotsize=0.5012461](7.0470085,2.076496)
\psdots[linecolor=colour1, dotsize=0.5012461](8.647008,2.476496)
\psdots[linecolor=colour1, dotsize=0.5012461](9.847009,2.476496)
\psdots[linecolor=black, dotstyle=o, dotsize=0.2, fillcolor=white](0.24700837,1.6764959)
\psdots[linecolor=black, dotstyle=o, dotsize=0.2, fillcolor=white](2.6470084,2.076496)
\psdots[linecolor=black, dotstyle=o, dotsize=0.2, fillcolor=white](4.2470083,2.476496)
\psdots[linecolor=black, dotstyle=o, dotsize=0.2, fillcolor=white](5.847008,2.076496)
\psdots[linecolor=black, dotstyle=o, dotsize=0.2, fillcolor=white](7.0470085,2.076496)
\psdots[linecolor=black, dotstyle=o, dotsize=0.2, fillcolor=white](8.647008,2.476496)
\psdots[linecolor=black, dotstyle=o, dotsize=0.2, fillcolor=white](9.847009,2.476496)
\psdots[linecolor=black, dotstyle=o, dotsize=0.2, fillcolor=white](5.447008,-0.3235041)
\end{pspicture}
}
\]
\caption{\label{tresnicky}Viewing elements of $B_*(\oMc,\fopc,\lunit)(\emptyset)$
  as forests. Internal vertices are decorated by fields on punctured
  blobs, 
leaves by
generators of  local relations, the soil by a field on $\fM$.}
\end{figure}
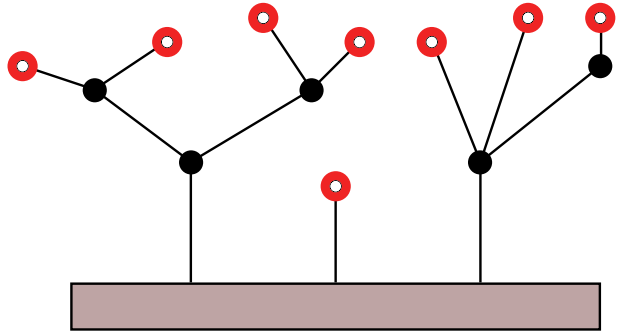

In the `forest reprezentation' of  $B_*(\oMc,\fopc,\fopc)(\emptyset)$, the
homological 
degree is the number of vertices. In particular, 
\[
B_0(\oMc,\fopc,\fopc)(\emptyset) \cong \bigoplus_{(D;c)} \oMc\big((D;c)\big) \ot
\colorop {\fopc}({\emptyset};{(D;c)}) =  \bigoplus_{(D;c)} \C(\fM \setminus \Int
D; c \sqcup b)\ot\loc(D;c)
\]
with the augmentation $\epsilon :B_0(\oMc,\fopc,\fopc)(\emptyset) \to
\C(\fM;b)$ given by gluing the fields in $\C(\fM \setminus \Int
D; c\sqcup b)$ with the fields from  $\loc(D,c)$. 

\begin{example}
Figure~\ref{Terezka a Mourek} symbolizes the types of terms in the initial part 
\[
\C(\fM;b) \stackrel\epsilon\longleftarrow B_0(\oMc,\fopc,\fopc)(\emptyset)\stackrel\pa\longleftarrow
B_1(\oMc,\fopc,\fopc)(\emptyset) \stackrel\pa\longleftarrow \cdots
\]
of the augmented bar construction. This should be compared to the
explicit description of the initial terms of the blob complex given on
pages 1500--1502 of \cite{blob}. 
\begin{figure}
\[
\psscalebox{.8 .8} % Change this value to rescale the drawing.
{
\begin{pspicture}(70,-.5)(2,2.3006945)
\definecolor{colour1}{rgb}{0.74509805,0.6431373,0.6431373}
\definecolor{colour2}{rgb}{0.9372549,0.14117648,0.14117648}
\definecolor{colour0}{rgb}{0.8509804,0.75686276,0.75686276}
\psframe[linecolor=black, linewidth=0.04, fillstyle=solid,fillcolor=colour1, dimen=outer](7.554669,0.61479825)(5.3006945,-0.10175844)
\psdots[linecolor=black, dotsize=0.4](18.900694,1.4000001)
\psframe[linecolor=black, linewidth=0.04, fillstyle=solid,fillcolor=colour1, dimen=outer](12.354669,0.61479825)(10.100695,-0.10175844)
\psframe[linecolor=black, linewidth=0.04, fillstyle=solid,fillcolor=colour1, dimen=outer](17.154669,0.61479825)(14.900695,-0.10175844)
\psframe[linecolor=black, linewidth=0.04, fillstyle=solid,fillcolor=colour1, dimen=outer](20.35467,0.61479825)(18.100695,-0.10175844)
\psline[linecolor=black, linewidth=0.04](11.300694,0.6000001)(11.300694,1.4000001)(11.300694,1.4000001)
\psline[linecolor=black, linewidth=0.04](15.700695,0.6000001)(14.900695,1.8000001)
\psline[linecolor=black, linewidth=0.04](16.500694,0.6000001)(16.900694,1.4000001)
\psline[linecolor=black, linewidth=0.04](19.300695,0.6000001)(18.900694,1.4000001)(19.700695,2.2)
\psdots[linecolor=colour2, dotsize=0.5](14.900695,1.8000001)
\psdots[linecolor=colour2, dotsize=0.5](16.900694,1.4000001)
\psdots[linecolor=colour2, dotsize=0.5](19.700695,2.2)
\psdots[linecolor=colour2, dotsize=0.5](11.300694,1.4000001)
\psdots[linecolor=colour0, dotstyle=o, dotsize=0.2, fillcolor=white](11.300694,1.4000001)
\psdots[linecolor=colour0, dotstyle=o, dotsize=0.2, fillcolor=white](0.10069458,-6.2)
\psdots[linecolor=colour0, dotstyle=o, dotsize=0.2, fillcolor=white](14.900695,1.8000001)
\psdots[linecolor=colour0, dotstyle=o, dotsize=0.2, fillcolor=white](16.900694,1.4000001)
\psdots[linecolor=colour0, dotstyle=o, dotsize=0.2, fillcolor=white](19.700695,2.2)
\rput[cc](8.8,0.20000008){\large $\longleftarrow$}
\rput[cc](13.700695,0.20000008){\large $\longleftarrow$}
\rput[cc](17.700695,0.20000008){\large $\oplus$}
\rput[tc](17.695,-0.5999999){
$\underbrace{\rule{14em}{0mm}}_{
\hbox{$B_1(\oMc,\fopc,\lunit)(\emptyset)$}}
$}
\rput[bc](6.5,-0.5999999){$\fopc(\fM)$}
\rput[bc](11.300694,-0.5999999){$B_0(\oMc,\fopc,\lunit)(\emptyset)$}
\end{pspicture}
}
\]
\caption{\label{Terezka a Mourek}%
The initial part of the augmented bar construction.}
\end{figure}
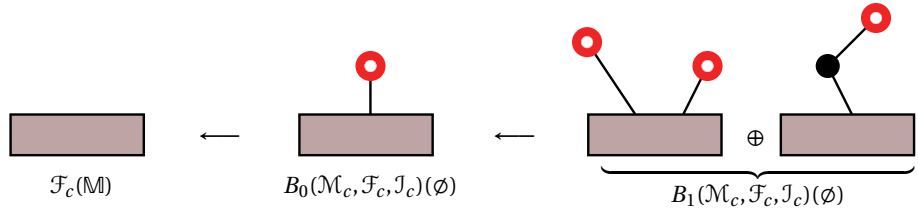
\end{example}

We finally arrive at

\begin{proposition}
\label{Je porad zima.}
The blob complex \/ $\calB(\fM,\C)$ of \/ \cite[Section~3]{blob} is isomorphic
to the piece 
\begin{equation}
\label{Jackill and Hyde}
\C(\fM;b)
\stackrel\epsilon\leftarrow B_*(\oMc,\fopc,\fopc)(\emptyset)
\end{equation} 
of the augmented differential bar construction.
\end{proposition}

\begin{proof}
Comparing the respective definitions, we easily construct the
required isomorphism
\[
\xymatrix{\C(\fM;b) \ar@{=}[d]&\ar[l]_\epsilon B_0(\oMc,\fopc,\fopc)(\emptyset)\ar@{<->}[d]^\cong
& \ar[l]_\pa\ar@{<->}[d]^\cong
B_1(\oMc,\fopc,\fopc)(\emptyset)&  \ar[l]_\pa B_2(\oMc,\fopc,\fopc)(\emptyset)\ar@{<->}[d]^\cong &
\ar[l]_(.3)\pa \cdots
\\
\calB_0(\fM,\C) &\ar[l]_\pa \calB_1(\fM,\C)
& \ar[l]_\pa
\calB_2(\fM,\C)&  \ar[l]_\pa \calB_3(\fM,\C)&
\ar[l]_(.3)\pa \cdots
}
\]
of chain complexes. Notice the degree shift.
\end{proof}

Recall that Fresse introduced, for a `traditional' operad $P$, a right
$P$-module $L$ and a left module $R$, the {\em levelization morphism\/}  
\[
\phi_*(L,P,R) : B_*(L,P,R) \to N_*(L,P,R) 
\]
from the differential bar construction to the normalization of the simplicial
bar construction. 
While the elements of $B_*(L,P,R)$ are represented by decorated trees,  the
elements of $N_*(L,P,R)$ are decorated trees equipped with levels. The
chain map
$\phi_*(L,P,R)$ sends a given decorated tree to the sum, with
appropriate signs, of all decorated
trees with levels whose underlying non-leveled decorated tree equals
the given one.
Fresse then proved in~\cite[Theorem~4.1.8]{Fresse} that $\phi_*(L,P,R)$ is a
\qi. Although he assumed simple connectivity, his theorem holds
without this assumption, which expresses the folklore fact  that 
the space of levels of a given tree is a contractible
groupoid, cf.~also~\cite{HeMo}. Combining this with
isomorphisms~(\ref{Malem jsem se zabil.}) and~(\ref{Jackill and Hyde})
results in
the following comparison between the original blob complex  $\calB_*(\fM,\C)$ 
defined in~\cite[Section~3]{Walker} and the normalized bar resolution
$B_*(\fop,\oM)(\fM;b)$, cf.~Definition~\ref{Za 70 minut.}.

\begin{theoremC}
\label{Veta C}
The levelization morphism of\/~\cite[Theorem~4.1.8]{Fresse} induces a \qi\
\[
\xymatrix@C=4em{\calB_{*+1}(\fM,\C)
  \ar[d]_{\ell_*}^\sim \ar[r] & \ar@{=}[d]
  \C(\fM;b)
\\
B_*(\fop,\oM)(\fM;b) \ar[r] & \C(\fM;b)
}
\]
of augmented chain  complexes. 
\end{theoremC}

The colored operad $\fopc$ is a right module over itself, so one may
also consider $B_*(\fopc,\fopc,\fopc)$ instead of
$B_*(\oMc,\fopc,\fopc)$. Since the pieces of $\fopc$ possess 
also the output color, the components of
$B_*(\fopc,\fopc,\fopc)$ are
\begin{equation}
\label{Dnes jsem si nechal udelat brejle.}
\colorop{B_*(\fopc,\fopc,\fopc)}(\crada{(D_1;c_1)}{(D_r;c_r)} ;{(D;c)}).
\end{equation}
Although the connectivity assumption of \cite[Lemma~4.1.3]{Fresse} is
not fulfilled, a simple explicit contracting homotopy which was in fact
constructed in the proof of~\cite[Proposition~3.2.1]{Walker} shows that~(\ref{Dnes jsem si nechal udelat
  brejle.}) is acyclic  in positive dimensions for each choice of colors.
If the base manifold $\fM$  is
isomorphic to a ball $D$, we easily verify that
\[
B(\oMc,\fopc,\lunit) \cong \bigoplus_{c \in \C(\pa D)} 
\colorop{B(\fopc,\fopc,\lunit)}({\emptyset} ;{(\fM;b)}).
\]
The acyclicity of~(\ref{Dnes jsem si nechal udelat brejle.})
combined with Proposition \ref{Je porad zima.} gives

\begin{corollary}[Corollary~3.2.2 of \cite{blob}]
If \/ $\fM$ is isomorphic to a $(d\!+\!1)$-dimensional ball, then 
the chain complex
$\calB(\fM,\C)$ is contractible.
\end{corollary}

\section*{Summary}

Chain complexes featured in Part~2 together with
the connecting maps are
summarized in Figure~\ref{Pozitri prof. Vondra.}.
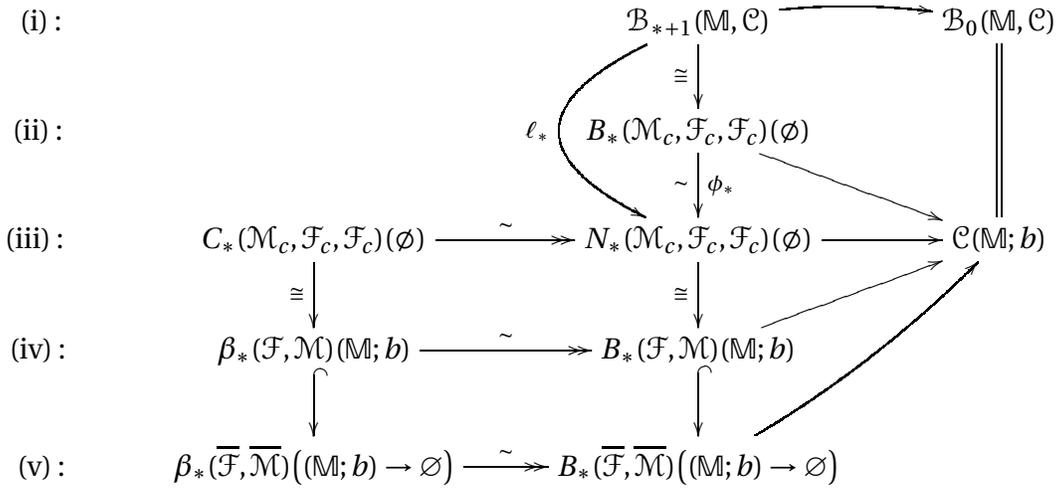
\begin{figure}
\[
\xymatrix@C=3em{
{\rm \ (i):} &&  \calB_{*+1}(\fM,\C) \ar@/_5em/[dd]_{\ell_*}
  \ar[d]_\cong  \ar@/^/[r]     &  \calB_0(\fM,\C) \ar@{=}[dd]
\\
{\rm (ii):}&&  B_*(\oMc,\fopc,\fopc)(\emptyset)
  \ar[d]^{\phi_*}_\sim    \ar[rd]  & 
\\
{\rm (iii): \ }& 
C_*(\oMc,\fopc,\fopc)(\emptyset) \ar@{->>}[r]^\sim  \ar[d]_\cong  
&
N_*(\oMc,\fopc,\fopc)(\emptyset) \ar[r]  \ar[d]_\cong    & \C(\fM;b)
\\
{\rm (iv):}&\beta_*(\fop,\oM)(\fM;b) \ar@{->>}[r]^\sim  \ar@{^{(}->}[d]  
&
B_*(\fop,\oM)(\fM;b)  \ar@{^{(}->}[d]   \ar[ru]   &
\\
{\rm \ (v):}&\beta_*(\fopn,\oMn)\big((\fM;b) \to \varnothing\big)
\ar@{->>}[r]^\sim
&
B_*(\fopn,\oMn)\big((\fM;b) \to \varnothing\big)    \ar@/_/[ruu]  &
}
\]
\caption{\label{Pozitri prof. Vondra.}Sundry chain complexes and their
maps.}
\end{figure}
In that figure: \hfill\break 
\hglue 1em 
-- $\calB_{*+1}(\fM,\C)$ in row~(i) is the original blob complex
of~\cite{Walker},
\hfill\break 
\hglue 1em -- $B_*(\oMc,\fopc,\fopc)(\emptyset)$ in row~(ii) is 
the piece of Fresse's bar construction,
\hfill\break 
\hglue 1em -- $C_*(\oMc,\fopc,\fopc)(\emptyset)$ in row~(iii) is the
piece of Fresse's simplicial bar construction,\hfill\break 
\hglue 1em -- $N_*(\oMc,\fopc,\fopc)(\emptyset)$ in row~(iii)  
is the normalization of $C_*(\oMc,\fopc,\fopc)(\emptyset)$ ,
\hfill\break 
\hglue 1em -- $\beta_*(\fop,\oM)(\fM;b)$ in row~(iv) is the
piece of the un-normalized bar resolution in Definition~\ref{Prestane mi to
  paleni?},
\hfill\break 
\hglue 1em -- 
$B_*(\fop,\oM)(\fM;b)$  in row~(iv) is the normalization of
$\beta_*(\fop,\oM)(\fM;b)$, and
\hfill\break 
\hglue 1em -- the items in row~(v) are as in row~(iv) but this time
applied on $\fopn$ and $\oMn$. \hfill\break
The vertical map $\phi_*$ is Fresse's levelization morphism, $\ell_*$
is the
map in Theorem~C. The remaining maps are either natural isomorphisms,
or augmentations, 
or inclusions, or projections. The vertical isomorphism between row
(i) and (ii) comes from Proposition~\ref{Je porad zima.}, 
the two vertical isomorphism between rows (iii) and (iv)
are that of Proposition~\ref{Na kontrolu a rok.}.

%\listoftodos


\frenchspacing
\begin{thebibliography}{10}

\bibitem{duodel}
M.A. Batanin and M.~Markl.
\newblock {Operadic categories and duoidal Deligne's conjecture}.
\newblock {\em Adv. Math.}, 285:1630--1687, 2015.

\bibitem{BataninKockWeber}
M. Batanin, J. Kock and M. Weber.
\newblock Work in progress. Meanwhile see `Operadic categories and
sets,' in {\em Oberwolfach Reports},~46/2021.

\bibitem{env}
M.A. Batanin and M.~Markl.
\newblock {Operadic categories as a natural environment for Koszul duality}.
\newblock Preprint {\tt arXiv:1812.02935}, version~4, July 2022.


\bibitem{kodu}
M.A. Batanin and M.~Markl.
\newblock {Koszul duality for operadic categories}.
\newblock Preprint {\tt arXiv:arXiv:2105.05198}, version~2, July 2022.


\bibitem{EM}
S. Eilenberg and J.C. Moore.
\newblock Adjoint functors and triples.
{\em Illinois J. Math.}, 9(3):381--398, 1965.

\bibitem{Fresse}
B. Fresse.
\newblock {Koszul duality of operads and homology of partition
posets.} {\em  Contemp. Math.}, 346:  115--215, 2004.

\bibitem{GarnerKockWeber}
R. Garner, J. Kock and M. Weber.
\newblock Operadic categories and d\'ecalage.
\newblock {\em Adv. Math.}, 377:107440, 2021.

\bibitem{HeMo}
G. Heuts, I. Moerdijk. 
\newblock Partition complexes and trees.
\newblock Preprint {\tt arXiv:2112.08043}, December 2021.


\bibitem{Homology}
S. MacLane. {\em Homology.\/}
\newblock Springer Verlag, 1963


\bibitem{haha}
M. Markl.
\newblock 
Homotopy algebras are homotopy algebras. 
\newblock {\em Forum Matematicum}  {16(1)}: 129--160, 2004.

\bibitem{markl:zebrulka}
M.~Markl.
\newblock Models for operads.
\newblock {\em Communications in Algebra}, 24(4):1471--1500, 1996.

\bibitem{markl:handbook}
M.~Markl.
\newblock {Operads and {PROP}s}.
\newblock In {\em {Handbook of algebra. {V}ol. 5}}, pages 87--140.
  Elsevier/North-Holland, Amsterdam, 2008.
    

\bibitem{markl-shnider-stasheff:book}
M.~Markl, S.~Shnider, and J.D. Stasheff.
\newblock {\em {Operads in algebra, topology and physics}}, volume~96 of {\em
  {Mathematical Surveys and Monographs}}.
\newblock American Mathematical Society, Providence, RI, 2002.

\bibitem{blob}
S. Morrison and K. Walker.
\newblock Blob homology. 
\newblock {\em Geometry \& Topology}, 16:1481--1607, 2012. 

\bibitem{Walker}
K. Walker. {\em TQFTs.\/} Early incomplete draft, version 1h, May
2006. Available from the author's web homepage. 


\end{thebibliography}
\end{document}